\documentclass[12pt,reqno]{amsart}   
\pdfoutput=1 

\usepackage{graphicx,subfigure,color}
\usepackage{latexsym,amsmath,amsfonts,amscd,amsthm,epstopdf,lineno}
\usepackage[T1]{fontenc}
\usepackage{caption,here}
\usepackage{textcomp}
\usepackage{tikz}
\usetikzlibrary{decorations.pathmorphing,decorations.pathreplacing}
\usepackage{siunitx}
\usepackage{algorithm}  
\usepackage{enumitem}   
\usepackage{fullpage}
\usepackage[colorlinks=true,citecolor=blue,urlcolor=blue]{hyperref}
\newtheorem{theorem}{Theorem}[section]
\newtheorem{remark}{Remark}
\newtheorem{lemma}[theorem]{Lemma}

\newcommand{\dx}{\Delta x}

\begin{document}

\title{Method of lines transpose:  Energy gradient flows using direct operator inversion for phase field models}
\author[M. Causley]{Matthew Causley}
\address{Mathematics Department, Kettering University, Flint, MI 48504}
\email{mcausley@kettering.edu}

\author[H. Cho]{Hana Cho}
\address{Department of Mathematics, Michigan State University, East Lansing, MI 48824}
\email{chohana@math.msu.edu}

\author[A. Christlieb]{Andrew Christlieb}
\address{Department of Mathematics, Michigan State University, East Lansing, MI 48824}
\email{andrewch@math.msu.edu}


\date{}
\begin{abstract}
In this work, we develop an $\mathcal{O}(N)$ implicit real space method in 1D and 2D for the Cahn Hilliard (CH) and vector Cahn Hilliard (VCH) equations, based on the Method Of Lines Transpose (MOL$^\text{T}$) formulation. This formulation results in a semi-discrete time stepping algorithm, which we prove is gradient stable in the $H^{-1}$ norm.

The spatial discretization follows from dimensional splitting, and an $\mathcal{O}(N)$ matrix-free solver, which applies fast convolution to the modified Helmholtz equation. We propose a novel factorization technique, in which fourth order spatial derivatives are incorporated into the solver. The splitting error is included in the nonlinear fixed point iteration, resulting in a high order, logically Cartesian (line-by-line) update. Our method is fast, but not restricted to periodic boundaries like the fast Fourier transform (FFT).

The basic solver is implemented using the Backward Euler formulation, and we extend this to both backward difference (BDF) stencils, implicit Runge Kutta (SDIRK) and spectral deferred correction (SDC) frameworks to achieve high orders of temporal accuracy. We demonstrate with numerical results that the CH, and VCH equations maintain gradient stability in one and two spatial dimensions. We also explore time-adaptivity, so that meta-stable states and ripening events can be simulated both quickly and efficiently.


\end{abstract}

\maketitle

\section{Introduction}

Many materials such as metals, ceramics, and polymers have physical properties which depend strongly on their microstructure. The morphology evolves during various physical processes, such as solidification, solid-state phase transformation, grain coarsening and grain growth. These processes are governed complex and nonlinear chemical reactions, and so the location and motion of phase interfaces typically evolve over long time scales. Consequently the phase-field method has become a consolidated tool for simulating microstructure evolution. See \cite{chen2002phase} for a survey on the subject.

One of the most well-known phase field models is the Cahn-Hilliard (CH) equation. Cahn and Hilliard \cite{cahn1958free} introduced it to explain the phase separation by predicting the interfacial free energy between two coexisting phases. The free energy is modeled using an order parameter $u$, and a reaction term $f(u)$. The reaction function is defined as the derivative of a double-well potential with two minima, $u = \pm 1$, which represent each phase of the composite material. The phases are separated by {\em transition layers}, which are of width $\mathcal{O}(\epsilon)$. The CH equation has been used extensively to model phenomena in material science; and several generalizations have been developed, including the vector Cahn-Hilliard \eqref{eqn:CH_vector} and functionalized Cahn-Hilliard equations.

In this paper, we will develop numerical schemes for phase field models based on the Method of Lines Transpose (MOL$^T$) framework. The MOL$^{\text{T}}$ formulation, which is a practical approach to Rothe's method \cite{rothe1930zweidimensionale}, starts by discretizing a PDE in time, and then developing efficient approaches to solve the resulting set of coupled boundary value problems. The traditional formulation, proposed by Rothe \cite{rothe1930zweidimensionale}, resulted in large matrices that were cumbersome; hence the method did not get a great deal of attention. In the late 2000s Rothe's method received more attention, due to the development of spectral deferred correction methods and Krylov subspace techniques \cite{jia2008krylov}. Soon, other fast summation techniques such as the fast multipole method \cite{kropinski2011fast} were employed based on potential theory and boundary integral techniques. 

Even more recently, MOL$^{\text{T}}$ algorithms for PDEs have been developed in the context of wave propagation \cite{causley2014method,causley2014higher,causley2015method}, in which the resulting Helmholtz boundary value problem is solved with a dimensionally split algorithm, similar to alternating direction implicit (ADI) methods \cite{douglas1955numerical,fairweather1967new}. Here the spatial operator is inverted one dimension at a time through fast convolution, in a logically Cartesian fashion. In \cite{causley2014higher}, successive convolution was introduced, leading to an A-stable method of arbitrary order in time for hyperbolic problems. This led the current authors to revisit this idea in context of linear and nonlinear (second order) parabolic problems \cite{causley2016method}, and successive convolution was inserted into the resolvent expansion of pseudo differential operators to achieve L-stable real space solvers which scale as $\mathcal{O}(N)$ for $N$ spatial points.


Here we turn our attention to parabolic problems of higher order, specifically the CH equation and its extensions. To this end, we highlight the relevant new features of our approach:
\begin{enumerate}[label=\roman*]
\item We combine our previous work on the MOL$^{\text{T}}$ formulation \cite{causley2014method,causley2014higher,causley2015method,causley2016method} with traditional implicit time stepping methods (BDF, SDIRK and IDC methods, with time adaptivity);
\item we directly invert the linear part of the differential operator using the Green's function, combined with fast $\mathcal{O}(N)$ convolution;
\item we use dimensional splitting to avoid boundary integrals, and the need for global inversion;
\item we directly include the splitting error in our formulation by explicitly incorporating it into the nonlinear iterations;
\item we remain competitive with explicit methods in terms of wall clock time needed to solve each fixed point iteration.
\end{enumerate}
It is worth pointing out that the MOL$^{\text{T}}$ formulation is really targeted at developing new paradigms for parallel multicore computing. The point to this paper, as well as our previous work, is to establish the theoretical and practical underpinnings necessary to address a wide class of PDEs. In particular, our task is to investigate the implementation for higher order PDEs. Parallel implementation  will be the focus of our upcoming work in this area.

The rest of this organized as follows. In Section \ref{sec:models}, we present several models which are of interest in this work. In Section \ref{sec:molt}, we derive a first order scheme for CH equation in a basic 1D setting. 

In Section \ref{sec:higher_order}, we modify the traditional time stepping scheme,  to achieve higher orders of accuracy and present temporal refinement studies.  In addition, we extend 1D schemes to multiple spatial dimension in Section \ref{sec:multiD}. Finally, we describe an adaptive time stepping strategy in Section \ref{sec:adaptiveTime} and present numerous numerical results including 1D, 2D solutions to the CH equations and the vector CH model in Section \ref{sec:numerical}, followed by some concluding remarks.

\section{Models}
\label{sec:models}

In this paper, we present the proposed Method of lines transpose (MOL$^{\text{T}}$) scheme for Cahn-Hilliard (CH) equation
\begin{equation}  \label{eqn:CahnHilliard} 
	u_t= -\Delta \left[ \epsilon^2 \Delta u - f(u) \right], \quad {\bf x} \in \Omega \subset \mathbb{R}^d,
\end{equation}
with the appropriate initial and boundary conditions. The phase function $u \in H^1(\Omega)$  describes the volume fraction of one component of a binary mixture where $H^1(\Omega)$ is the standard Sobolev space. The reaction function $f(u)$ is the derivative of classical Ginzburg-Landau double-well potential $F(u) =  \frac{1}{4}(u^2-1)^2$ whose local minima is at $u = \pm 1$. The small parameter $0 < \epsilon \ll 1$ is the width of the interfacial transition layer. In a 1D setting ($\Omega \equiv [a,b] \subset \mathbb{R}$), the Laplacian $\Delta$ in the above equations is replaced by $\partial_{xx}$.

The CH equation \eqref{eqn:CahnHilliard} describes the $H^{-1}$ gradient flow of the CH free energy \cite{cahn1958free}
\begin{equation}  \label{eqn:EnergyFunctional}
	\mathcal{E}(u) =\int_{\Omega} \frac{\epsilon^2}{2} |\nabla u|^2 + F(u) d{\bf x}.
\end{equation}
If we assume zero-flux boundary conditions, the CH free energy is dissipative
\begin{equation} \notag
\frac{d}{dt} \mathcal{E}(u) = \left<u_t, \frac{\delta \mathcal{E}}{\delta u} \right>_{L^2} \leq 0. 
\end{equation}
We emphasize that the physical property of energy stability is vital, and any consistent numerical scheme for the CH model must also possess this property.
Many such schemes have been proposed, such as the operator splitting approach introduced by Eyre \cite{eyre1998unconditionally}, the semi-implicit spectral deferred correction method by Shen, \cite{shen2010, shen2012}, and even fully implicit schemes utilizing nonlinear solvers, such as the conjugate gradient method \cite{Jaylan}. We shall use a linearly implicit fixed point method to solve the CH equation, which is stabilized by shifting the phase field about the background state $u=-1$.

We also consider vector version of CH (VCH) equation \cite{Jaylan}.  For ${\bf u} = (u_1,u_2)$,
\begin{equation} \label{eqn:CH_vector}
{\bf u}_t = -\Delta \left[ \epsilon^2 \Delta {\bf u} - \nabla_{\bf u} W({\bf u}) \right]
\end{equation}
where the reaction term is the derivative of the potential function,
\begin{equation} \label{eqn:CH_vector_reaction}
W({\bf u}) = \Pi_{i=1}^{3} | {\bf u} - {\bf z}_i |^2, \qquad {\bf z}_i: \text{cube roots of unity in the } (u_1,u_2) \text{ plane}.
\end{equation}
The potential $W$ is non-negative and its minimum values are attained at three vectors $\{ {\bf z}_i \}$, so as to model a three-phase physical system. This model can be seen as the higher order volume preserving version of the vector-valued Ginzber-Landau equation  \cite{bronsard1993three} , which suggests a model for three phase boundary motion, such as the grain-boundary motion in alloys. The energy functional is a vector version of \eqref{eqn:EnergyFunctional} such that 
\begin{equation}  \label{eqn:EnergyFunctional_vector}
\mathcal{E}_{VCH}({\bf u}) =\int_{\Omega} \frac{\epsilon^2}{2} |\nabla {\bf u}|^2 + W({\bf u}) d{\bf x}.
\end{equation}

\section{First order scheme in one spatial dimension}
\label{sec:molt}
Before employing a time discretization to CH equation \eqref{eqn:CahnHilliard}, we first introduce a new transformed variable $v = u+1$. Since pure states $u=\pm 1$ dominates the solution during ripening, we see that equivalently $v=0$, 2. This transformation is partially motivated by \cite{Zhengfu}, in which the same approach is combined with operator splitting to produce a contractive operator. Similarly, in \cite{Jaylan} the linearized variable $v\approx u+1$ is used as a preconditioner for the full problem. Our work is distinct, however in that we are not solving the linearized equation, but instead the fully nonlinear equation
\begin{equation} \label{eqn:CahnHilliard_v} 
v_t= -\epsilon^2 v_{xxxx} + f(v-1)_{xx}, \quad f(v-1) = v^3 -3v^2+2v,
\end{equation}
which follows from the transformation $v=u+1$ inserted into the CH equation \eqref{eqn:CahnHilliard}, with $\Omega = (a,b)$. Similarly, the energy functional \eqref{eqn:EnergyFunctional} becomes 
\begin{equation} \label{eqn:EnergyFunctional_v} 
\mathcal{E}(v) =\int_a^b \frac{\epsilon^2}{2} |v_x|^2 + F(v-1) d{\bf x}, \quad F(v-1) = \frac{1}{4}v^4 - v^3 + v^2
\end{equation}
In Section \ref{sec:molt_CH}, we will formulate our first-order MOL$^{\text{T}}$ scheme using above equation \eqref{eqn:CahnHilliard_v}, and will show that this transformation makes our scheme gradient stable in Section \ref{sec:molt_CH_energy}.

\subsection{Semi-discrete scheme for 1D CH equation} \label{sec:molt_CH} 
We utilize the MOL$^T$ by discretizing  \eqref{eqn:CahnHilliard_v} in time as the Backward Euler (BE) scheme, 
\begin{equation}  \label{eqn:CH_BE}
\dfrac{v^{n+1}-v^n}{\Delta t} = -\epsilon^2 \partial_{xxxx} v^{n+1} + 2\partial_{xx} v^{n+1} + \partial_{xx} \tilde{f}^{n+1},  \quad \tilde{f}(v) = v^3 -3v^2,
\end{equation}
where $v^n = u(x,t^n)+1$ and $\Delta t = t^{n+1} - t^n$. This scheme \eqref{eqn:CH_BE} is first-order in time but still continuous in space. We note that the equation contains both linear and nonlinear implicit terms, and so an efficient iterative scheme is required to construct a fully implicit solution. We will suggest two nonlinear iteration schemes to solve the solution $v^{n+1}$ in Section \ref{sec:fixed_pt_iteration}. We first prove that the semi-discrete solution of \eqref{eqn:CH_BE} is unconditionally gradient stable. 

\subsubsection{Energy stability}\label{sec:molt_CH_energy}
Our proof is similar to those which have appeared elsewhere \cite{shen2010,eyre1998unconditionally}, in that we utilize the $H^{-1}$, under reasonable assumptions. However, we investigate the fully implicit scheme \eqref{eqn:CH_BE}, which alters the linear part of the PDE. We first make useful observations.
\begin{itemize}
	\item The operator $\Delta^{-1} : \left \{ v \in L^2 | \int_\Omega v dx = 0\right \} \rightarrow H^2(\Omega)$ is defined \cite{Zhengfu} as 
\begin{equation}  \label{item:lemma1}
\Delta^{-1}v = h \quad \Longleftrightarrow \quad <v, q> = <\Delta h, q>, \quad \forall q \in L^2(\Omega).
\end{equation}

	\item  For any $p, q \in L^2(\Omega)$, the following identity holds 
\begin{equation}  \label{item:lemma2}
<p, p - q> = \frac{1}{2} \left( \| p \|_0^2 -  \| q \|_0^2 +  \| p-q \|_0^2 \right) \geq \frac{1}{2} \left( \| p \|_0^2 -  \| q \|_0^2 \right),
\end{equation}
where $< \cdot, \cdot>$ is standard $L^2$ inner product and $\| \cdot \|_0$ is the $L^2$ norm.
\end{itemize}
We will consider the first-order scheme  \eqref{eqn:CH_BE}  in general spatial dimension $(\Omega \in \mathbb{R}^d, d \in \mathbb{N})$.

\begin{lemma}
Under the following assumption, 
$$ 0 \leq v \leq 2, \quad  v = u+1$$
the fully-implicit scheme \eqref{eqn:CH_BE} satisfies the discrete energy law for any time step $\Delta t$: 
$$ \mathcal{E} (v^{n+1}) \leq \mathcal{E} (v^{n}), \quad \forall n \geq 0.$$
where $v^n$ is approximation of $v(x,t^n) \equiv u(x,t^n)+1$ of Cahn-Hilliard equation \eqref{eqn:CahnHilliard}.
\end{lemma}
\begin{proof} 
The solution $v^{n+1}$ satisfies the following weak formulation in $H^2(\Omega)$,
\begin{equation} \label{eqn:weakform}
\frac{1}{\Delta t} <v^{n+1}-v^n,\phi> = -\epsilon^2 <\Delta v^{n+1}, \Delta \phi>+2<v^{n+1},\Delta \phi>+<\tilde{f}^{n+1},\Delta \phi>,
\end{equation}
for all $\phi \in H^2(\Omega)$. Utilizing \eqref{item:lemma1}, we choose the test function $\phi = \Delta^{-1} (v^{n+1} - v^{n})$, so that
\begin{align} \notag
-\frac{1}{\Delta t} \| \nabla \cdot \Delta^{-1} (v^{n+1} - v^{n}) \|_0^2= \epsilon^2 &<\nabla v^{n+1}, \nabla (v^{n+1}- v^{n})> \\ \label{eqn:weakform2}
&+2<v^{n+1},v^{n+1} - v^n>+<\tilde{f}^{n+1}, v^{n+1} - v^n>, 
\end{align}
where integration by parts is employed. Next we apply the identity \eqref{item:lemma2}, 
\begin{equation} \label{weakform_inequal}
0 \geq \hspace{2mm} \frac{\epsilon^2}{2} ( \| \nabla v^{n+1} \|_0^2 -  \| \nabla v^n \|_0^2 ) + ( \| v^{n+1} \|_0^2 -  \| v^n \|_0^2 +  \| v^{n+1} -v^n \|_0^2 ) +<\tilde{f}^{n+1}, v^{n+1} - v^n>.
\end{equation}
We replace the last term with its Taylor expansion
$$ <\tilde{F}^{n+1} - \tilde{F}^n, 1> = <\tilde{f}^{n+1}, v^{n+1} - v^n> - <\frac{\tilde{f}'(\xi^{n+1})}{2}(v^{n+1} - v^n), v^{n+1} - v^n>,$$
so that the inequality \eqref{weakform_inequal} becomes
\begin{align} \notag
\mathcal{E}(v^{n+1}) - \mathcal{E}(v^n) &= \frac{\epsilon^2}{2} ( \| \nabla v^{n+1} \|_0^2 -  \| \nabla v^n \|_0^2 ) + ( \| v^{n+1} \|_0^2 -  \| v^n \|_0^2 ) + <\tilde{F}^{n+1} - \tilde{F}^n, 1> \\ \notag
							& \leq \left(\frac{\|\tilde{f}'\|_\infty}{2}-1\right)  \| v^{n+1} -v^n \|_0^2 \leq 0,
\end{align}
where $\|\cdot\|_\infty$ is $L^{\infty}$ norm. The last inequality holds because $\tilde{f}' = 3v^2 - 6v \leq 0$ for $0\leq v\leq 2$. 
Therefore the discrete energy is non-increasing, and the implicit scheme \eqref{eqn:CH_BE} is unconditionally gradient stable. 
\end{proof}

\begin{remark}  
In \cite{shen2010, shen2012}, Shen et.al. introduced a stabilizing term $S(u^{n+1} - u^n)$ for their semi-implicit scheme for the AC and CH equations. They choose the constant to satisfy $S\geq \dfrac{\max_{u \in \mathbb{R}} | \tilde{f}'(u) |}{2}$ with the truncated potential $-1\leq u \leq 1$. Equivalently, we assume that $0\leq v\leq 2$. However, we alter the linear part of the differential operator to ensure stability, rather than introducing an additional term.The distinction is that the transformation we employ modifies the linear part of the differential operator, rather than introducing an additional term, for stabilization.
\end{remark}


\subsubsection{Factorization of the linear operator} \label{sec:fixed_pt_iteration} 
We will now propose a simple fixed point iterative solver for equation \eqref{eqn:CH_BE},
\begin{equation}  \label{eqn:CH_BE_fixed}
\left( I - (2\Delta t \partial_{xx} - \epsilon^2 \Delta t  \partial_{xxxx})\right) v^{n+1,k+1} = v^{n} + \Delta t \partial_{xx} \tilde{f}^{n+1,k},
\end{equation}
where $k$ indicates the iteration index. By lagging the the nonlinear term $\tilde{f}^{n+1,k}$, we can make the update explicit through formal \textit{analytic} inversion of the fourth-order operator
\begin{equation}  \label{eqn:CH_BE_fixed2}
v^{n+1,k+1} = \left( I - 2\Delta t \partial_{xx} + \epsilon^2 \Delta t  \partial_{xxxx}\right)^{-1} \left[ v^{n} + \Delta t \partial_{xx} \tilde{f}^{n+1,k}\right].
\end{equation}
However, the Green's function of this fourth order operator contain both decaying and oscillatory components, which would not produce an efficient method. Instead we propose a novel factorization technique, which produces two second order decaying (modified Helmholtz) differential operators, and in doing so leverage the fast computation methods previously employed in \cite{causley2016method}. After some algebra, we rewrite \eqref{eqn:CH_BE_fixed2} as
\begin{equation}  \label{eqn:CH_BE_fixed2_factor}
v^{n+1,k+1} = \left( I -  \frac{\partial_{xx}}{\alpha_1^2} \right)^{-1} \left( I -  \frac{\partial_{xx}}{\alpha_2^2} \right)^{-1}\left[ v^{n} + \Delta t \partial_{xx} \tilde{f}^{n+1,k}\right], 
\end{equation}
where
\[
	\dfrac{1}{\alpha_1^2} = \Delta t + \sqrt{\Delta t^2 - \epsilon^2 \Delta t}, \qquad \dfrac{1}{\alpha_2^2} = \Delta t - \sqrt{\Delta t^2 - \epsilon^2 \Delta t}, 
\]
so that $ \alpha_1, \alpha_2 \in \mathbb{R}$ when $\Delta t \geq \epsilon^2$. This implies convolution with the Green's functions $e^{-\alpha_1|x|}$ and $e^{-\alpha_2|x|}$. For $\Delta t < \epsilon^2$, these parameters become complex, and the Green's functions are oscillatory. To overcome this difficulty, we switch to an alternate factorization to solve $v^{n+1}$: 
\begin{align}  \label{eqn:CH_BE_fixed_complete}
 v^{n+1,k+1} = \left( I - \sqrt{\epsilon^2 \Delta t}  \partial_{xx}\right)^{-2}  \left[   v^{n} + \Delta t \partial_{xx} \left( \tilde{f}^{n+1,k} - 2\left(1+\sqrt{\frac{\epsilon^2}{\Delta t}}\right)v^{n+1,k} \right)   \right], 
\end{align}
which is based on completing the square, and lagging linear terms which are incorporated into $\tilde{f}$. In this case, we convolve twice with the Green's function $e^{-\alpha|x|}$, where $\alpha = \frac{1}{\epsilon\sqrt{\Delta t}}$. This switching is summarized by the bifurcation diagram shown in Figure \ref{fig:bifurcation}.

\begin{figure}[h]
\centering
   \includegraphics[width=0.55\linewidth]{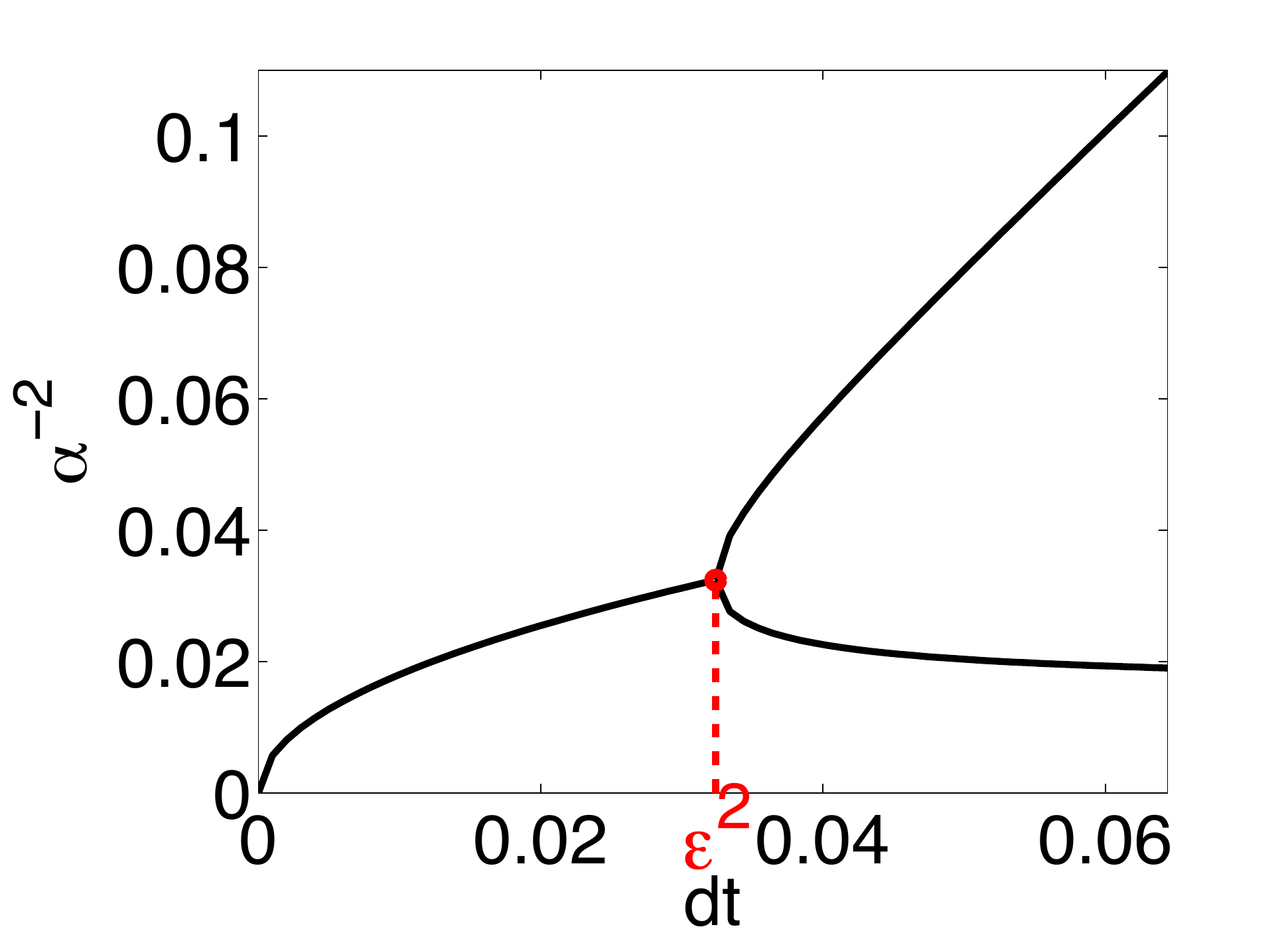}
    \caption{Bifurcation of the parameter $\dfrac{1}{\alpha^2}$ in the modified Helmholtz operator $\mathcal{L}$. }
    \label{fig:bifurcation}
\end{figure}

Since both equations \eqref{eqn:CH_BE_fixed2_factor} and \eqref{eqn:CH_BE_fixed_complete} converge to \eqref{eqn:CH_BE} as the iterations $k\to \infty$, the overall scheme will be gradient stable. Thus, we invert the operator based on the parameters $\Delta t$ and $\epsilon$ as follows:
\begin{align} \notag
	\text{Factorization method} (\Delta t \geq  \epsilon^2): \quad &\text{Invert } \left(I - \frac{\partial_{xx}}{\alpha_1^2} \right) \left(I - \frac{\partial_{xx}}{\alpha_2^2} \right) \\ \notag
	\text{Completed square method } (\Delta t <  \epsilon^2):  \quad &\text{Invert } \left(I - \frac{\partial_{xx}}{\alpha^2} \right)^2.
\end{align}

\begin{remark} \label{remark:twoiterations}
The second iterative scheme \eqref{eqn:CH_BE_fixed_complete} can be equivalently viewed as introducing a term of the form $S\partial_{xx}(v^{n+1,k+1} - v^{n+1,k})$ into the first \eqref{eqn:CH_BE_fixed2_factor}, with $S = 2\sqrt{\epsilon^2 \Delta t}$. Numerical results confirm that both regimes, namely $\Delta t<\epsilon^2$ and $\Delta>\epsilon^2$, exhibit gradient stability.
\end{remark}

Both iteration schemes require two inversions of a modified Helmholtz operator in \cite{causley2016method}, hence
\begin{equation}\label{eqn:L_Inverse}
 \mathcal{L} = I - \frac{\partial_{xx}}{\alpha^2}, 
\end{equation}
with the different values of $\alpha$. As we have indicated above, formal inversion of this operator \cite{causley2016method} yields:
\begin{equation}
\notag
\mathcal{L}^{-1}[v](x) = \underbrace{\mathcal{I}[v](x)}_{\text{Particular Solution}}+ \underbrace{\vphantom{I_x[u](x)} B_a e^{-\alpha(x-a)} + B_b e^{-\alpha(b-x)}}_{\text{Homogeneous Solution}}, \quad a\leq x \leq b,
\end{equation}
where the particular solution is a convolution with the Green's function,
\begin{equation}\label{eqn:Iu}
 \mathcal{I}[v](x) = \frac{\alpha}{2} \int_a^b e^{-\alpha |x-x'|} v(x') dx', 
\end{equation}
and the coefficients $B_a$ and $B_b$ are determined by applying boundary conditions at $x=a,b$ in Section \ref{sec:molt_CH_fully}.

\begin{remark} \label{remark:Laplacian}
We observe that the Green's function will be convolved with terms of the form $\partial_{xx} f$, and so we will use the following result below
$$ \partial_{xx} = \alpha^2 (I - \mathcal{L}) \quad \Longrightarrow \quad  \mathcal{L}^{-1} [\partial_{xx} v]  = \alpha^2 \left(  \mathcal{L}^{-1} -  I \right) [v],$$
where we use the definition \eqref{eqn:L_Inverse} of the modified Helmholtz operator $\mathcal{L}$.
\end{remark}

\subsection{Fully-discrete solution to 1D CH equation} \label{sec:molt_CH_fully} 
It remains to discretize the convolution integral \eqref{eqn:L_Inverse}, and obtain a fully discrete algorithm. In our previous works \cite{causley2014method,causley2014higher,causley2015method,causley2016method}, we accomplish this with fast convolution; for convenience, we will summarize the fast algorithm here. 

First, the particular solution  \eqref{eqn:Iu} is split into $\mathcal{I}[v](x) = \mathcal{I}^L(x)+ \mathcal{I}^R(x)$, where
$$\mathcal{I}^L(x) = \frac{\alpha}{2} \int_a^x e^{-\alpha(x-x')}v(x')dx', \quad  \mathcal{I}^R(x)= \frac{\alpha}{2}\int_{x}^b e^{-\alpha(x'-x)}v(x')dx'. $$
The domain $\Omega \equiv [a,b]$ is partitioned into $N$ subdomains $[x_{j-1},x_j]$ 
$$a=x_0 < x_1 <\cdots < x_N = b, \quad h_j = x_{j} - x_{j-1},$$
we then evaluate the convolution operator at each grid point through $\mathcal{I}_j \equiv \mathcal{I}[v](x_j) = \mathcal{I}^L_j+ \mathcal{I}^R_j$.
Since each integral satisfies exponential recursion \cite{causley2014higher,causley2015method,causley2016method}, 
\begin{align} \label{eqn:IL}
\mathcal{I}^L_j &= e^{-\alpha h_j}\mathcal{I}^L_{j-1} + \mathcal{J}^L_j, \quad \mathcal{J}^L_j = \frac{\alpha}{2} \int_{x_{j-1}}^{x_j} e^{-\alpha (x_j -x')} v(x') dx', \\ \label{eqn:IR}
\mathcal{I}^R_j &= e^{-\alpha h_{j+1}}\mathcal{I}^R_{j+1} + \mathcal{J}^R_j, \quad \mathcal{J}^R_j = \frac{\alpha}{2} \int_{x_j}^{x_{j+1}} e^{-\alpha (x'-x_j)} v(x') dx'.
\end{align}
Those recursive updates are still exact, and only require computing the "local" convolutions $\mathcal{J}_j^L$ and $\mathcal{J}_j^R$. In \cite{causley2016method}, we derived a spatial quadrature of general order $M \geq 2$ for computing  these local integrals. For example, the second-order accurate $(M=2)$ quadrature on a uniform grid $(h_j = \Delta x)$ is given by 

\begin{align}
	J^{L}_{j} &\approx P v(x_{j})+Qv(x_{j-1})+R(v(x_{j+1})-2v(x_{j})+v(x_{j-1})) \\
	J^{R}_{j} &\approx P v(x_{j})+Qv(x_{j+1})+R(v(x_{j+1})-2v(x_{j})+v(x_{j-1}))
\end{align}
where defining $\nu = \alpha \dx$ and $d = e^{-\nu}$, the quadrature weights are given by
\begin{align} \notag
	P&= 1-\frac{1-d}{\nu}, \\ \notag
	Q&= -d+\frac{1-d}{\nu}, \\ \notag
	R&= \frac{1-d}{\nu^{2}}-\frac{1+d}{2\nu}.
\end{align}
Likewise, for higher order methods ($M>2$), the quadrature weights are pre-computed, so that the fast convolution algorithm is achieved as $\mathcal{O}(MN)$ per time step with user-defined order $M$ in space. In every simulation in this work, we choose $M=4$ or $M=6$. \\

Next, the homogeneous solution \eqref{eqn:L_Inverse} is used to enforce various boundary conditions. For example periodic boundary conditions lead to
$$ v^n (a) = v^n(b), \quad v_x^n(a) = v_x^n(b), \quad \forall n \in \mathbb{N}. $$ 
and evaluation of \eqref{eqn:L_Inverse} at $x=a,b$, produces a $ 2 \times 2 $ system for the unknowns $B_a$ and $B_b$: 
\begin{align} \notag
\mathcal{L}^{-1}[v^n](a) = \mathcal{L}^{-1}[v^n](b) \quad &\Longleftrightarrow \quad \mathcal{I}_0 + B_a + B_b \mu =  \mathcal{I}_N + B_a\mu + B_b, \\ \notag
\mathcal{L}_x^{-1}[v^n](a) = \mathcal{L}_x^{-1}[v^n](b) \quad &\Longleftrightarrow \quad \alpha (\mathcal{I}_0 - B_a + B_b \mu )=  \alpha (-\mathcal{I}_N + B_a\mu + B_b).
\end{align}
Solving this linear system yields
\begin{equation} \label{eqn:homogeneous} 
B_a = \frac{\mathcal{I}_N}{1 - \mu}, \quad B_b = \frac{\mathcal{I}_0}{1 - \mu},
\end{equation}
where $\mu = e^{-\alpha (b-a)}$ and $\mathcal{I}_0 = \mathcal{I}[v](a)$ and $\mathcal{I}_N = \mathcal{I}[v](b)$. Other boundary conditions (e.g. Neumann, Dirichlet) follow the analogous procedure that requires solving a simple $2 \times 2$ linear system. (cf. \cite{causley2015method})

\section{Higher order schemes}
\label{sec:higher_order}
In this Section we formulate second and third order time accurate methods by combining the ideas of MOL$^\text{T}$ formulation with backward difference formulas (BDF), singly diagonal implicit Runge-Kutta (SDIRK), and spectral deferred correction (SDC) methods. We will present several refinement studies to compare these approaches. 

\subsection{MOL$^\text{T}$ with BDF} \label{sec:BDF}
The BDF time stepping methods are one of the most commonly used implicit linear multi-step methods \cite{iserles2009first}. To achieve the desired order of accuracy, we discretize the time derivative using BDF formulas as follows:

\begin{align}
	\label{eqn:BDF2}
	\text{BDF2:}\quad & \left(I+\frac{2}{3}\epsilon^2 \Delta t \partial_{xxxx} \right)v^{n+1} = \frac{4}{3} v^n - \frac{1}{3}v^{n-1} + \frac{2}{3}\Delta t \partial_{xx} f^{n+1}, \\
	\label{eqn:BDF3}
	\text{BDF3:}\quad &  \left(I+\frac{6}{11}\epsilon^2 \Delta t \partial_{xxxx} \right)v^{n+1} = \frac{18}{11} v^{n} - \frac{9}{11}v^{n-1} + \frac{2}{11} v^{n-2}+ \frac{6}{11}\Delta t \partial_{xx} f^{n+1},
\end{align}
where we now require two and three previous time steps, respectively. Similar to the first-order scheme in Section \ref{sec:molt_CH}, we define two fixed-point iterations, depending on the time step $\Delta t$.
 First, energy-stable factorization iteration which has a time step lower bound for real-valued Green's functions. Second, completed square version by adding $S(v^{n+1,k+1} - v^{n+1,k})$ term, which has a time step upper bound for energy stability.
Specifically, the BDF 2 formulation \eqref{eqn:BDF2} utilizes
\begin{align}
	\notag
	\Delta t \geq \frac{3}{2} \epsilon^2: \quad & \mathcal{L} = I - \frac{4}{3}\Delta t \partial_{xx} +\frac{2}{3}\epsilon^2 \Delta t \partial_{xxxx} = \left(I - \frac{\partial_{xx}}{\alpha_1^2} \right) \left(I - \frac{\partial_{xx}}{\alpha_2^2} \right), \\
	\notag
	\Delta t < \frac{3}{2} \epsilon^2: \quad & \mathcal{L} = I - 2\sqrt{\frac{2}{3}\epsilon^2 \Delta t} \partial_{xx} +\frac{2}{3}\epsilon^2 \Delta t \partial_{xxxx} = \left(I - \frac{\partial_{xx}}{\alpha^2} \right)^2, 
\end{align}
where
\[
	\frac{1}{\alpha_1^2} = \frac{2}{3}\Delta t + \sqrt{\left(\frac{2}{3}\Delta t\right)^2 - \frac{2}{3}\epsilon^2 \Delta t}, \quad \frac{1}{\alpha_2^2} = \frac{2}{3}\Delta t - \sqrt{\left(\frac{2}{3}\Delta t\right)^2 - \frac{2}{3}\epsilon^2 \Delta t}, \quad \frac{1}{\alpha^2} = \sqrt{\frac{2}{3}\epsilon^2 \Delta t}.
\]
A similar result is obtained for the BDF 3 method.
The main advantage of BDF schemes is that the extension of first-order methods to higher-order is very straightforward. However, they also require initialization of the first few time steps. In next section, we will look at Ruge-Kutta (RK) methods, which do not have this requirement. 

\subsection{MOL$^\text{T}$ with SDIRK} \label{sec:SDIRK}
Singly implicit RK methods (SDIRK) are those which have the same diagonal elements in the Butcher Table, \cite{alexander1977diagonally}. To obtain a $P^{\text th}$ order method, RK methods, $P$ intermediate steps are computed at each time step.
The second order method, SDIRK2, is as follows
\begin{align} \notag
\left( I + \eta \epsilon^2 \Delta t \partial_{xxxx} \right) K_1&= -\epsilon^2 \partial_{xxxx} v^n + \partial_{xx} f\left(v^n + \eta \Delta t K_1\right), \quad \eta = 1 - \frac{\sqrt{2}}{2}, \\ \notag
\left( I + \eta \epsilon^2 \Delta t \partial_{xxxx} \right) K_2&=  -\epsilon^2 \partial_{xxxx} \left( v^n + (1-\eta)\Delta t K_1 \right) + \partial_{xx} f\left(v^n + (1-\eta) \Delta t K_1+\eta \Delta t K_2\right) \\ \notag
 v^{n+1} &= v^n + \Delta t \{ (1-\eta) K_1 + \eta K_2 \},
\end{align}
where $\eta$ is the root of the polynomial $\frac{1}{2} - 2\eta + \eta^2$, as derived from the order conditions \cite{alexander1977diagonally}. Next, we construct two intermediate solution $K_1$ and $K_2$, using nonlinear iterative schemes based on Section \ref{sec:BDF}. First, if $\Delta t \geq \dfrac{\epsilon^2}{\eta}$, 
\begin{align} \notag
 K_1^{n+1,k+1}&= \mathcal{L}_1^{-1} \mathcal{L}_2^{-1} \left[- \epsilon^2 \partial_{xxxx} v^n + \partial_{xx} \left( \tilde{f}\left(v^n + \eta \Delta t K_1^{n+1,k}\right) +2v^n \right)\right], \quad \mathcal{L}_i = I - \frac{\partial_{xx}}{\alpha_i^2}\\ \notag
 K_2^{n+1,k+1}&=  \mathcal{L}_1^{-1} \mathcal{L}_2^{-1} \biggl[ -\epsilon^2 \partial_{xxxx} \left( v^n + (1-\eta)\Delta t K_1^{n+1,k} \right) \\ \notag
+& \partial_{xx} \left( \tilde{f}\left(v^n + (1-\eta)\Delta t K_1^{n+1,k}+ \eta \Delta t K_2^{n+1,k}\right) + 2(v^n + (1-\eta) \Delta t K_1^{n+1,k}) \right) \biggr],
\end{align}
where $\dfrac{1}{\alpha_1^2} = \eta\Delta t + \sqrt{(\eta\Delta t)^2 - \epsilon^2 \eta\Delta t}$ and $\dfrac{1}{\alpha_2^2} = \eta\Delta t - \sqrt{(\eta\Delta t)^2 - \epsilon^2 \eta\Delta t}$.
If $\Delta t < \dfrac{\epsilon^2}{\eta}$, 
\begin{align} \notag
 K_1^{n+1,k+1}&= \mathcal{L}^{-2}  \left[ -\epsilon^2 \partial_{xxxx} v^n + \partial_{xx} \left( f\left(v^n + \eta \Delta t K_1^{n+1,k}\right) -2\sqrt{\eta \epsilon^2 \Delta t} K_1^{n+1,k} \right)\right], \\ \notag
 K_2^{n+1,k+1}&= \mathcal{L}^{-2} \biggl[ -\epsilon^2 \partial_{xxxx} \left( v^n + (1-\eta)\Delta t K_1^{n+1,k} \right) \\ \notag
& + \partial_{xx} \left(f\left(v^n + (1-\eta)\Delta t K_1^{n+1,k}+ \eta \Delta t K_2^{n+1,k}\right) -2\sqrt{\eta \epsilon^2 \Delta t} K_2^{n+1,k} \right) \biggr],
\end{align}
where $\mathcal{L} \equiv  I - \sqrt{\eta \epsilon^2 \Delta t}\partial_{xx} $ is used. Below we will also make use of SDIRK3, which can be derived in a similar fashion.

\begin{remark} \label{remark:biLaplacian}
Similar to Remark \ref{remark:Laplacian}, the fourth derivative can be calculated as follows,
$$ \partial_{xxxx} = \alpha_1^2 \alpha_2^2(I - \mathcal{L}_1)(I - \mathcal{L}_2) \quad \Longrightarrow \quad \mathcal{L}_1^{-1} \mathcal{L}_2^{-1}[\partial_{xxxx} v]  = \alpha_1^2 \alpha_2^2 \left(  \mathcal{L}_1^{-1} - I \right)  \left(  \mathcal{L}_2^{-1} - I \right)[v].$$
\end{remark}

In Section \ref{sec:Refinement} we will present refinement studies for both the BDF and SDIRK methods. We note that for a small enough time step $\Delta t$, both methods converge as expected. However for the corresponding third-order methods (BDF3 and SDIRK3), as $\Delta t$ increases the order of convergence begins to plateau. To resolve this issue, we will suggest another higher order method.

\subsection{MOL$^\text{T}$ with SDC} \label{sec:SDC}
The spectral deferred correction (SDC) methods are a class of time integrators \cite{dutt2000spectral}. First a prediction to the solution ("level 0") is computed using low order schemes (e.g. Backward Euler), and then the error (defect) is corrected at subsequent stages, using higher order integrators. Our SDC procedure follows the presentation in \cite{Ong}.

	\begin{enumerate}[itemsep=2mm]
		\item {\bf Prediction step (level [0]):} Subdivide the time interval $[0,T]$ with uniform $\Delta t$: 
		$$0\equiv t^0 < t^1 < \cdots < t^{N_t} \equiv T.$$
		Compute $\{ v^{[0]}_m (x) \}_{0 \leq m \leq N_t}$ via Backward Euler approximation in \eqref{eqn:CH_BE} for CH equation, which is first-order approximation to the exact solution $\{ y_m (x) \}_{0 \leq m \leq N_t}$.
		
		\item {\bf Correction step (level [1]):} Assume that $v^{(0)}(x,t)$ is a polynomial interpolant, approximating the exact solution $y(x,t)$ satisfying
			$$v^{(0)}(x,t^m) \equiv v^{[0]}_m(x), \quad m = n+1, n, \cdots, n+1-P,$$ 
			where $P$ will be specified later. The error equation is defined 
			$$ e(x,t) = y(x,t) - v^{(0)}(x,t),$$ 
			and the residual (or "defect") is 
			$$ \gamma(x,t) = v^{(0)}_t - \mathcal{F}_{\text{CH}}(v^{(0)}), \quad \mathcal{F}_{\text{CH}}(v) = -\epsilon^2\partial_{xxxx}v + \partial_{xx} f(v).$$ 
			We take the derivative of the error equation with respect to $t$, and rewrite it using the residual definition, 
			\begin{align} \notag 
 &\quad e_t (x,t) = y_t(x,t)- v^{(0)}_t(x,t) \equiv \mathcal{F}_{\text{CH}}(y(x,t)) - \mathcal{F}_{\text{CH}}(v^{(0)}(x,t)) - \gamma(x,t),  \\ \notag 
\Longleftrightarrow& \quad e_t(x,t) + \gamma(x,t) = \mathcal{F}_{\text{CH}}((e+v^{(0)})(x,t)) - \mathcal{F}_{\text{CH}}(v^{(0)}(x,t)), \\ \notag
\Longleftrightarrow& \quad \frac{\partial}{\partial t} \left (v^{(1)}(x,t) - \int_0^t  \mathcal{F}_{\text{CH}}(v^{(0)}(x,\tau)) d\tau \right) = \mathcal{F}_{\text{CH}}(v^{(1)}(x,t)) - \mathcal{F}_{\text{CH}}(v^{(0)}(x,t)), 
\end{align} 
			
			where we assume that the initial condition $e(x,0) = 0$, and $v^{(1)}= v^{(0)}+e$.
			Hence, our updating "level [1]" solution $\{ v^{[1]}_m \}$ is found by approximating the above differential equation with the same method as "level 0". We apply the backward Euler scheme, thus
			\begin{equation} \label{eqn:IDC_BE}
v^{[1]}_{n+1} - v^{[1]}_n = \Delta t \left(  \mathcal{F}_{\text{CH}}(v^{[1]}_{n+1}) -  \mathcal{F}_{\text{CH}}(v^{[0]}_{n+1}) \right) + \int_{t^n}^{t^{n+1}}  \mathcal{F}_{\text{CH}}(v^{(0)}(x,\tau)) d\tau.
\end{equation}

			\item {\bf Correction step (level [j]):} The process is then iterated by generalizing \eqref{eqn:IDC_BE}, 
				\begin{equation} \label{eqn:IDC_BE_j}
v^{[j]}_{n+1} - \Delta t  \mathcal{F}_{\text{CH}}(v^{[j]}_{n+1}) = v^{[j]}_n - \Delta t \mathcal{F}_{\text{CH}}(v^{[j-1]}_{n+1})  + \int_{t^n}^{t^{n+1}}  \mathcal{F}_{\text{CH}}(v^{(j-1)}(x,\tau)) d\tau,
\end{equation}
where the terms including updating solution $v^{[j]}$ have been collected on the left hand side, and the old solution $v^{[j-1]}$ are on the right hand side.
	\end{enumerate}

To complete SDC, we must consider an approximation of the integral in \eqref{eqn:IDC_BE_j}:
\begin{equation} \label{eqn:integral_quadrature}
 \int_{t^n}^{t^{n+1}}  \mathcal{F}_{\text{CH}}(v^{(j-1)}(x,\tau)) d\tau = \begin{cases} 
\Delta t \sum_{i=0}^{P} \tilde{q}_i  \mathcal{F}^{[j-1]}_{n+1-i}, \qquad &\text{if} \quad n \geq P-1, \\ 
\Delta t \sum_{i=0}^{P}\ \tilde{q}_i  \mathcal{F}^{[j-1]}_{i} \qquad &\text{if} \quad n < P-1,
\end{cases}
\end{equation}
where $ \mathcal{F}^{[j-1]}_{n} =  \mathcal{F}_{\text{CH}}(v^{[j-1]}_n)$, and $\tilde{q}_i$ are quadrature weights (cf. \cite{Ong}). Note that the number of terms in the sum \eqref{eqn:integral_quadrature} is $P+1$, where $P$ is the order of polynomial interpolation $v^{(j-1)}$. The integral must be approximated with increasing accuracy as the level increases, so that $P \geq j$ at "level $[j]$". In Section \ref{sec:Refinement}, we will show that this order $P$ affects the asymptotic region of stability of SDC3 method.\\

We now combine MOL$^\text{T}$ scheme with these higher-order SDC methods. For instance, the second-order SDC scheme (SDC2) only requires one more correction (level $[1]$). If  $\Delta t \geq \epsilon^2$,
\begin{align} \label{eqn:SDC_fixed}
v^{[1]}_{n+1,k+1} = \mathcal{L}_1^{-1} \mathcal{L}_2^{-1} \left[ v^{[1]}_{n} +\Delta t \partial_{xx} \left( \tilde{f}^{[1]}_{n+1,k} - \frac{ {f}^{[0]}_{n+1} - f^{[0]}_{n}}{2}\right) +\frac{\epsilon^2 \Delta t}{2} \partial_{xxxx} \left( v^{[0]}_{n+1} -v^{[0]}_{n} \right)\right]
\end{align}
where quadrature weights of \eqref{eqn:integral_quadrature} are $\tilde{q}_1 = \tilde{q}_2 = \frac{1}{2}$ ($P=1$: trapezoidal rule) and $\mathcal{L}_i$ are same with the first-order scheme in Section \ref{sec:molt_CH}. Similarly, if $\Delta t < \epsilon^2,$ 
\begin{align} \label{eqn:SDC2_fixed}
v^{[1]}_{n+1,k+1} = \mathcal{L}^{-2} \left[ v^{[1]}_{n} +\Delta t \partial_{xx} \left( f^{[1]}_{n+1,k} - 2\sqrt{\frac{\epsilon^2}{\Delta t}}v^{[1]}_{n+1,k} - \frac{ {f}^{[0]}_{n+1} - f^{[0]}_{n}}{2}\right) +\frac{\epsilon^2 \Delta t}{2} \partial_{xxxx} \left( v^{[0]}_{n+1} -v^{[0]}_{n} \right)\right]
\end{align}

\subsection{Refinement studies for higher order methods} \label{sec:Refinement}

We solve \eqref{eqn:CahnHilliard} with periodic boundary conditions, with the same initial condition used in \cite{Jaylan},
\begin{equation} \label{eqn:CH1D_initial}
u_0 (x) = \cos(2x) + \dfrac{1}{100} e^{\cos\left(x+\frac{1}{10}\right)}, \qquad x \in [0,2\pi].
\end{equation}
We integrate up to a final time $T_{\text{final}}$ using each second-order method. Since an exact solution is not available, we perform time refinement using successive approximations, and the approximate error with step size $\Delta t$ is $\| u_{\Delta t} - u_{\Delta t /2 } \|_{\infty}$, measured in the maximum norm. Successive errors for each second-order method is presented in Table \ref{tab:CH_refinement_second}.
The parameters used for all computations were:
\begin{equation} \label{eqn:parameters}
 \epsilon = 0.18, \quad \Delta x = \dfrac{2\pi}{512} \approx 0.0123, \quad  T_{\text{final}} = 1, \quad N_{\text{tol}} = 10^{-12},
\end{equation}
where $N_{\text{tol}}$ is a tolerance for fixed point iterations such that $\| v^{n+1,k+1} - v^{n+1,k} \|_{\infty} < N_{\text{tol}} $. We use $6^{\text{th}}$-order spatial quadrature, so that the dominant error was temporal.\\

\begin{figure}[hbtp] \label{1D_CH_energy}
	\centering
	\subfigure[Energy decay ($2^{\text{nd}}$ order methods)]{\label{fig:energy_secondorder}\includegraphics[width=.48\textwidth]{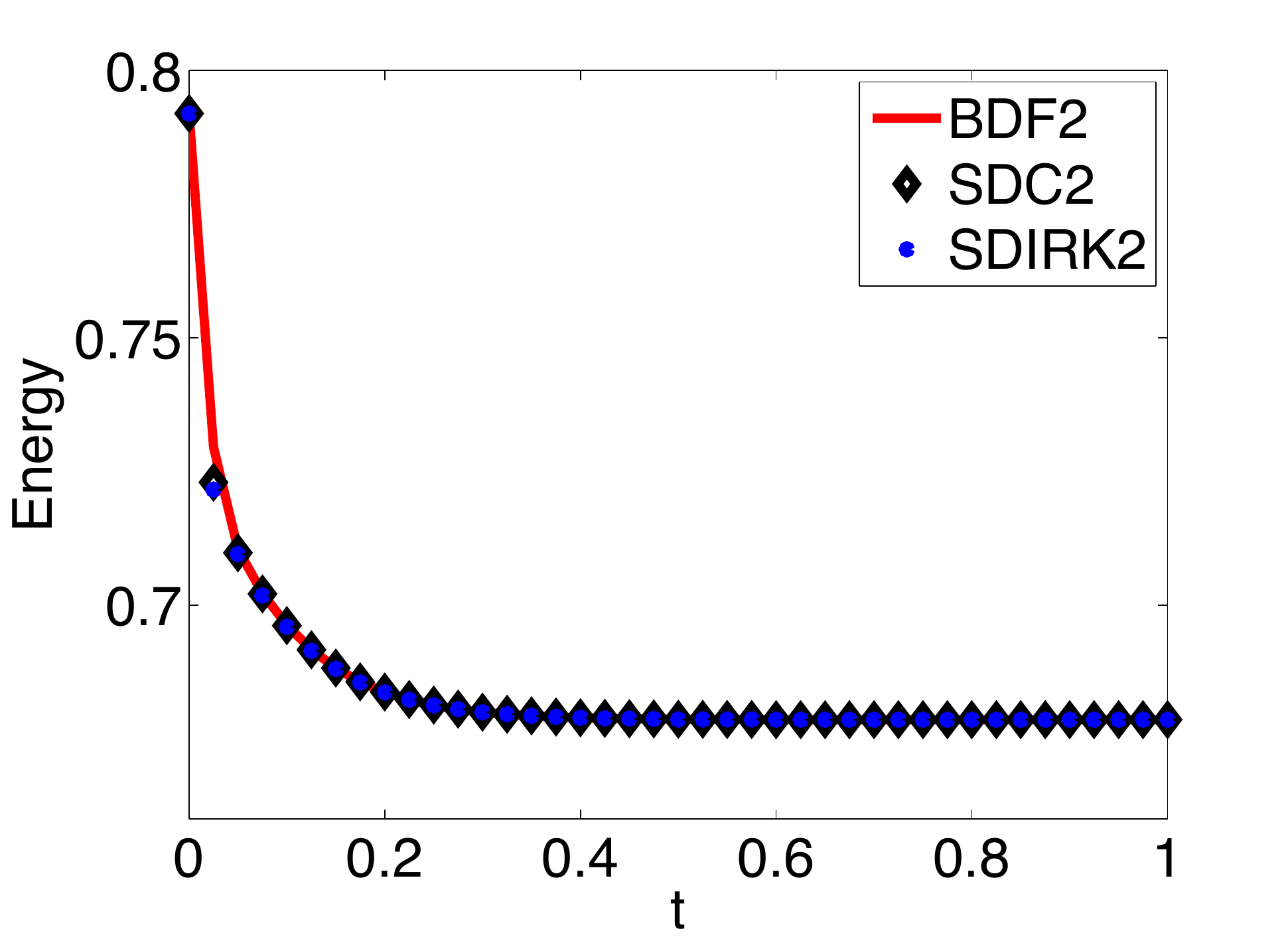}}
	\subfigure[Iteration count($2^{\text{nd}}$ order methods)]{\label{fig:iteration_number}\includegraphics[width=.48\textwidth]{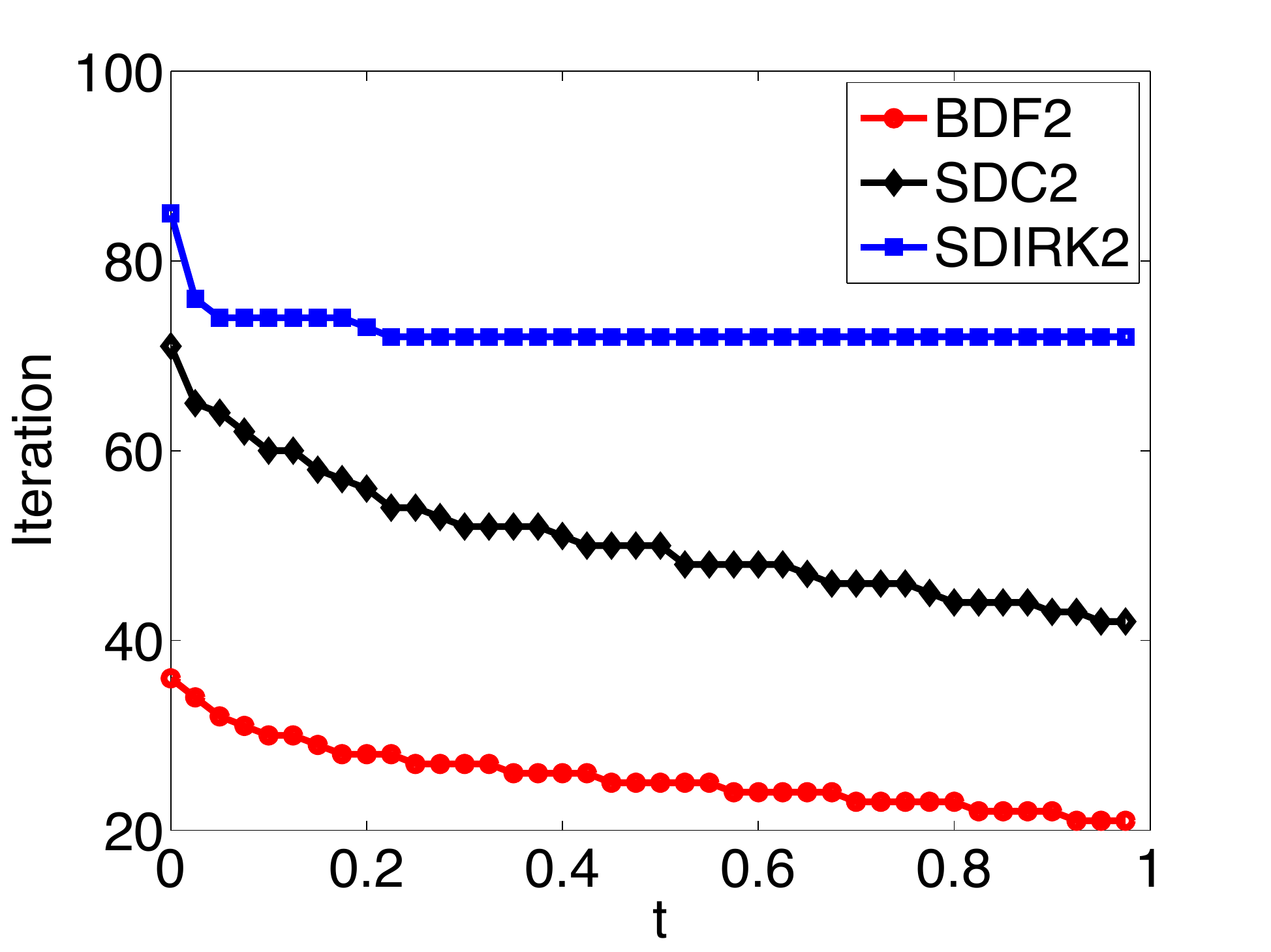}} \\ 
	\subfigure[Energy decay ($3^{\text{rd}}$ order methods)]{\label{fig:energy_thirdorder}\includegraphics[width=.48\textwidth]{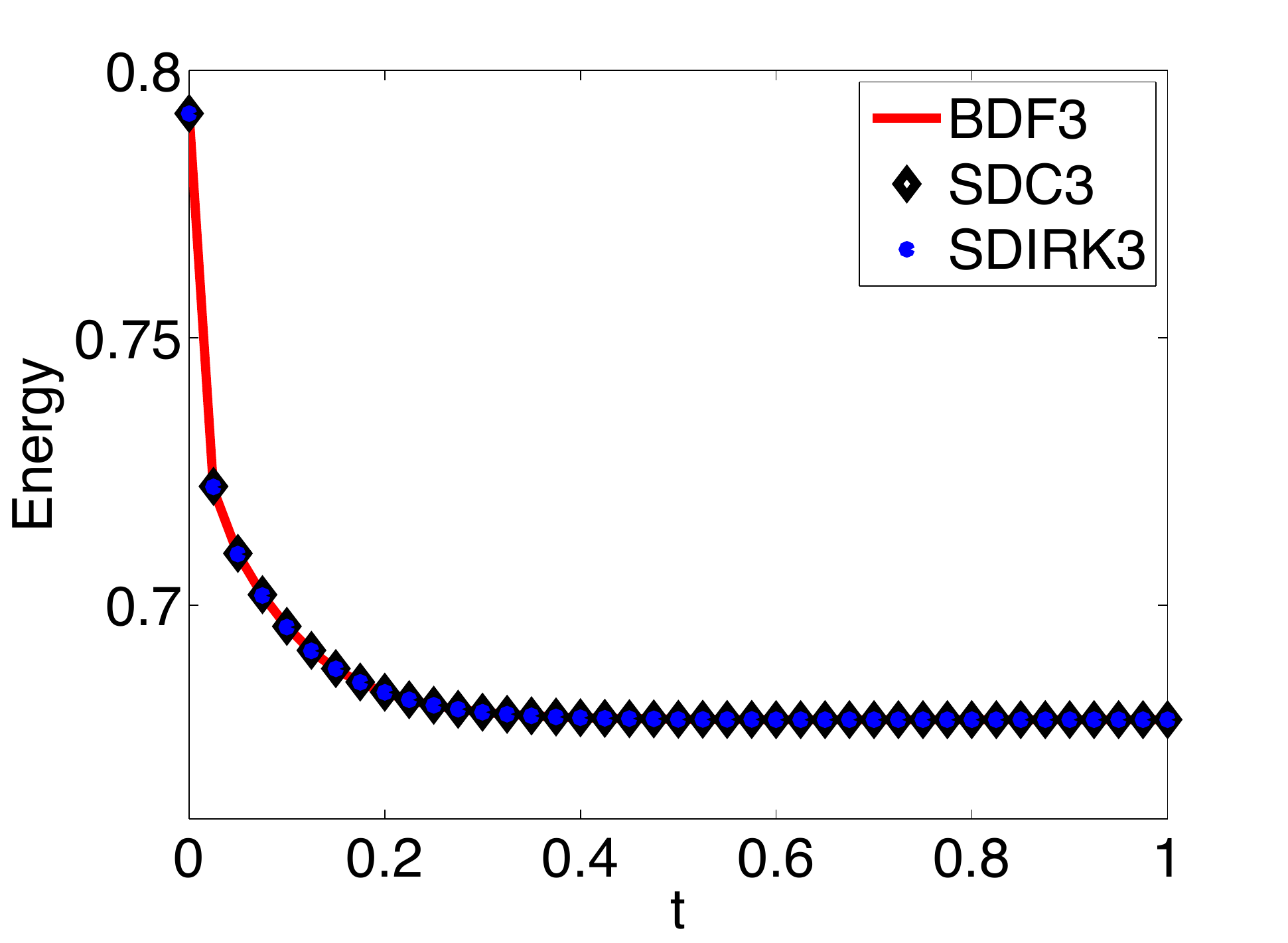}}
	\subfigure[Iteration count ($3^{\text{rd}}$ order methods)]{\label{fig:iteration_number3}\includegraphics[width=.48\textwidth]{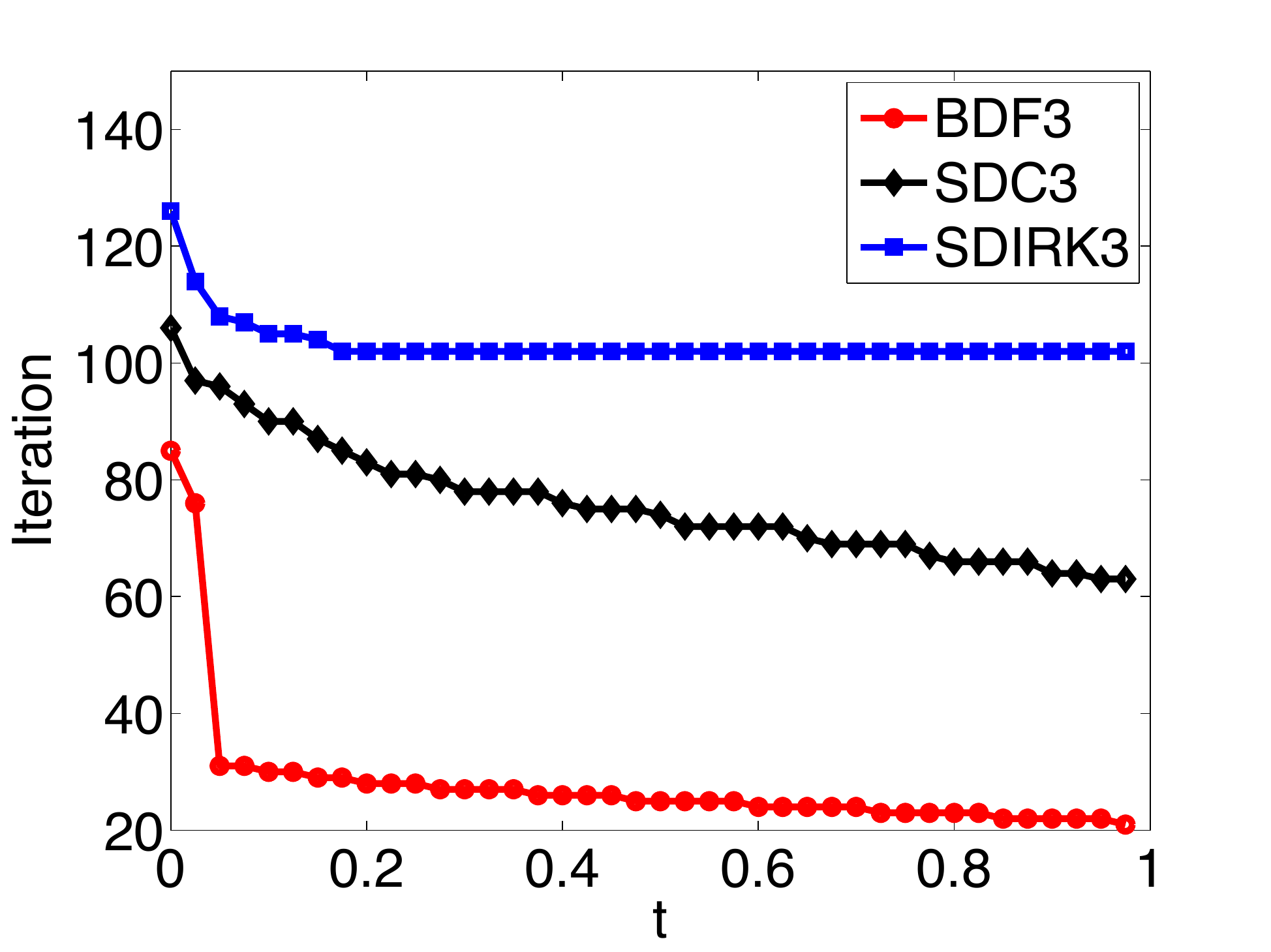}} 
	\caption{Energy descent and nonlinear iteration count history of all second- and third- order methods with fixed time step $\Delta t = 0.025$ and parameters \eqref{eqn:parameters}.}
	\label{fig:second_order_methods}
\end{figure}

\begin{table}[htbp]
\begin{centering}
\caption{Refinement studies of second-order methods for 1D CH equation with periodic BC.}
\begin{tabular}{|c||c|c||c|c||c|c|}
\hline 
		& \multicolumn{2}{c||}{BDF2}	& \multicolumn{2}{c||}{SDIRK2}				& \multicolumn{2}{c|}{SDC2}			\\ \hline
$\Delta t$		& 	$L^{\infty}$ error			& order & 	$L^{\infty}$ error	& order & 	$L^{\infty}$ error	& order	\\ \hline
$0.0500$		& $\num{0.000029454}$	& $-$			& $\num{0.0000047787}$	& $-$	 & $\num{0.000063763}$	& $-$	 		\\ \hline
$0.0250$		& $\num{0.0000072771}$	& $2.0170 $ & $\num{0.0000012050}$	& $1.9876$		& $\num{0.000015412}$	& $2.0487$		\\ \hline
$0.0125$		& $\num{ 0.0000019755}$	& $1.8812$ & $\num{0.00000030081}$	& $2.0021$		& $\num{0.0000037618}$	& $2.0345$		\\ \hline
$0.0063$	& $\num{0.00000053869}$	& $1.8747$	& $\num{0.000000075076}$	& $2.0024$	&$\num{0.00000092786}$	& $2.0194$		\\ \hline
$0.0031$	& $\num{0.00000014352}$	& $1.9082$	& $\num{0.000000018732}$	& $2.0029$	& $\num{0.00000023032}$	& $2.0103$		\\ \hline
\end{tabular}
\label{tab:CH_refinement_second}
\end{centering}.
\end{table}

The factorization of the Helmholtz operators is switched based on the time step size, and is different for each method:
$$ \text{BDF2: } \Delta t = \frac{3}{2} \epsilon^2 = 0.0486, \quad \text{SDIRK2: } \Delta t = \frac{\epsilon^2}{1-\frac{\sqrt{2}}{2}}\approx 0.1106, \quad \text{SDC2: } \Delta t = \epsilon^2 = 0.0324.$$
Table \ref{tab:CH_refinement_second} shows that all methods exhibit second order convergence. Moreover in Figure \ref{fig:energy_secondorder}, we see that the total energy of each method decays during this time evolution. 

In Figure \ref{fig:iteration_number}, we plot the total iteration history of each scheme for the tolerance $N_{\text{tol}} = 10^{-12}$. The wall clock time (seconds) of BDF2, SDC2, SDIRK2 is $0.9725 (s)$, $1.9003 (s)$, and $2.6468 (s)$, respectively, for a final simulation time $T_{\text{final}} = 1$; so the BDF2 method computationally the most efficient. We observe that SDIRK2 is the slowest, as a nonlinear solve is required for both stages $K_1$ and $K_2$, at each time step.

Next, we compare the corresponding third-order schemes for the CH equation, with the same parameters in \eqref{eqn:parameters}, and with the following considerations: \\
\begin{itemize}
\item {\bf [BDF3]} 
	Need two initial steps: SDIRK2 used. $( \text{Switch criterion: } \Delta t_{\text{switch}} = \frac{11}{6} \epsilon^2 )$ \\
\item {\bf [SDIRK3]} 
	Need to compute $K_1, K_2, K_3$ at each update. $( \Delta t_{\text{switch}}= \frac{\epsilon^2}{\eta}, \eta \approx 0.4359)$ \\
\item {\bf [SDC3]} 
	Compute from level [0] to level [2] with the quadrature order $P$ in \eqref{eqn:integral_quadrature}: First, choose $P=2$ at level $[2]$, then weights are $\tilde{q}_0 = \dfrac{5}{12}, \tilde{q}_1 = \dfrac{8}{12}, \tilde{q}_2 = -\dfrac{1}{12}$. $( \Delta t_{\text{switch}} = \epsilon^2)$ \\
\end{itemize}
The energy and iteration history of third order schemes are shown in Figures \ref{fig:energy_thirdorder} and \ref{fig:iteration_number3}. Again we see that the numerical energy decays. The beginning iteration number of BDF3 is larger than that of BDF2, since the initialization requires SDIRK2. But BDF3 is still the most efficient method of the third order schemes. The wall clock time (seconds) of BDF3, SDC3, SDIRK3 is $1.0355 (s)$, $2.8539(s)$, and $3.6613 (s)$, respectively, for a final simulation time $T_{\text{final}} = 1$.

The successive errors of third-order methods are shown in Table \ref{tab:CH_refinement_third}.
\begin{table}[htbp]
\begin{centering}
\caption{Refinement studies of third-order methods for the 1D CH equation with periodic BC.}
\begin{tabular}{|c||c|c||c|c||c|c|}
\hline 
		& \multicolumn{2}{c||}{BDF3}	& \multicolumn{2}{c||}{SDIRK3}				& \multicolumn{2}{c|}{SDC3}			\\ \hline
$\Delta t$		& 	$L^{\infty}$ error			& order & 	$L^{\infty}$ error	& order & 	$L^{\infty}$ error	& order	\\ \hline
$0.0500$		& $\num{0.000016001}$	& $-$			& $\num{0.0000017259}$	& $-$	 & $\num{0.000012669}$	& $-$	 		\\ \hline
$0.0250$		& $\num{0.0000052832}$	& $1.5987 $ & $\num{0.00000036118}$	& $2.2566$		& $\num{0.0000021002}$	& $2.5927$		\\ \hline
$0.0125$		& $\num{ 0.0000011756}$	& $2.1680$ & $\num{0.000000062823}$	& $2.5234$		& $\num{0.00000031123}$	& $2.7545$		\\ \hline
$0.0063$	& $\num{0.00000019314}$	& $2.6057$	& $\num{0.0000000096163}$	& $2.7077$	&$\num{0.000000042615}$	& $2.8686$		\\ \hline
$0.0031$	& $\num{0.000000026506}$	& $2.8652$	& $\num{0.0000000013011}$	& $2.8858$	& $\num{0.0000000054185}$	& $2.9754$		\\ \hline
$0.0016$	& $\num{0.0000000032519}$	& $3.0270$	& $\num{0.00000000017447}$	& $2.8987$	& $\num{0.00000000066924}$	& $3.0173$		\\ \hline
\end{tabular}
\label{tab:CH_refinement_third}
\end{centering}
\end{table}
Each method achieves third order convergence for small $\Delta t$; but the order of convergence begins to plateau for larger time steps. To this end, SDC3 has a slight advantage. Indeed upon raising the polynomial order \eqref{eqn:integral_quadrature} from $P=2$ to $P=3$, which requires one more correction at each time step, we obtain third order accuracy even for large time steps, as shown in Table \ref{tab:CH_refinement_third_2}.

\begin{table}[htbp]
\begin{centering}
\caption{Refinement study of SDC3 ($P=3$) for the 1D CH equation.}
\begin{tabular}{|c||c|c|}
\hline 
		& \multicolumn{2}{c|}{SDC3 ($P=3$)}			\\ \hline
$\Delta t$		& 	$L^{\infty}$ error			& order  	\\ \hline 
$0.0500$		& $\num{0.0000099462}$	& $-$			\\ \hline
$0.0250$		& $\num{0.0000013061}$	& $2.9289 $ 	\\ \hline
$0.0125$		& $\num{ 0.00000014142}$	& $3.2072$ 	\\ \hline
$0.0063$	& $\num{0.000000016155}$	& $3.1299$	         \\ \hline 
$0.0031$	& $\num{0.0000000017851}$	& $3.1779$	\\ \hline
$0.0016$	& $\num{0.00000000022767}$	& $2.9710$	\\ \hline
\end{tabular}
\label{tab:CH_refinement_third_2}
\end{centering}
\end{table}


\section{Multiple spatial dimension}
\label{sec:multiD}

We now extend the 1D solver to multiple spatial dimensions via dimensional splitting \cite{causley2014higher,causley2015method, causley2016method}. We first write the 2D modified Helmholtz operator as
\begin{equation} \label{eqn:2D_Helmholtz}
\mathcal{L} = I - \frac{\Delta}{\alpha^2} = \left( I - \frac{\partial_{xx}}{\alpha^2} \right) \left( I - \frac{\partial_{yy}}{\alpha^2} \right) - \frac{\partial_{xx}\partial_{yy}}{\alpha^4}  \equiv \mathcal{L}_x \mathcal{L}_y - \frac{\partial_{xx}\partial_{yy}}{\alpha^4},
\end{equation}
where $\mathcal{L}_x$ and $\mathcal{L}_y$ are univariate operators, and so the fourth order term represents a splitting error.
Now $\mathcal{L}_x^{-1}$ is applied for fixed $y$, and vice versa for $\mathcal{L}_y^{-1}$ in a line-by-line fashion, similar to alternating direction implicit (ADI) type method \cite{douglas1955numerical,fairweather1967new}. Our key observation is that if we include the splitting error term from equation \eqref{eqn:2D_Helmholtz} in the fixed point iteration, then we can simultaneously solve the nonlinear problem, and correct the splitting error. More detailed analysis of the 2D CH equation \eqref{eqn:CahnHilliard} and 2D CH vector equations \eqref{eqn:CH_vector_reaction} will be presented below.

\subsection{Semi-discrete solution to 2D CH equation}\label{sec:CH_2D}
We now construct the semi-discrete scheme for the 2D CH equation, using the Backward Euler method. The analogous higher order schemes follow accordingly. Starting from the 1D scheme \eqref{eqn:CH_BE_fixed_complete}, we replace $\partial_{xx}$ with the Laplacian operator $\Delta = \partial_{xx} + \partial_{yy}$. If $\Delta t < \epsilon^2$, we  complete the square of the operator, and have
\begin{equation} \label{eqn:CH_BE_2D} 
 \left( I - \frac{\Delta}{\alpha^2} \right)^{2} [v^{n+1,k+1}]= v^n + \Delta t \Delta \left( f^{n+1,k} -2\sqrt{\frac{\epsilon^2}{\Delta t}} v^{n+1,k}\right), 
\end{equation}
where $\frac{1}{\alpha^2} = \sqrt{\epsilon^2 \Delta t}$. Plugging the identity in \eqref{eqn:2D_Helmholtz}, by lagging the mixed derivative term along with the nonlinear term,
\begin{equation} \label{eqn:CH_BE_2D_2} 
\left( \mathcal{L}_x  \mathcal{L}_y \right)^2 [v^{n+1,k+1}]= v^n + \Delta t \Delta \left( f^{n+1,k} -2\sqrt{\frac{\epsilon^2}{\Delta t}} v^{n+1,k}\right) + \left( 2\frac{\partial_{xx}\partial_{yy}}{\alpha^4}  - \left(\frac{\partial_{xx}\partial_{yy}}{\alpha^4} \right)^2 \right) v^{n+1,k}
\end{equation} 
Note that Laplacian operator and mixed derivative can be replaced as follows:
$$\Delta = \alpha^2 \left(I -  \mathcal{L}_x \mathcal{L}_y + \frac{\partial_{xx}\partial_{yy}}{\alpha^4}\right), \quad \frac{\partial_{xx}\partial_{yy}}{\alpha^4} = (\mathcal{L}_x - I)(\mathcal{L}_y - I). $$
Now we formally invert both operators to the right hand side of \eqref{eqn:CH_BE_2D_2}, 
\begin{equation} \label{eqn:CH_BE_2D_3} 
v^{n+1,k+1}= \left(  \mathcal{L}_x \mathcal{L}_y  \right)^{-2} \left[v^n \right]+ \left(  \mathcal{L}_x \mathcal{L}_y  \right)^{-1}   \mathcal{C}_1 [ \alpha^2 \Delta t ( f^{n+1,k} -2\sqrt{\frac{\epsilon^2}{\Delta t}} v^{n+1,k})] + \mathcal{C}_2 \left[v^{n+1,k} \right]
\end{equation} 
where 
\begin{align} \label{eqn:CH_2D_Doperator1} 
 \mathcal{C}_1 &=\left(  \mathcal{L}_x \mathcal{L}_y  \right)^{-1} -I +\mathcal{D}_x \mathcal{D}_y, \quad \mathcal{D}_\gamma = I - \mathcal{L}_{\gamma}^{-1}, (\gamma = \{x,y\}) \\ \label{eqn:CH_2D_Doperator2} 
  \mathcal{C}_2 &= 2 \mathcal{D}_x \mathcal{D}_y - \left ( \mathcal{D}_x \mathcal{D}_y \right)^2.
\end{align}
As shown, every mixed-derivative splitting error term can be controlled by applying $\mathcal{D}_{\gamma}$, ($\gamma = \{x,y\}$) operators, which can be constructed in a line-by-line fashion. We emphasize that this allows us to remove splitting error, which is $O(\frac{1}{\alpha^4}) = O(\Delta t)$. Similar treatment of higher order BDF, SDIRK and SDC methods shows that the corresponding splitting errors can also be removed in this manner. Also, the fully discrete scheme follows from line-by-line spatial discretization, as outlined in Section \ref{sec:molt_CH_fully}.\\

We now consider the standard benchmark initial states in the 2D setting \cite{Jaylan, willoughby2011high}, to confirm the temporal order of accuracy. The initial condition is
\begin{equation}\label{eqn:CH2D_initial}
u_0(x,y) = 2 e^{(\sin(x)+\sin(y)-2)} + 2.2 e^{(-\sin(x)-\sin(y)-2)} -1, \qquad (x,y) \in [0,2\pi]^2,
\end{equation}
with the periodic boundary conditions, and the following parameters
\begin{equation} \label{eqn:parameters_2D}
 \epsilon = 0.18, \quad \Delta x = \dfrac{2\pi}{128} \approx 0.0491, \quad  T_{\text{final}} = 1 \hspace{2mm} (0 \leq t \leq T_{\text{final}}), \quad N_{\text{tol}} = 10^{-6}.
\end{equation}
We use a $4^{\text{th}}-$order spatial quadrature to ensure that temporal error is dominant.

The temporal refinement study of each second-order scheme for the 2D CH equation is shown in Table \ref{tab:CH_refinement_second_2D}.  We observe the expected second-order of convergence for all three methods. 
\begin{table}[htbp]
\begin{centering}
\caption{Refinement studies of second-order methods for the 2D CH equation.}
\begin{tabular}{|c||c|c||c|c||c|c|}
\hline 
		& \multicolumn{2}{c||}{BDF2}	& \multicolumn{2}{c||}{SDIRK2}				& \multicolumn{2}{c|}{SDC2}			\\ \hline
$\Delta t$		& 	$L^{\infty}$ error			& order & 	$L^{\infty}$ error	& order & 	$L^{\infty}$ error	& order	\\ \hline
$0.1000$		& $\num{0.0037891}$	& $-$			& $\num{0.0010250}$	& $-$	 & $\num{0.0023334}$	& $-$	 		\\ \hline
$0.0500$		& $\num{0.00082626}$	& $2.1972 $ & $\num{0.00026950}$	& $1.9272$		& $\num{0.00050570}$	& $2.2061$		\\ \hline
$0.0250$		& $\num{ 0.00020319}$	& $2.0238$ & $\num{0.000069796}$	& $1.9491$		& $\num{0.00011000}$	& $2.2007$		\\ \hline
$0.0125$	& $\num{0.000046909}$	& $2.1149$	& $\num{0.000017809}$	& $1.9706$	&$\num{0.000028312}$	& $1.9581$		\\ \hline
\end{tabular}
\label{tab:CH_refinement_second_2D}
\end{centering}
\end{table}
\begin{figure}[h!] \label{2D_CH_energy}
	\centering
	\subfigure[Energy decay ($2^{\text{nd}}$ order methods)]{\label{fig:energy_secondorder_2D}\includegraphics[width=.48\textwidth]{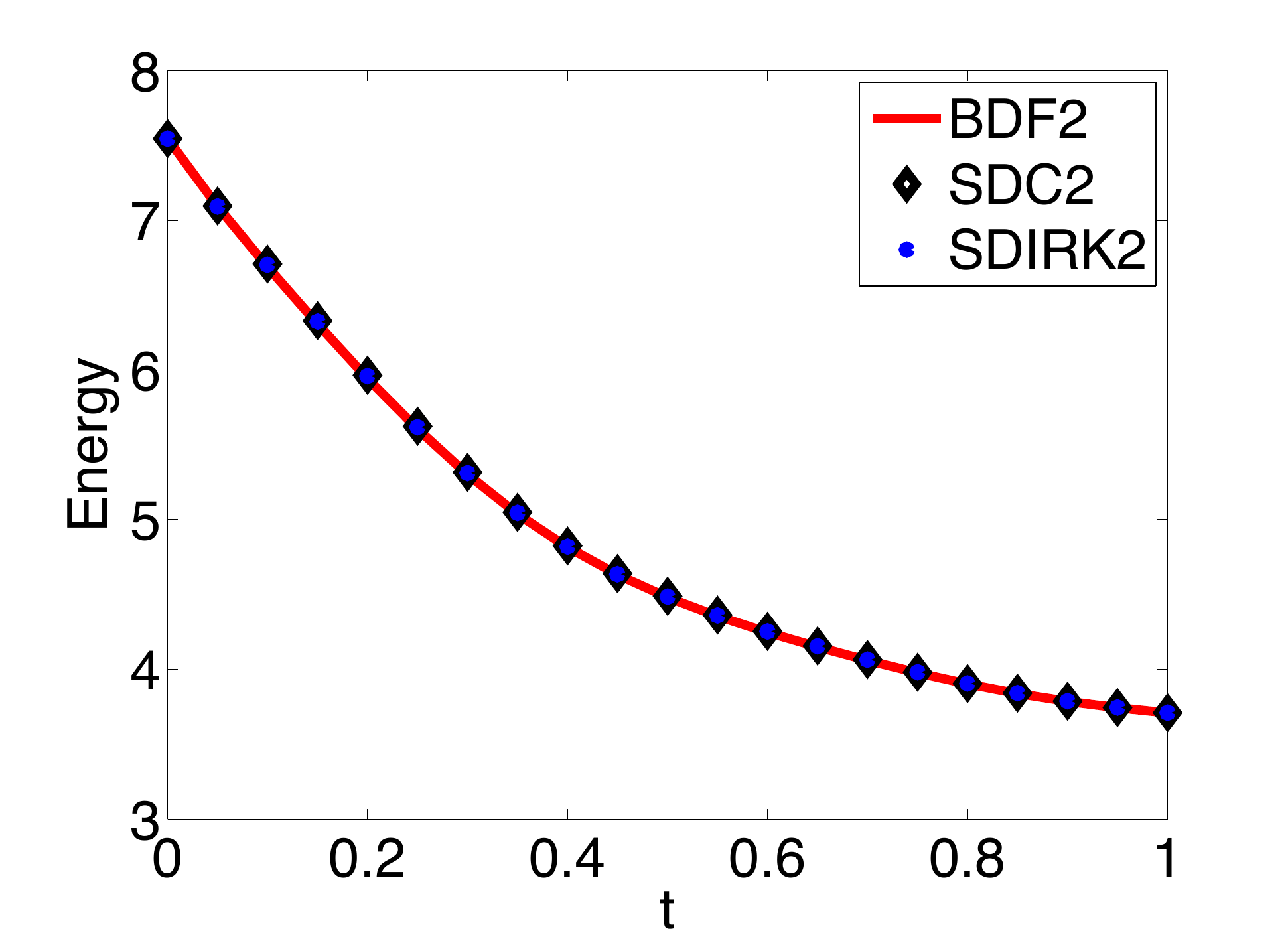}}
	\subfigure[Iteration count($2^{\text{nd}}$ order methods)]{\label{fig:iteration_number_2D}\includegraphics[width=.48\textwidth]{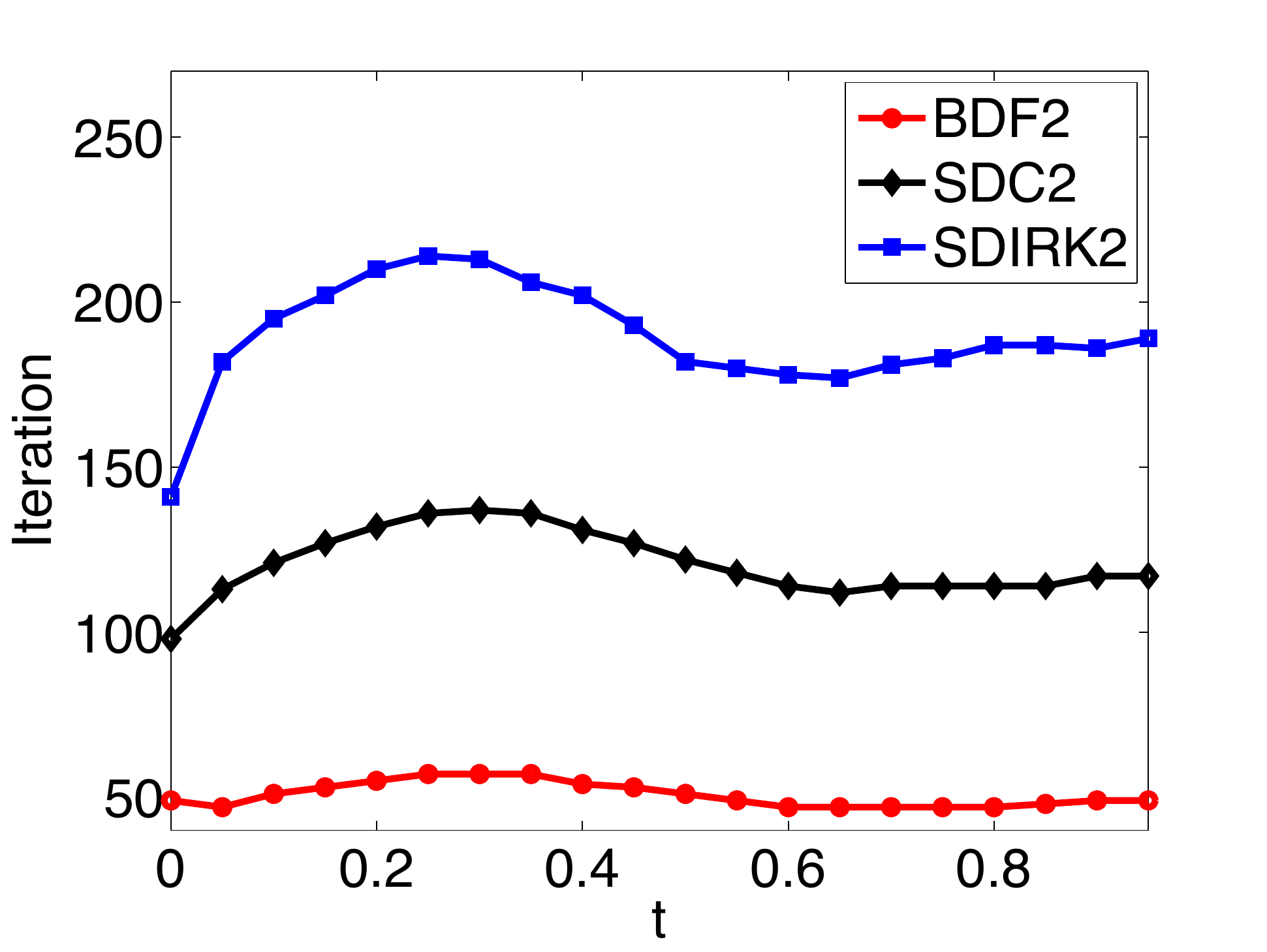}}
	\caption{Energy descent and nonlinear iteration count history of all second-order 2D methods with fixed time step $\Delta t = 0.05$ and with parameters \eqref{eqn:parameters_2D}}
	\label{fig:second_order_methods_2D}
\end{figure}

We also confirm the total energy, shown in Figure \ref{fig:energy_secondorder_2D}, is non-increasing during time evolution. Based on the iteration count, we see that BDF2 is the most efficient scheme, as was the case in 1D.

In principle, our method extends to even higher orders of accuracy in both time and space, as well as 3 spatial dimensions, since the splitting error technique in \eqref{eqn:CH_2D_Doperator1} can remove additional splitting errors.

\subsection{Semi-discrete solution to 2D vector CH equation} \label{sec:CH_vector_2D}
We now extend our algorithm to vector CH (VCH) model \eqref{eqn:CH_vector}. This model consists of two coupled variables $u_1$ and $u_2$ with local dynamic $\nabla_{\bf u} W$, comprised of partial derivatives of $6^{\text{th}}-$ order polynomials $W(u_1,u_2)$ defined in \eqref{eqn:CH_vector_reaction}.

Before employing the MOL$^\text{T}$ formulation to this system, we translate the vector $(u_1,u_2)$ about its equilibrium state, which is based on the functional $W(u_1,u_2)$, 
$$ (u_1^{\ast},u_2^{\ast}) = \left\{  \left(\cos(\theta_i), \sin(\theta_i)\right)   \,\middle| \, \theta_1 = 0, \theta_2 = \frac{2\pi}{3}, \theta = -\frac{2\pi}{3}  \right\} \equiv \left\{ (1,0), (-\frac{1}{2}, \frac{\sqrt{3}}{2}), (-\frac{1}{2}, -\frac{\sqrt{3}}{2}) \right\}$$
which are cube roots of unity in $(u_1,u_2)$ plane. A straightforward calculation yields the Jacobian of the potential $\nabla_{\bf u} W$ evaluated at these points
\begin{equation} \label{eqn:Jacobian}
 \mathcal{J}_{\nabla_{\bf u}W} (u_1^{\ast},u_2^{\ast})
 \equiv \begin{bmatrix}
       \frac{\partial^2W}{\partial u_1^2}&  \frac{\partial^2W}{\partial u_1u_2}         \\[0.3em]
       \frac{\partial^2W}{\partial u_2u_1} & \frac{\partial^2W}{\partial u_2^2} 
     \end{bmatrix}_{(u_1^{\ast},u_2^{\ast})}
= \begin{bmatrix}
      18 & 0     \\[0.3em]
       0  & 18
     \end{bmatrix},
\end{equation}
hence all points are stable  equilibrium solutions \cite{willoughby2011high}. We subtract ${\bf{u}}=(u_1,u_2)$ of the background state ${\bf{z}}_3 \equiv (-\frac{1}{2}, -\frac{\sqrt{3}}{2}) $, and introduce the new transformed vector
${\bf{v}}=(v_1,v_2) = \bf{u} - \bf{z}_3$ into the original system \eqref{eqn:CH_vector},
\begin{equation} \label{eqn:CH_vector_linearize}
 {\bf{v}}_t = -\epsilon^2 \Delta^2  {\bf{v}}
+ \Delta \nabla_{\bf v} W \left(\bf{v} + \bf{z}_3 \right)
\equiv -\epsilon^2 \Delta^2  {\bf{v}} + \Delta \left( \nabla_{\bf v} \tilde{W} ({\bf{v}})+18 {\bf{v}} \right)
\end{equation}
where $\tilde{W}_{v_1}({\bf{v}}) := W_{v_1}({\bf{v}} + {\bf{z}_3}) - 18v_1$ and  $\tilde{W}_{v_2}({\bf{v}}) := W_{v_2}({\bf{v}} + {\bf{z}_3}) - 18v_2$. 

The backward Euler scheme applied to the transformed system \eqref{eqn:CH_vector_linearize} results in
\begin{align} \label{eqn:CH_vector_BE}
\left( I - 18 \Delta t \Delta + \epsilon^2 \Delta t \Delta^2 \right) {\bf v}^{n+1,k+1} = {\bf v}^n + \Delta \nabla_{\bf v} \tilde{W}^{n+1,k}
\end{align}
where $k$ denotes the iteration index. We again introduce two factorizations of the left-hand side operator:
\begin{align}
	\label{eqn:CH_vector_fixed1}
	\Delta t \geq \frac{\epsilon^2}{81}: \quad & I - 18 \Delta t \Delta + \epsilon^2 \Delta t \Delta^2 = \left(I - \frac{\Delta}{\alpha_1^2} \right) \left(I - \frac{\Delta}{\alpha_2^2} \right), \\
	\label{eqn:CH_vector_fixed2}
	\Delta t < \frac{\epsilon^2}{81}: \quad & I - 2 \sqrt{\epsilon^2 \Delta t}\Delta + \epsilon^2 \Delta t \Delta^2 = \left(I - \frac{\Delta}{\alpha^2} \right)^2,
\end{align}
where
\[
	\frac{1}{\alpha_i^2} = 9\Delta t \pm \sqrt{81\Delta t^2 - \epsilon^2 \Delta t}, \qquad \frac{1}{\alpha^2} = \sqrt{\epsilon^2 \Delta t}.
\]
Inversion of the above operators follows from the same strategy as in Section \ref{sec:CH_2D}. As expected, the stabilized fixed point iteration \eqref{eqn:CH_vector_fixed1} permits larger time steps without limiting the foregoing energy stability property \eqref{eqn:EnergyFunctional_vector}. 

\section{Adaptive time stepping}
\label{sec:adaptiveTime} 
For phase-field models, adaptive time stepping is a crucial feature for an efficient and accurate numerical solution. For instance, the solution of CH equation \eqref{eqn:CahnHilliard} evolves on various time scales. During \textit{spinodal} evolution, transition layers are developed in $O(1)$ time. Subsequently, they slowly evolve and merge on a longer time scale, $O(e^{C/\epsilon})$ for 1D, which is called \textit{ripening} process. Simulating these phenomena with a fixed time stepping necessarily becomes inefficient, and so we must incorporate adaptive time into our above MOL$^{\text{T}}$ schemes.

The adaptive time step size control is based on the Local Truncation Error(LTE) $\eta_e$ at each time level $t=t^n$. The LTE can be approximated by $\eta_e \approx \eta = ||u^{\ast} - u^n||_\infty$, where $u^{\ast}$ is an explicit predictor solution, typically using the Forward Euler (FE) or Adams Bashforth (AB) schemes \cite{Jaylan}. On the other hand, Richardson extrapolation (known as step-doubling) or embedded Runge-Kutta pairs can also be used \cite{Ong}. Below we will compare LTE estimates based on Richardson extrapolation, as well BDF2 and SDC2 predictors. The time step size selection criteria presented in \cite{Jaylan} will be used, which is summarized in algorithm \ref{algorithm_adaptive} below.
\begin{algorithm}[htp]
	\caption{Adaptive time step-size control}
	\label{algorithm_adaptive}
	\begin{enumerate}[itemsep=2mm]
		
		\item Starting at $t=t^n$, approximate the local truncation error. For Richardson extrapolation, the solution is estimated twice: once with a full step $\Delta t$ (denoted by $u^{n+1}_{\Delta t}$), and again with two half steps ($u^{n+1}_{\Delta t/2}$). The difference between the two numerical approximations gives an estimate for the LTE of $u^{n+1}$
		\[
		\text{(Richardson extrapolation)} \quad \eta_e \approx \eta:= \frac{1}{2} || u^{n+1}_{\Delta t} - u^{n+1}_{\Delta t/2} ||_{\infty}.
		\]

		\item Define a tolerance $\sigma_{\text{tol}}$ for the above LTE. If the accuracy criterion is not met ($ \eta > \sigma_{\text{tol}}$), then the time step is reduced. If the desired accuracy is achieved, then we test for the following criteria,
\begin{enumerate}[label=(\Roman*)]
\item $ \eta \leq \sigma_{\text{tol}} : $ 

			\begin{align} \notag
				\frac{N_{\text{it}}}{N_{\max\text{it}}} < 0.7: &\quad \Delta t^{n+1} = \Delta t^n \cdot \min\left(\theta\sqrt{\frac{ \sigma_{\text{tol}}}{\eta}}, \gamma \right), \quad \theta = 0.8, \quad \gamma = 1.3 >1  \\ \notag
				0.7 \leq \frac{N_{\text{it}}}{N_{\max\text{it}}} <1 : &\quad  \Delta t^{n+1} = \Delta t^n \cdot \min\left(\theta\sqrt{\frac{ \sigma_{\text{tol}}}{\eta}}, 1 \right), \\ \notag
				\frac{N_{\text{it}}}{N_{\max\text{it}}} \geq 1: &\quad \text{Step fails.}  \quad \text{Reduce time step } \Delta t^{n} = \Delta t^n\cdot  \frac{1}{\gamma}
			\end{align} 

\hspace{3mm}

\item $ \eta > \sigma_{\text{tol}} : \quad \text{Step fails.}  \quad \text{Reduce time step}$ 
\begin{align} \notag
				\frac{\eta}{\sigma_{\text{tol}}} > 2: &\quad \Delta t^{n} = \Delta t^n \cdot  \frac{1}{\gamma},  \\ \notag
				\frac{\eta}{\sigma_{\text{tol}}} \leq 2: &\quad \Delta t^{n} = \Delta t^n \cdot \theta \sqrt{\frac{ \sigma_{\text{tol}}}{\eta}}
\end{align}
where $\gamma$ and $\theta$ are the same safety factors defined in \cite{Jaylan}.
\end{enumerate}
		
	\end{enumerate}
\end{algorithm}

In practice, this procedure leads to small time steps during spinodal evolution, or at the ripening event, to maintain a consistent LTE. 
On the other hand, during slow coarsening (metastable states), small time steps are unnecessary and so $\Delta t$ is increased within the upper bound for fixed-point iteration count.

\section{Numerical Results}
\label{sec:numerical} 
In this section we present adaptive time stepping results using the previously developed MOL$^\text{T}$ schemes.

\subsection{1D CH model} \label{sec:numerical_1D} 
We first solve the 1D CH equation \eqref{eqn:CahnHilliard} with a stiff initial condition \eqref{eqn:CH1D_initial} $(\epsilon=0.18)$. 
The second perturbation term of this initial state creates two intervals, $u = -1$ and $u = +1$, which are asymmetric, so that a finite number of transition layers are formed during spinodal evolution. After a long ripening process, such layers are eventually absorbed into one region, at the so-called ripening time \cite{Jaylan}. The aim of this simulation is to accurately capture all time scales using both fixed and adaptive time stepping strategies. (The spatial mesh size is fixed at $\Delta x = \frac{2\pi}{128} \approx 0.05$.)

In the first experiment, we implement our various time stepping schemes, with small fixed time step ($\Delta t = 0.01$). The ripening time $T_{\text{r}}$ is defined as that for which the midpoint value $u(\pi,t)$ changes from positive to negative. The fixed point iteration has residual tolerance $10^{-11}$ at each step, and the ripening times are presented in Table \ref{tab:1DCH_ripening_smalldt}. Our results agree well with the reference time $T_{\text{r}}=8318.63$ in \cite{Jaylan}.
\begin{table}[htbp]
\begin{centering}
\caption{Ripening time of 1D CH equation with small fixed time $(\Delta t=0.01)$}
 \begin{tabular}{rcc}
\hline 
		& $\qquad$     Ripening time	 \\ \hline 
BE$\quad$	& $\qquad$ $8317.81$	\\	
BDF2	& $\qquad$ $8318.70$	 \\	
BDF3	& $\qquad$ $8318.74$	  \\	
SDC2	& $\qquad$ $8318.99$	   \\	 
SDC3	& $\qquad$ $8318.84$            \\ \hline
\end{tabular}
\label{tab:1DCH_ripening_smalldt}
\end{centering}
\end{table}

We also compare the ripening time using several schemes, with larger fixed time steps in Table \ref{tab:1DCH_ripeninglargedt}. 
\begin{table}[htbp]
\begin{centering}
\caption{Ripening time of 1D CH equation with larger fixed time step sizes.}
 \begin{tabular}{rccc}
\hline 
		& $\quad$ time step  & $\qquad$ Ripening time & $\qquad$ Times$(s)$	 \\ \hline 
		& $10$ & $\qquad$ $8250.00$	 & $\qquad$ $313.38$ \\	
BE$\quad$& $1$ & $\qquad$ $8296.00$	& $\qquad$ $351.52$ \\	
		& $0.05$ & $\qquad$ $8311.85$	& $\qquad$ $1312.16$  \\ \hline 
		& $10$ & $\qquad$ $9050.00$	 & $\qquad$ $717.26$ \\	
SDC2	& $1$ & $\qquad$ $8582.00$	& $\qquad$ $692.65$ \\	
		& $0.05$ & $\qquad$ $8319.80$	& $\qquad$ $2779.61$ \\ \hline
		& $10$ & $\qquad$ $8290.00$	 & $\qquad$ $231.40$ \\	
BDF2	& $1$ & $\qquad$ $8303.00$	& $\qquad$ $281.56$ \\	
		& $0.05$ & $\qquad$ $8315.75$	& $\qquad$ $1288.68$ \\ \hline
\end{tabular}
\label{tab:1DCH_ripeninglargedt}
\end{centering}
\end{table}
Among the three methods, BDF2 is the most efficient, and provides better estimates of the true ripening time, even for larger $\Delta t$. In the most extreme instance of $\Delta t = 10$, we note that the first-order scheme (BE) predicts ripening too soon, and that SDC2 is too late; but BDF2 is still fairly accurate. However, to capture the ripening moment accurately, we still require small fixed time steps ($\Delta t \leq 0.05$), which is too expensive for long time simulations.

\begin{table}[htbp]
\begin{centering}
\caption{Ripening time of 1D CH equation with adaptive time stpe size.}
 \begin{tabular}{rccc}
\hline 
			& $\quad$ $\delta_{\text{tol}}$ $\quad$  & $\qquad$ Ripening time & $\qquad$ Times$(s)$	 \\ \hline 
			& $10^{-3}$ & $\qquad$ $8292.54$	 & $\qquad$ $224.03$ \\	
BE$\qquad$	& $10^{-4}$ & $\qquad$ $8276.08$	& $\qquad$ $226.06$ \\	
(Richardson)	& $10^{-5}$ & $\qquad$ $8308.23$	& $\qquad$ $235.44$  \\ \hline 
			& $10^{-3}$ & $\qquad$ $8311.08$	 & $\qquad$ $233.38$ \\	
BE-SDC2		& $10^{-4}$ & $\qquad$ $8311.47$	& $\qquad$ $239.58$ \\	
			& $10^{-5}$ & $\qquad$ $8312.91$	& $\qquad$ $226.43$ \\ \hline
			& $10^{-3}$ & $\qquad$ $8320.03$	 & $\qquad$ $184.71$ \\	
BE-BDF2		& $10^{-4}$ & $\qquad$ $8319.91$	& $\qquad$ $196.82$ \\	
			& $10^{-5}$ & $\qquad$ $8317.87$	& $\qquad$ $208.18$ \\ \hline
\end{tabular}
\label{tab:1DCH_ripening_adaptive}
\end{centering}
\end{table}

Thus, we consider adaptive time stepping for the same problem. As indicated in Section \ref{sec:adaptiveTime}, the local truncation error (LTE) can be approximated with Richardson extrapolation, or a higher order solver such as BDF2 or SDC2.

We implement three methods with the same fixed point residual tolerance $10^{-11}$, $N_{\max\text{it}} = 600$ in Algorithm \ref{algorithm_adaptive}, but with various error tolerance $\delta_{\text{tol}}$. The performance of the time-adaptive scheme is shown in Table \ref{tab:1DCH_ripening_adaptive}. We see that time adaptivity is superior to fixed time stepping both in terms of accuracy and time to solution. Additionally, the BDF2 method is the most efficient predictor.
\begin{figure}[h]
\begin{center}
\centering
    \subfigure[Time step size]{\label{fig:CH1D_time}\includegraphics[width = 0.325\linewidth]{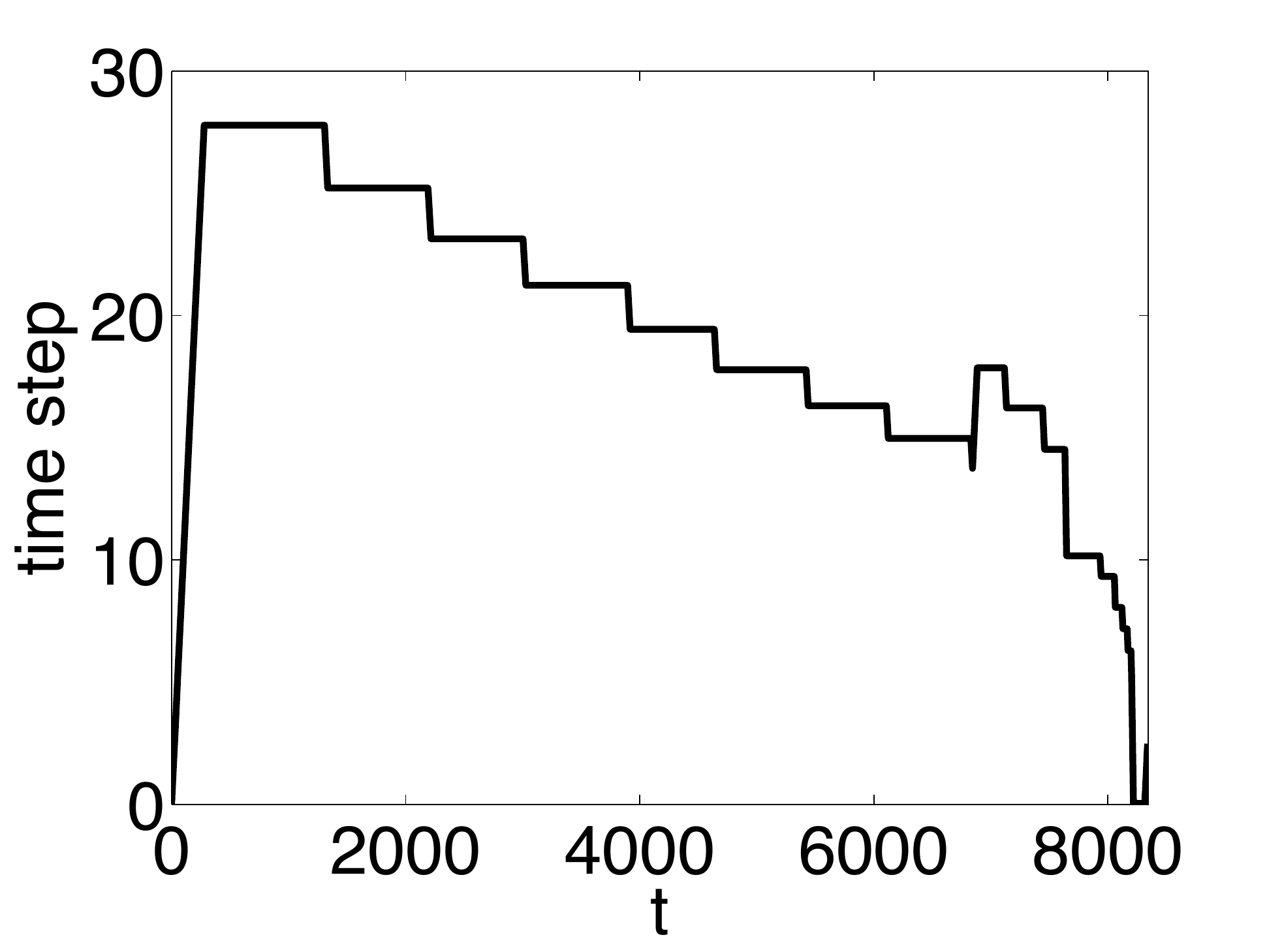}}
    \subfigure[Iteration count]{\label{fig:CH1D_iteration}\includegraphics[width = 0.325\linewidth]{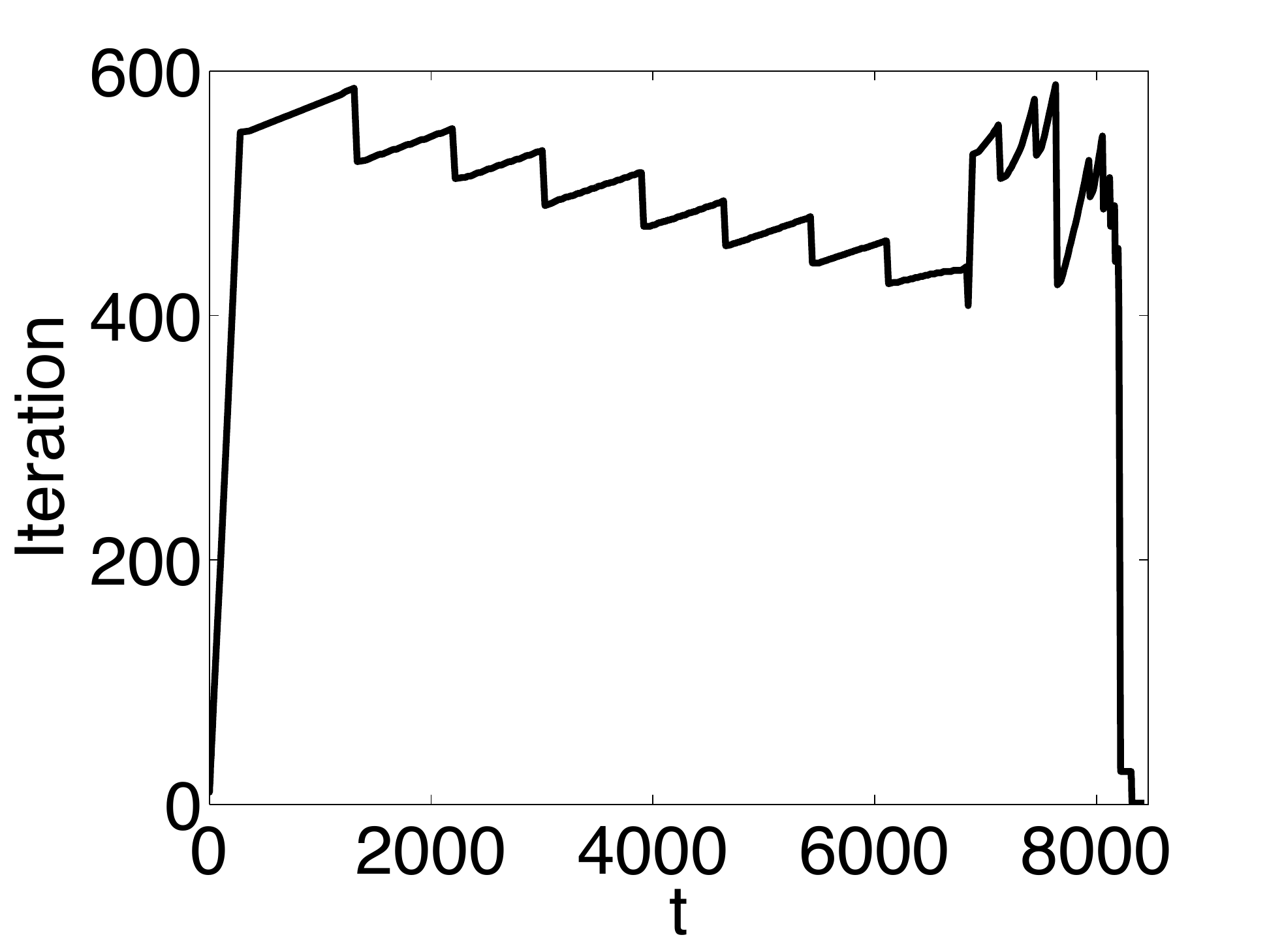}}
    \subfigure[Energy]{\label{fig:CH1D_energy}\includegraphics[width = 0.325\linewidth]{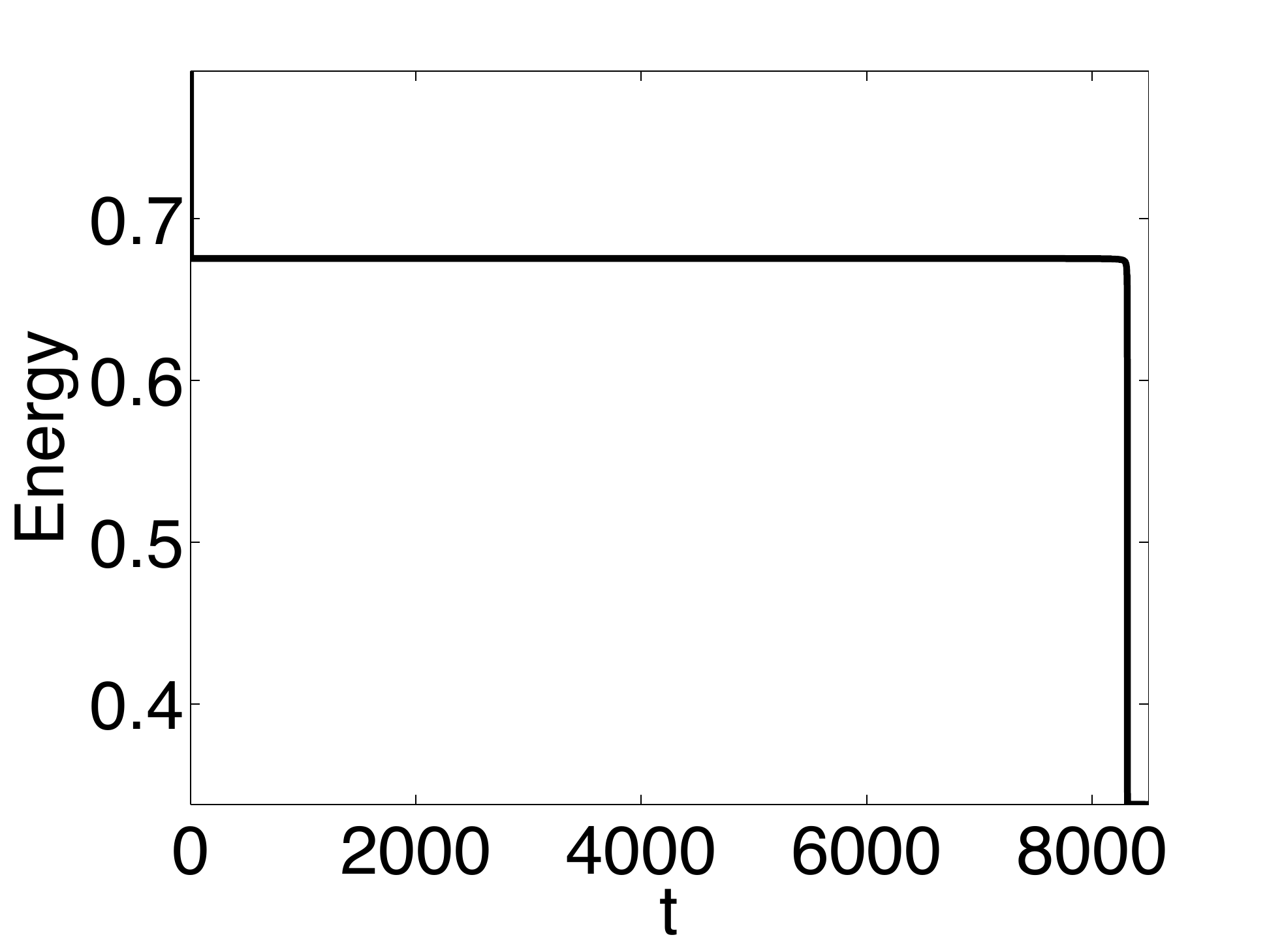}}
    \caption{Adaptive time stepping with the BE-BDF2 strategy.}
    \label{fig:CH_1d_history}
\end{center}
\end{figure}

\begin{figure}[h!]
	\begin{center}
		\centering
		\subfigure[$u(x,0)$]{\label{fig:CH1D_u0}\includegraphics[width = 0.3\textwidth]{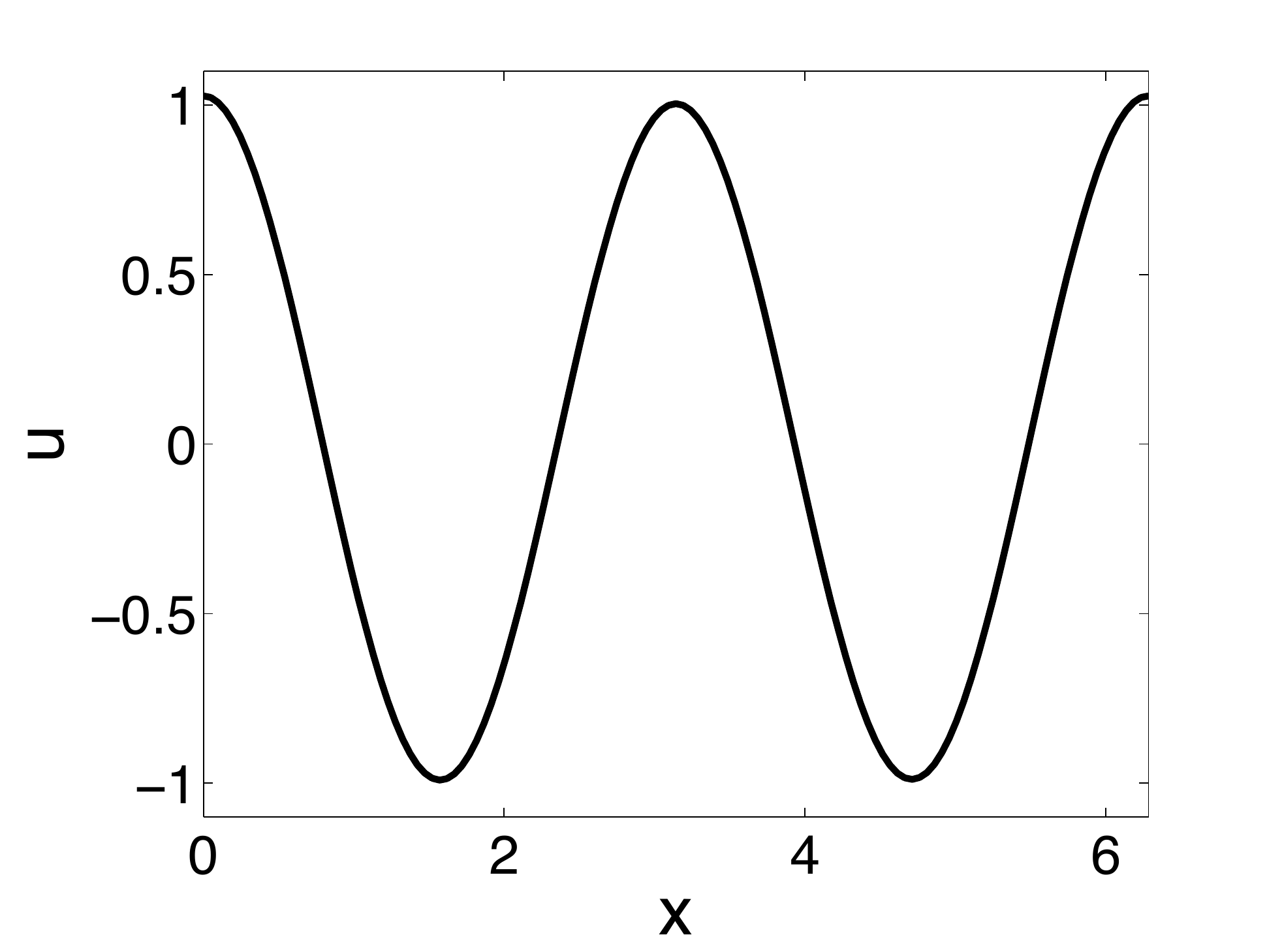}}
		\subfigure[$u(x,0.51)$]{\label{fig:CH1D_u1}\includegraphics[width = 0.3\textwidth]{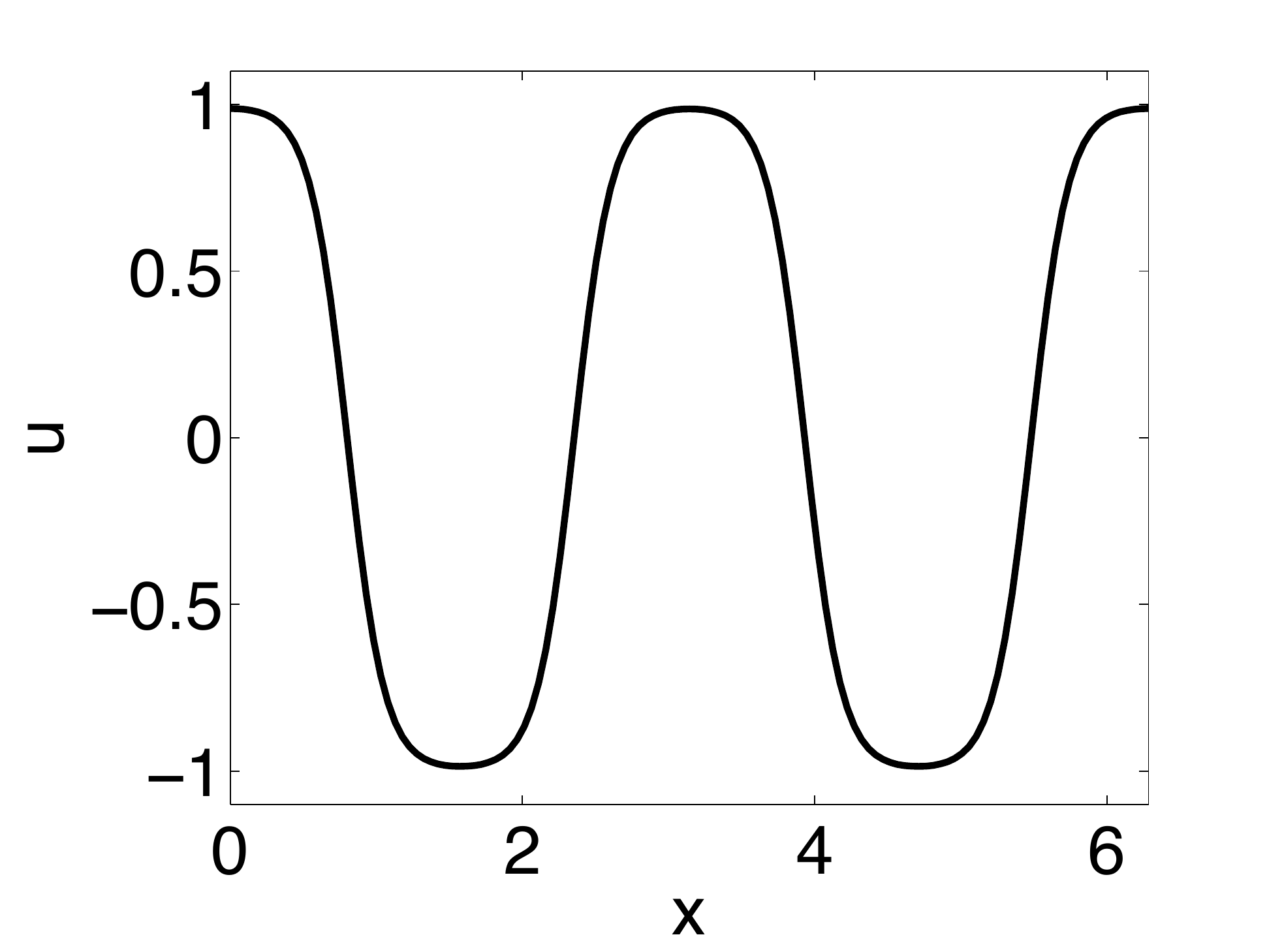}}
		\subfigure[$u(x,3669.8)$]{\label{fig:CH1D_u2}\includegraphics[width = 0.3\textwidth]{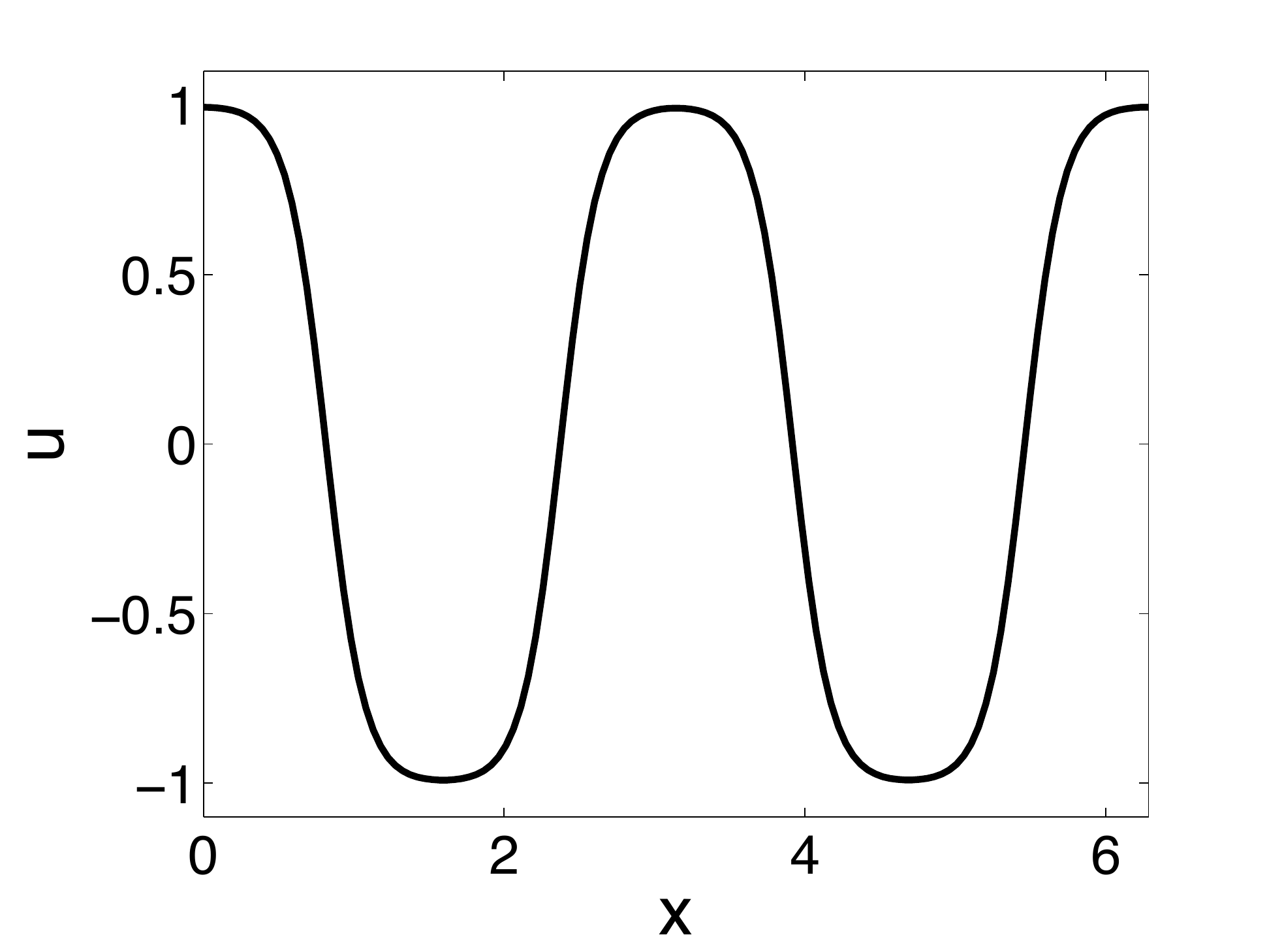}} \\
		\subfigure[$u(x,7005.7)$]{\label{fig:CH1D_u3}\includegraphics[width = 0.3\textwidth]{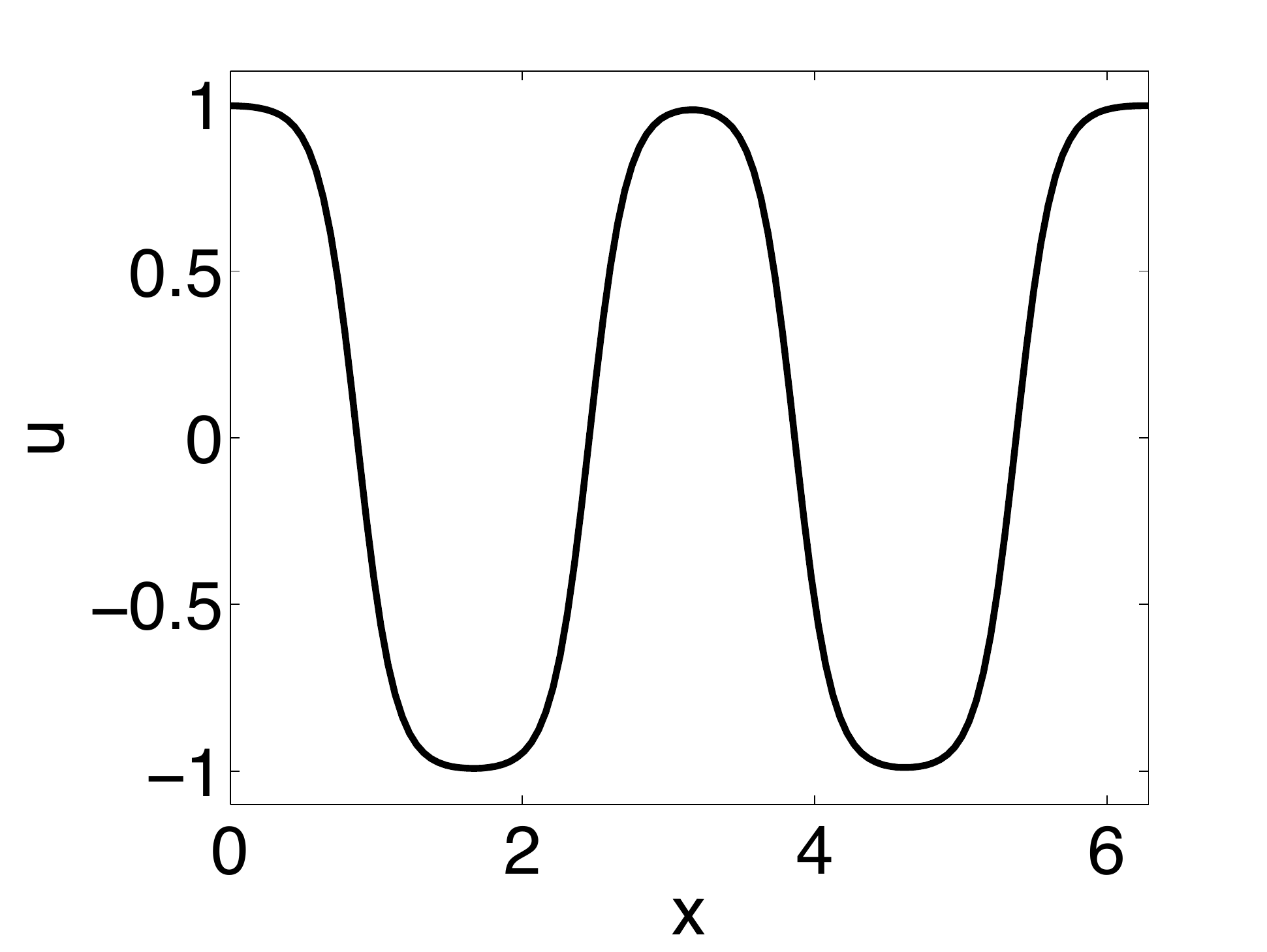}}
		\subfigure[$u(x,8317.9)$]{\label{fig:CH1D_u4}\includegraphics[width = 0.3\textwidth]{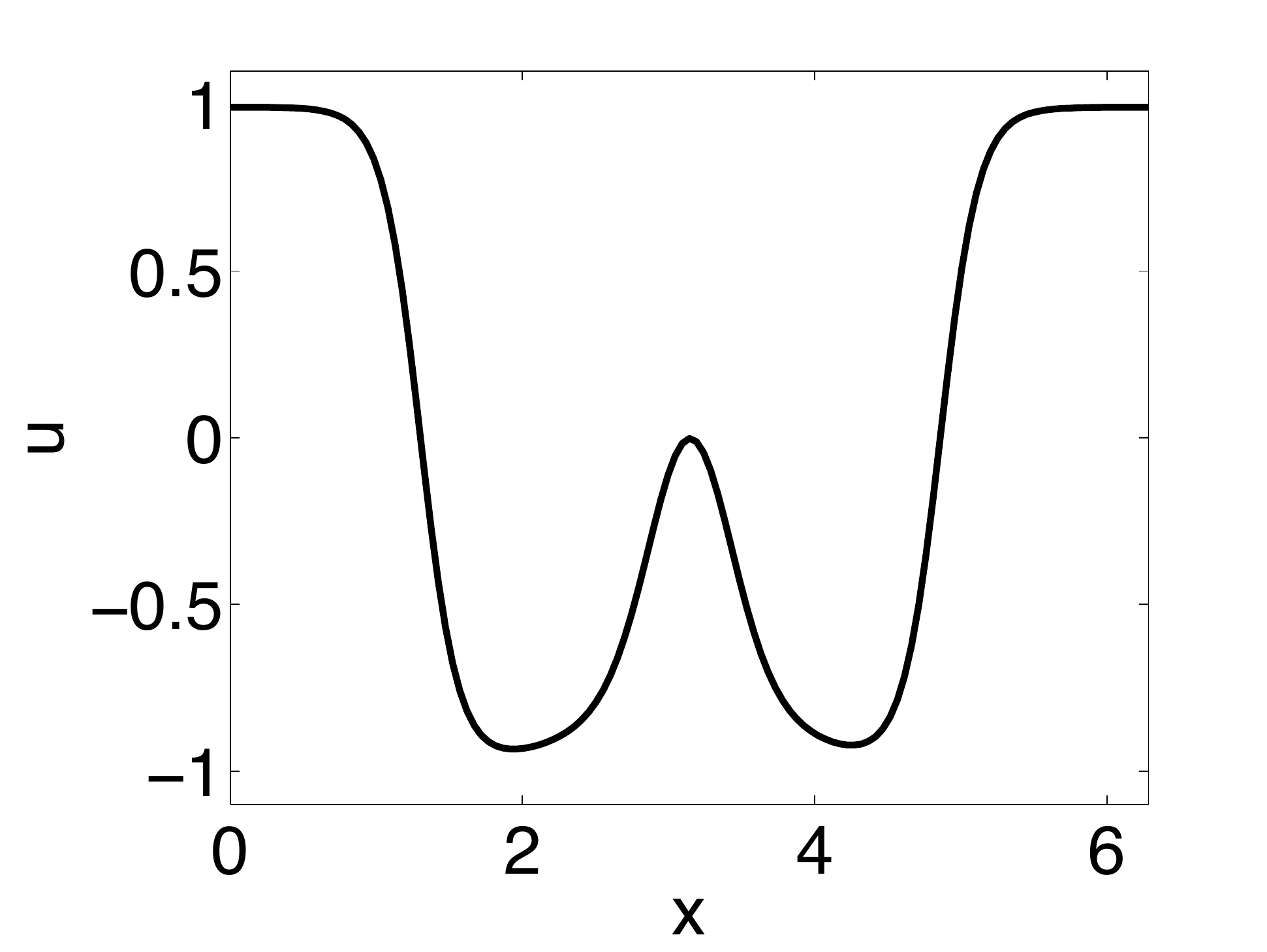}}
		\subfigure[$u(x,8319.2)$]{\label{fig:CH1D_uT}\includegraphics[width = 0.3\textwidth]{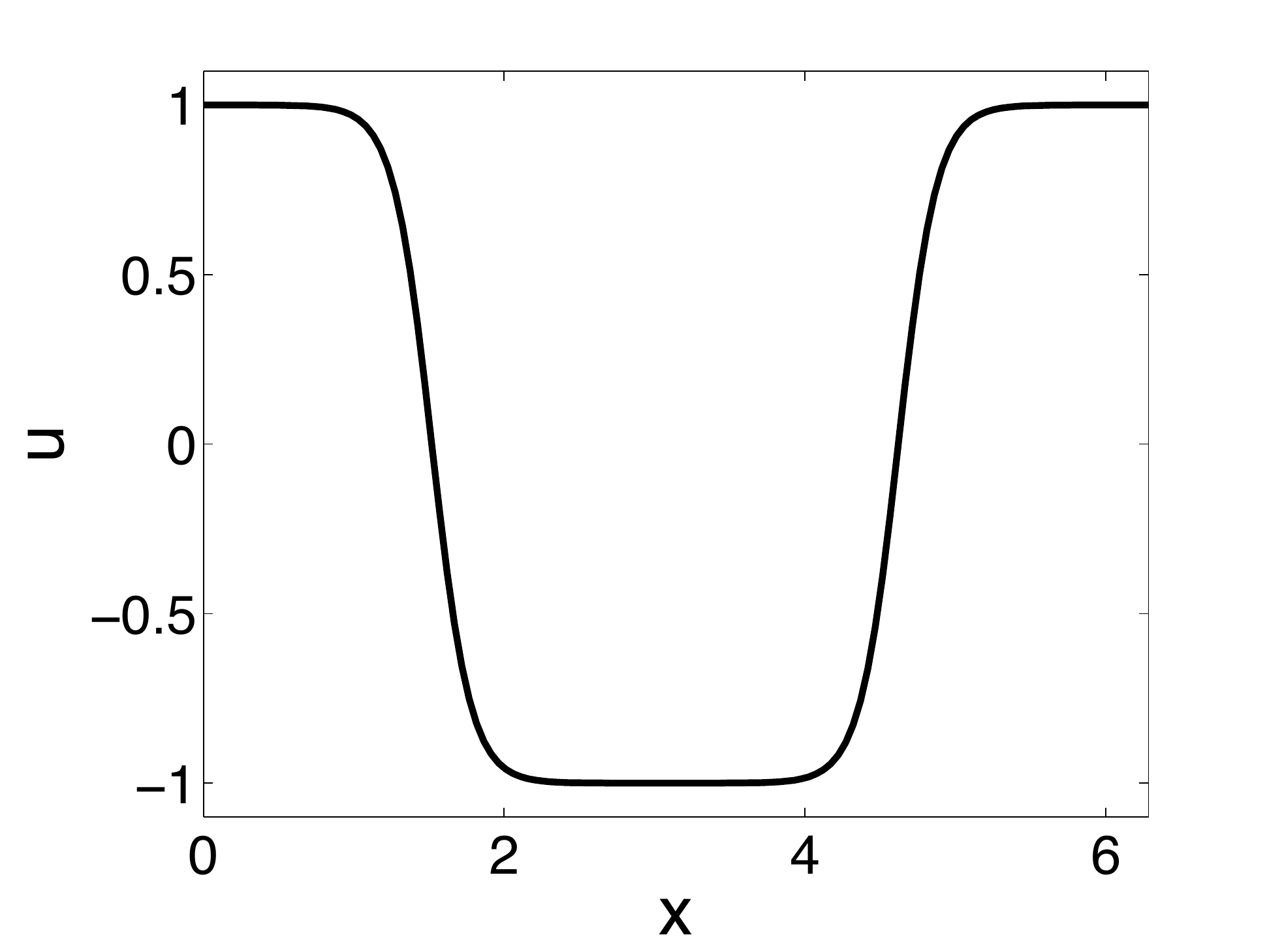}}
		\caption{Spinodal evolution for the 1D CH equation using adaptive time stepping (BE-BDF2).}
		\label{fig:CH_1d_ripening}
	\end{center}
\end{figure}
In Figure \ref{fig:CH_1d_history} we plot the time step size, fixed-point iteration history and energy of numerical solutions obtained by the BE-BDF2 strategy ($\delta_{\text{tol}} = 10^{-5}$). As shown in Figure \ref{fig:CH1D_time}, small time steps are used at early stage (spinodal evolution) but increase when the coarsening process starts, which speeds up the simulation without a loss of accuracy.  If the iteration count is too large ($N_{\text{it}}\geq N_{\max\text{it}}$), we reject the solution and compute $u(x,t)$ again with reduced $\Delta t$. We see this behavior in Figure \ref{fig:CH1D_iteration}, where time steps are decreased whenever $N_{\text{it}} \approx N_{\max\text{it}}$. We also observe in Figure \ref{fig:CH1D_energy} that adaptive time stepping does not affect energy decay.

The phase function $u(x,t)$ obtained by our numerical scheme is shown in Figure \ref{fig:CH_1d_ripening}. The initial state \ref{fig:CH1D_u0} quickly moves to the metastable state \ref{fig:CH1D_u1}, and then finally reaches the stable state \ref{fig:CH1D_uT} at which two layers are merged together after a long time. This simulation also shows that the ripening event happens over a very fast time scale in Figure \ref{fig:CH1D_u4}.

\subsection{2D CH model}  \label{sec:numerical_2D} 
We next solve the CH equation \eqref{eqn:CahnHilliard} in two spatial dimensions. The parameters are 
$$\epsilon = 0.18,\quad \Delta x = \Delta y = \frac{2\pi}{128} \approx 0.0491, \quad  N_{\max\text{it}} = 300, \quad N_{\text{tol}} = 10^{-6},$$ 
and the initial condition is \eqref{eqn:CH2D_initial}. We first use BE, BDF2, and BDF3 with a small time step ($\Delta t = 0.05$) and check the ripening time, which is defined such that $u(\frac{\pi}{2},\frac{\pi}{2})$ changes from positive to negative. Based on the results summarized in Figure \ref{test:2DCH_ripening_smalldt}, we can define such ripening time as $T_{\text{r}} = 80.10$. In addition, we note that raising the order of the scheme reduces the number of iterations per time step, and so the overall computational time is lower.

\begin{figure}[h]
  \begin{minipage}{0.45\linewidth}
    \centering
	\includegraphics[width = 0.8\linewidth]{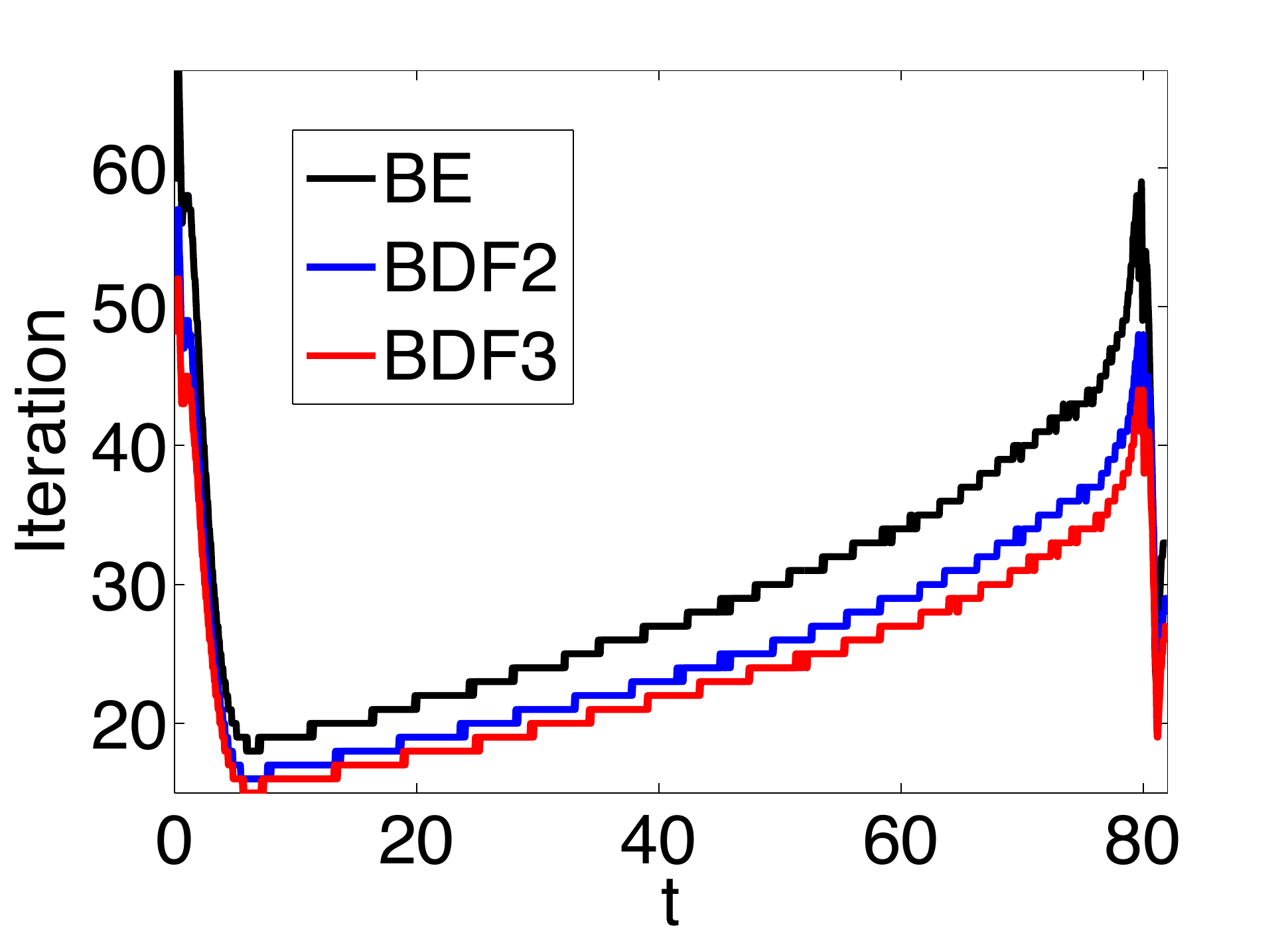}
  \end{minipage}%
  \begin{minipage}[b]{0.5\linewidth}
    \centering
 \begin{tabular}{rcc}
\hline 
		& $\qquad$     Ripening time & $\qquad$ Run Time $(s)$ \\ \hline 
BE$\quad$	& $\qquad$ $79.95$	& $\qquad$ $630.16$\\	
BDF2	& $\qquad$ $80.10$	& $\qquad$ $526.19$ \\	
BDF3	& $\qquad$ $80.10$	 & $\qquad$ $481.86$ \\	 \hline
\end{tabular}
\end{minipage}
\caption{Ripening time of 2D CH equation with small fixed time $(\Delta t=0.01)$}
\label{test:2DCH_ripening_smalldt}
\end{figure}

\begin{figure}[h]
	\begin{center}
		\centering
		\subfigure[Time step size]{\label{fig:CH2D_time}\includegraphics[width = 0.32\linewidth]{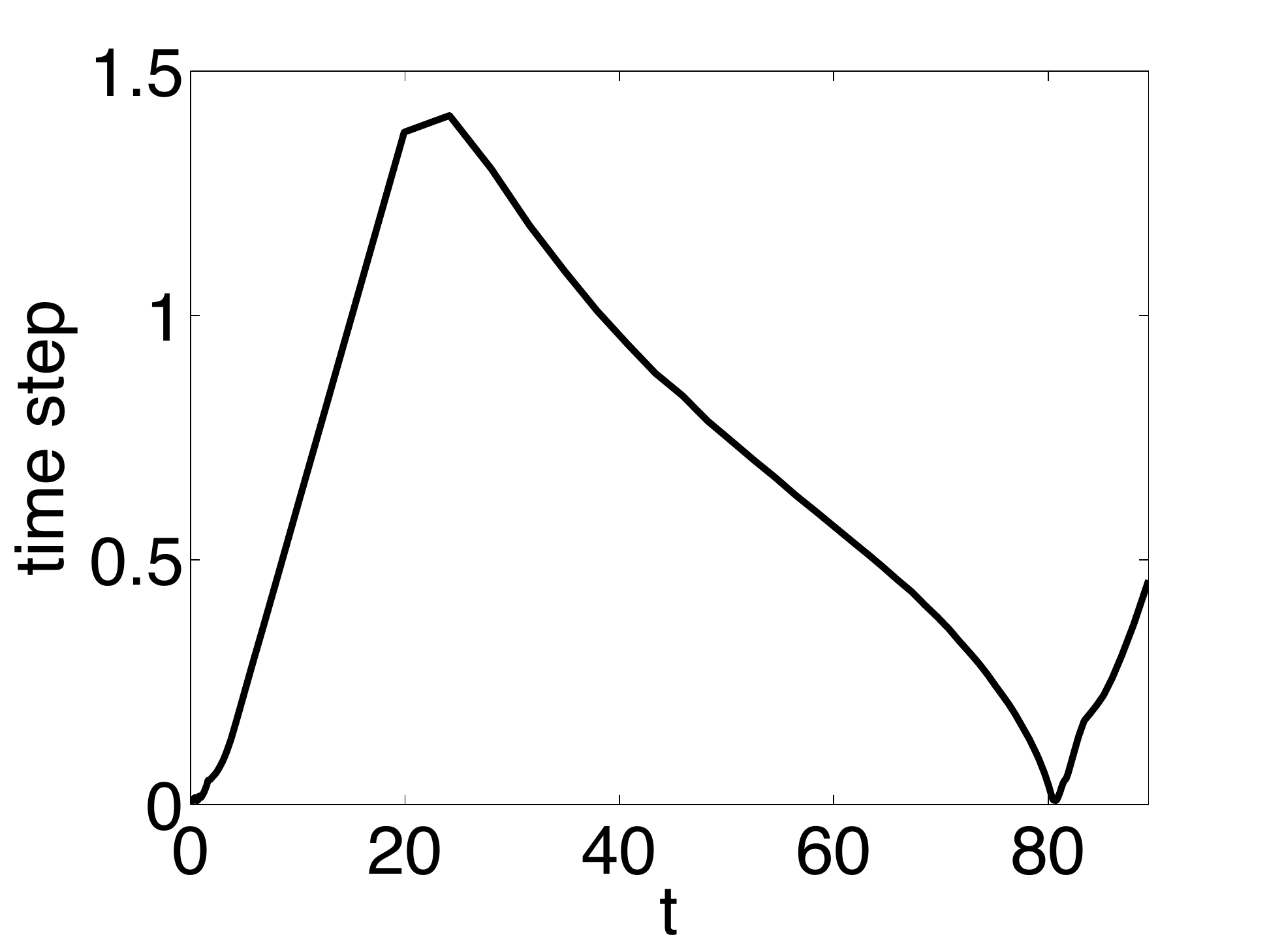}}
		\subfigure[Iteration count]{\label{fig:CH2D_iteration}\includegraphics[width = 0.32\linewidth]{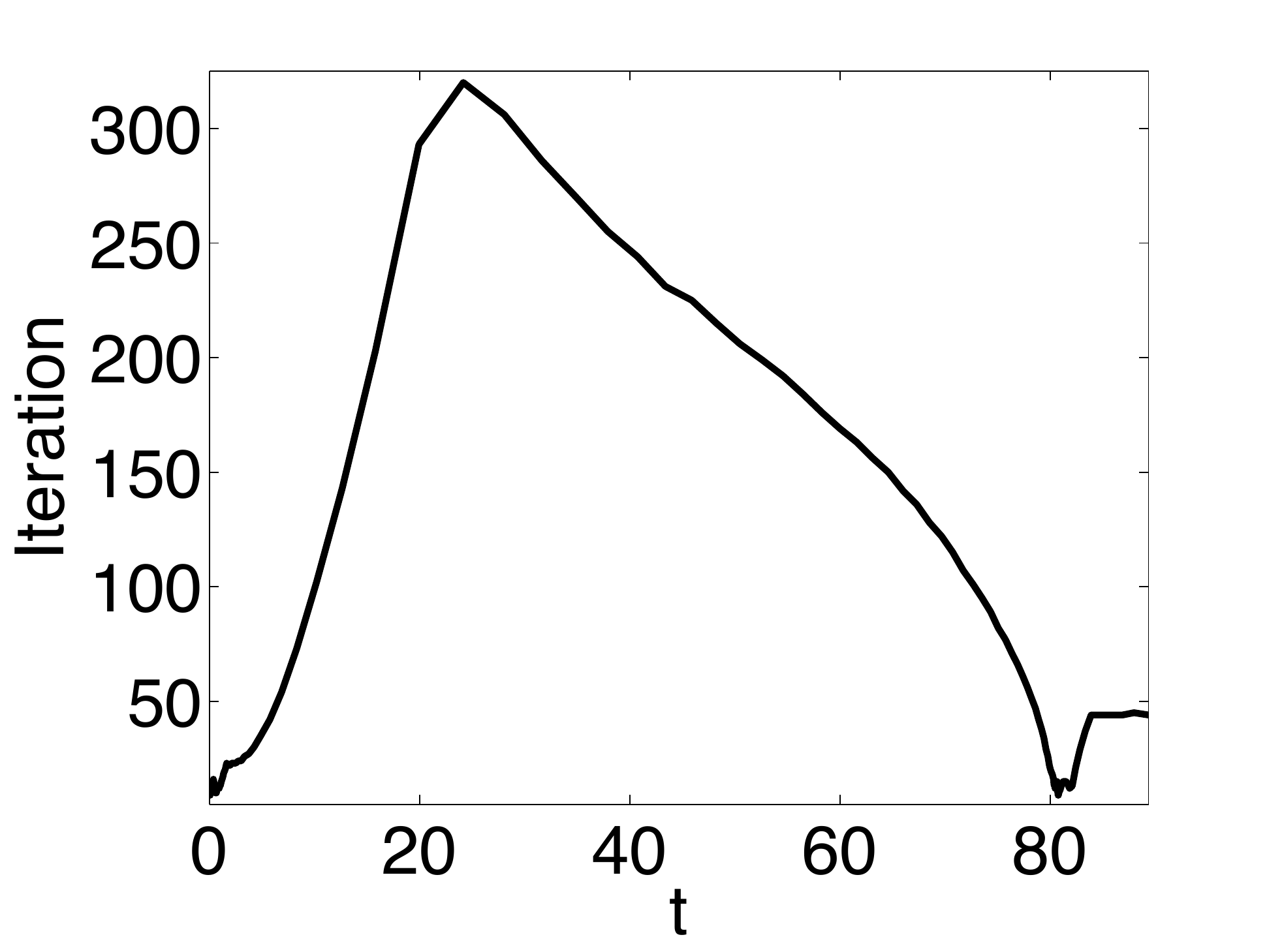}}
		\subfigure[Energy]{\label{fig:CH2D_energy}\includegraphics[width = 0.32\linewidth]{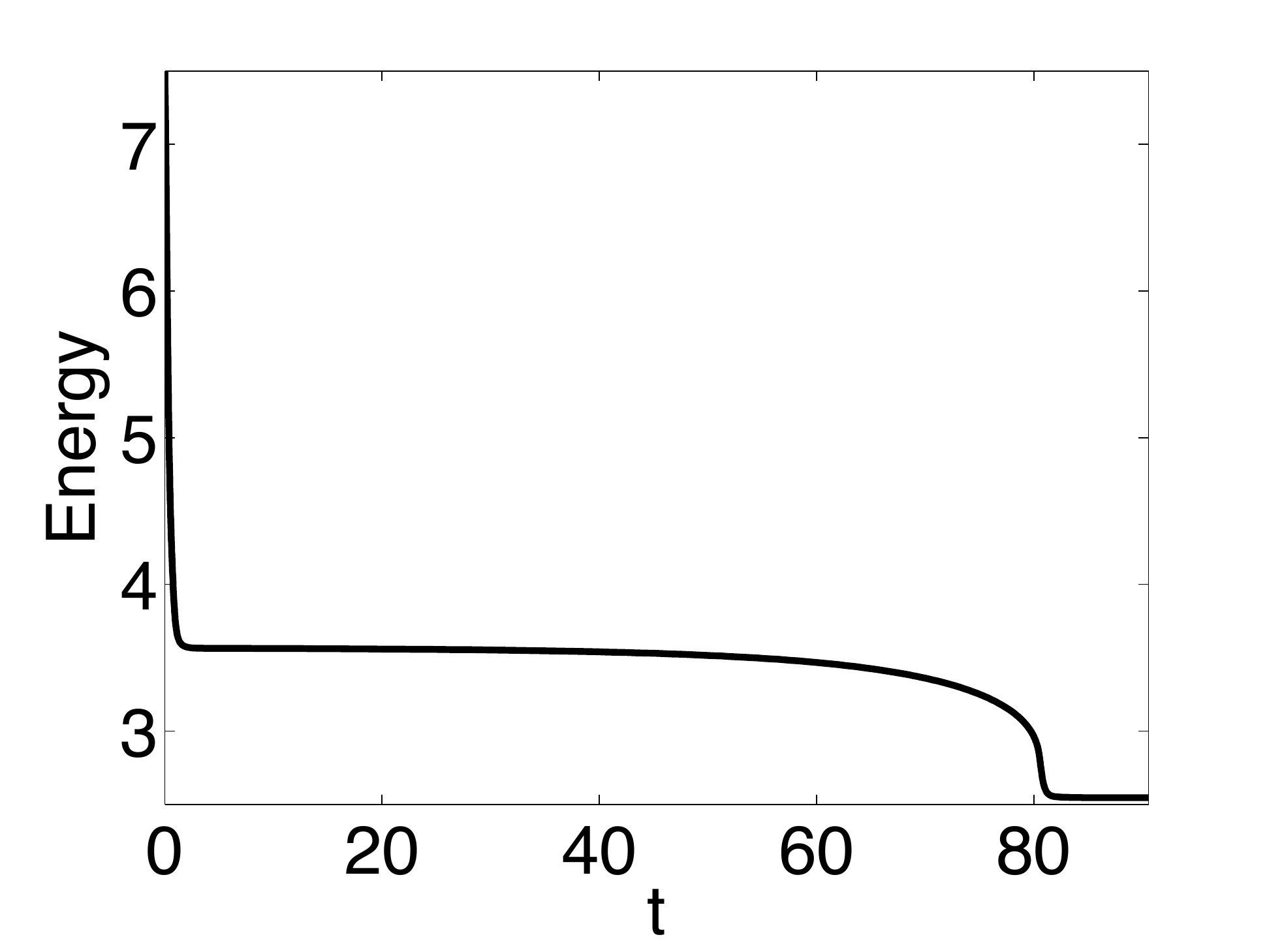}}
	    \caption{Adaptive time stepping for 2D, with the BE-BDF2 strategy.}
		\label{fig:CH_2d_history}
	\end{center}
\end{figure}
Next, we simulate the same problem with adaptive time step method (BE-BDF2). The time step history, number of nonlinear iterations, and energy history are presented in Figure \ref{fig:CH_2d_history}. The 2D results are comparable to those from 1D in Section \ref{sec:numerical_1D}, in that larger time step sizes are used during coarsening, and smaller steps are required only to capture ripening (see Figure \ref{fig:CH2D_time}). The energy in Figure \ref{fig:CH2D_energy} indicates that there are two sharp transitions in the energy $\mathcal{E}$; early on, and at ripening. Finally, we also observe the desired energy decay.
\begin{figure}[h!]
\begin{center}
\centering
    \subfigure[$u(x,y,0)$]{\label{fig:CH2D_u0}\includegraphics[width=0.32\linewidth]{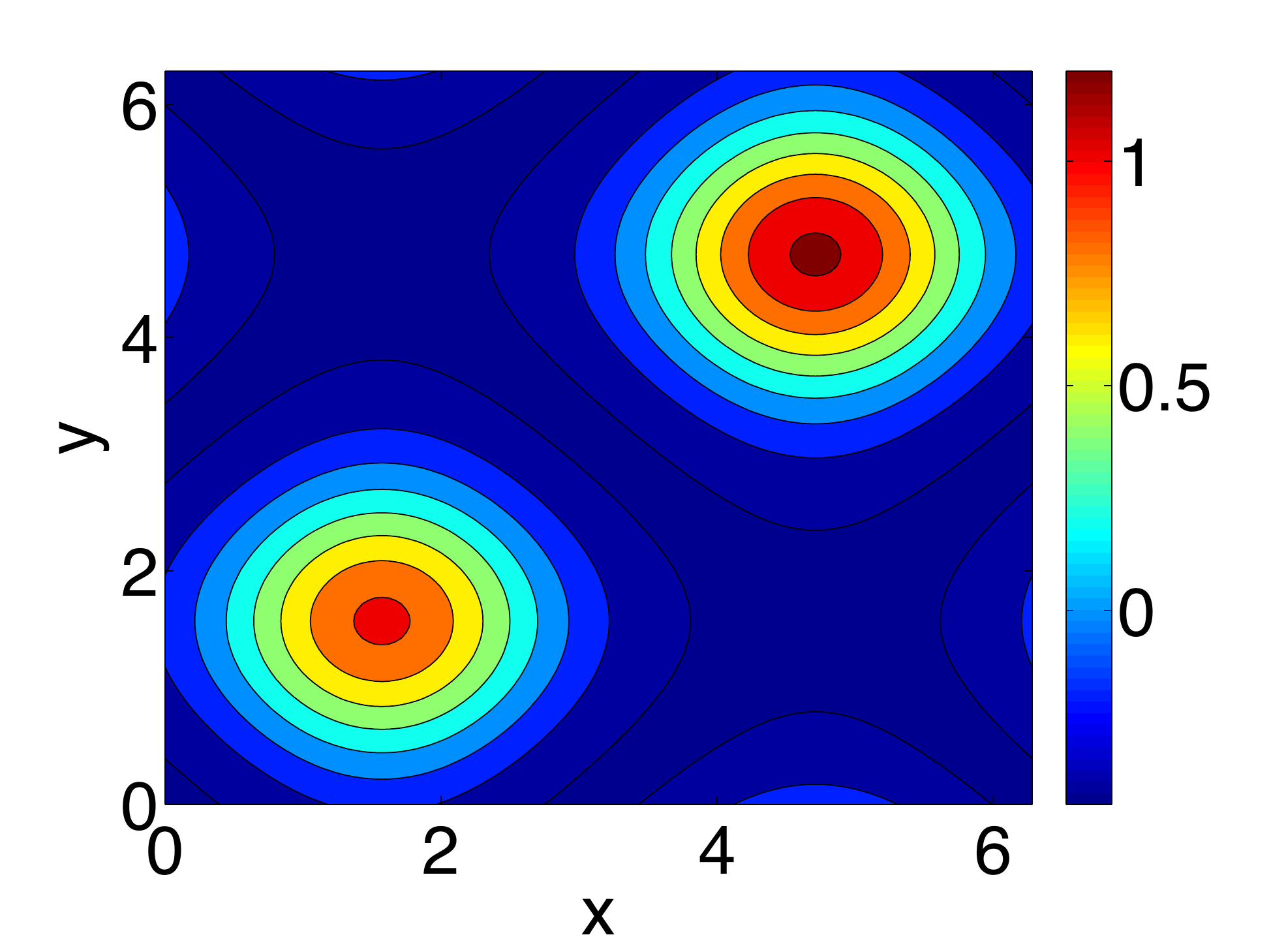}}
    \subfigure[$u(x,y,0.5083)$]{\label{fig:CH2D_u1}\includegraphics[width=0.32\linewidth]{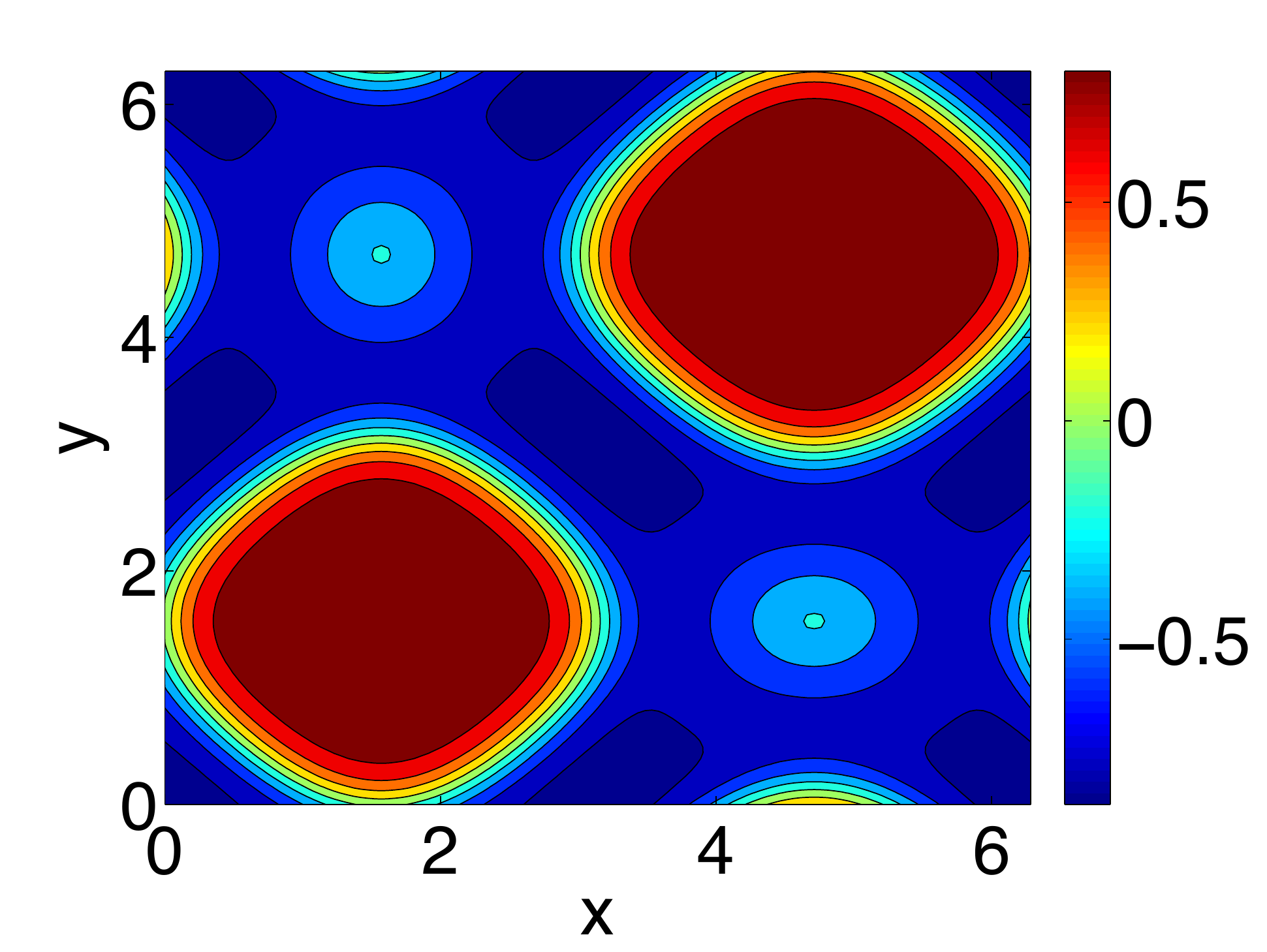}} 
    \subfigure[$u(x,y,15.807)$]{\label{fig:CH2D_u2}\includegraphics[width=0.32\linewidth]{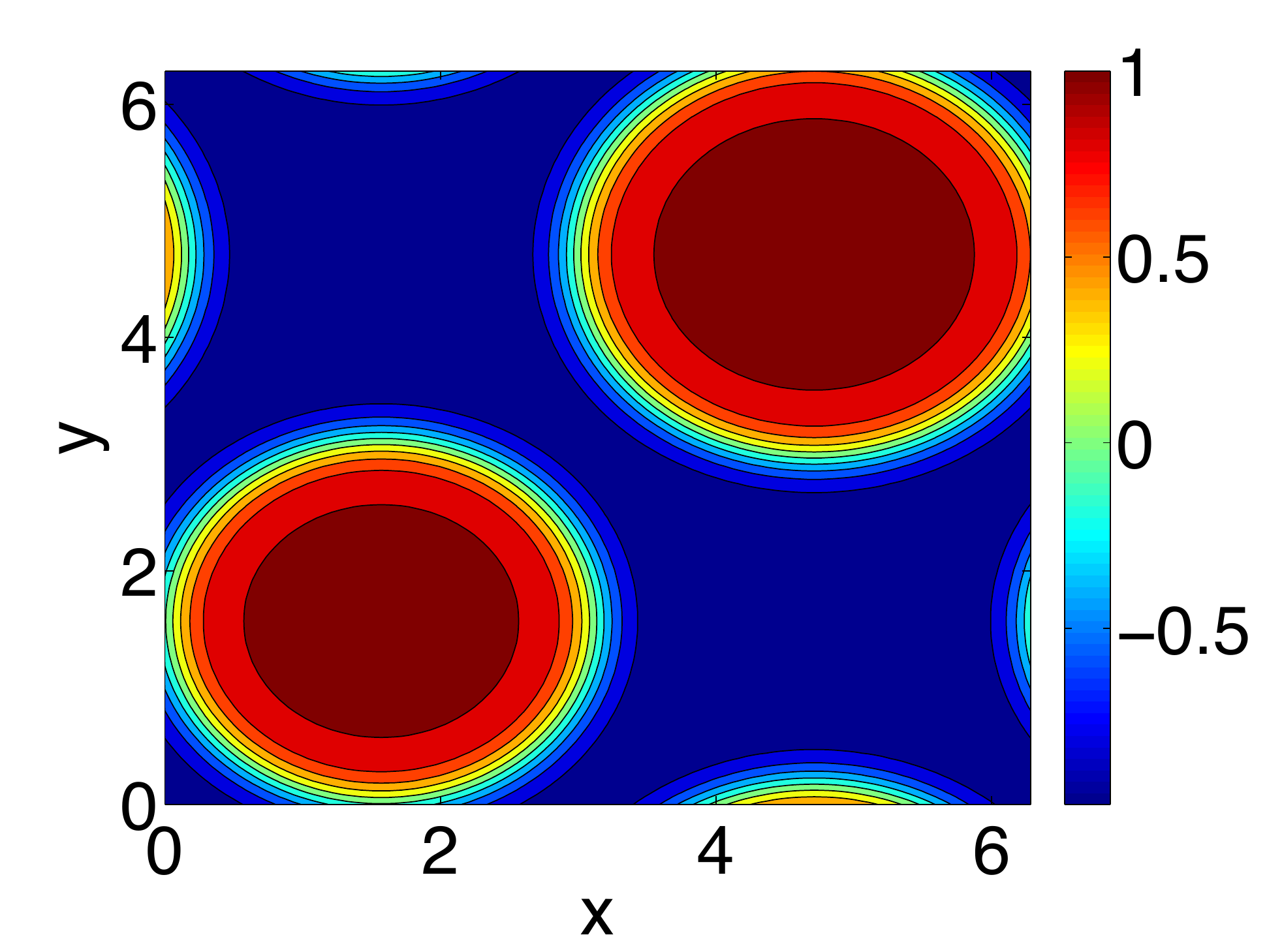}}\\
    \subfigure[$u(x,y,70.015)$]{\label{fig:CH2D_u3}\includegraphics[width=0.32\linewidth]{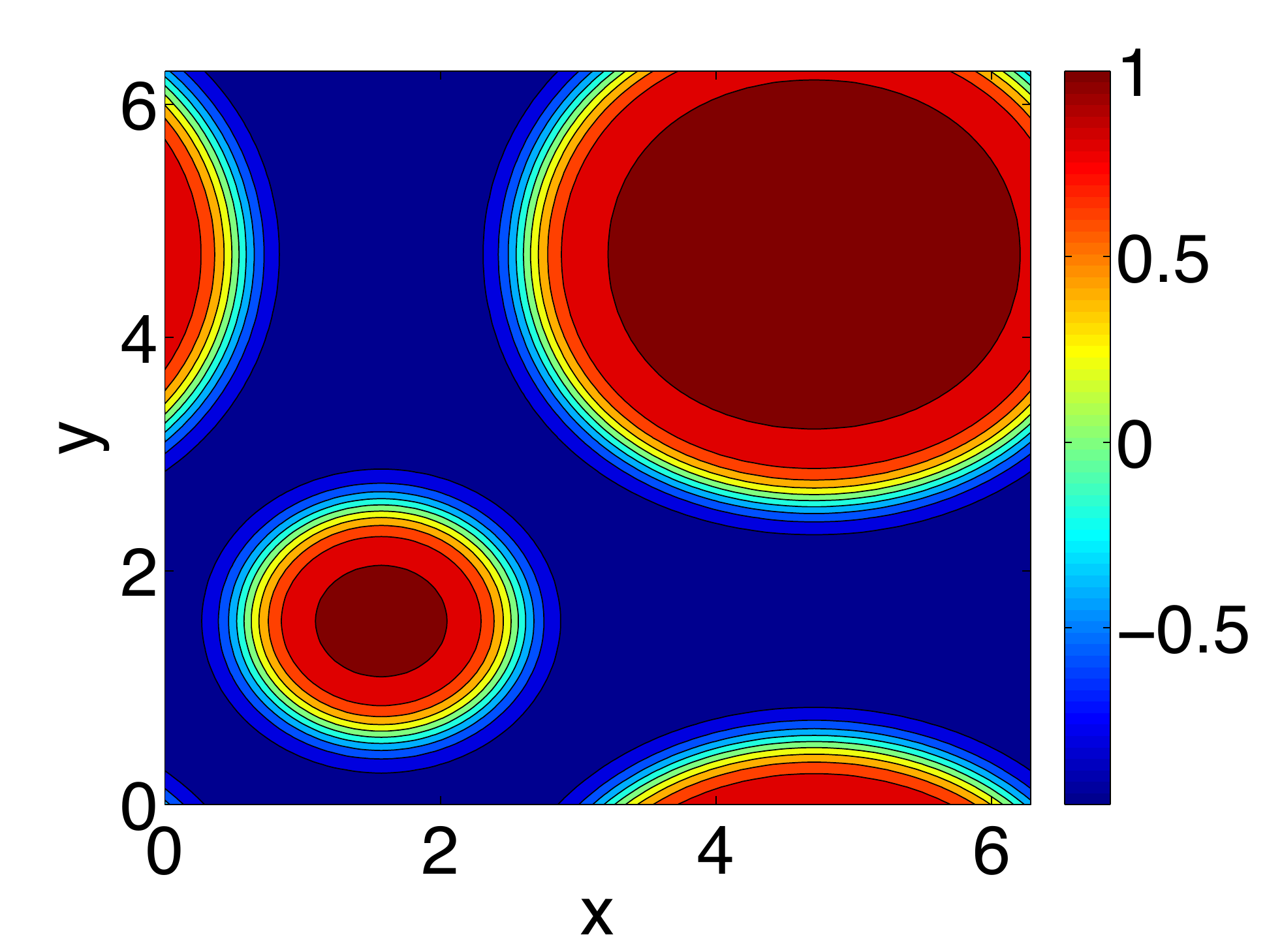}}
    \subfigure[$u(x,y,80.536)$]{\label{fig:CH2D_u4}\includegraphics[width=0.32\linewidth]{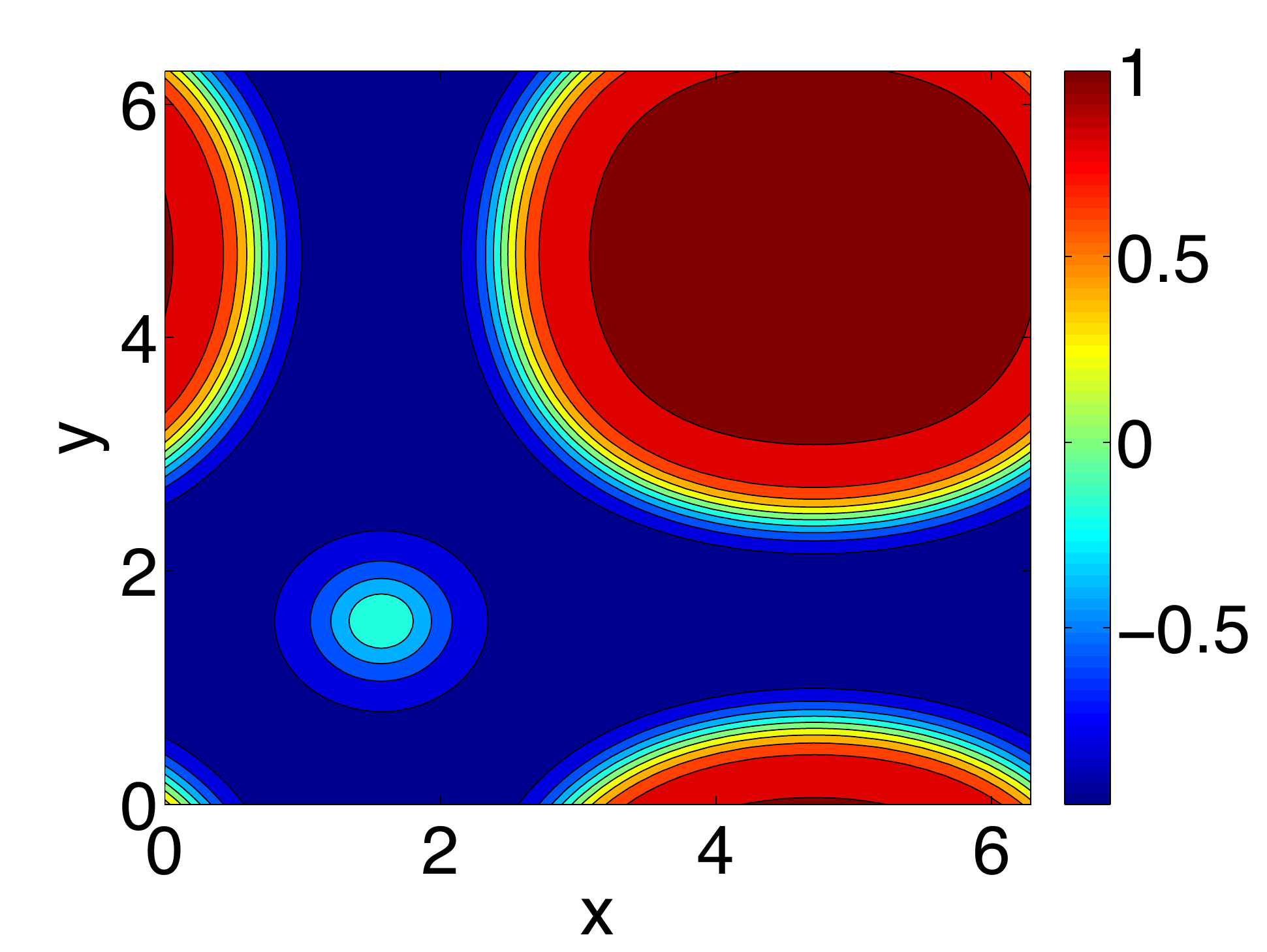}}
    \subfigure[$u(x,y,90)$]{\label{fig:CH2D_uT}\includegraphics[width=0.32\linewidth]{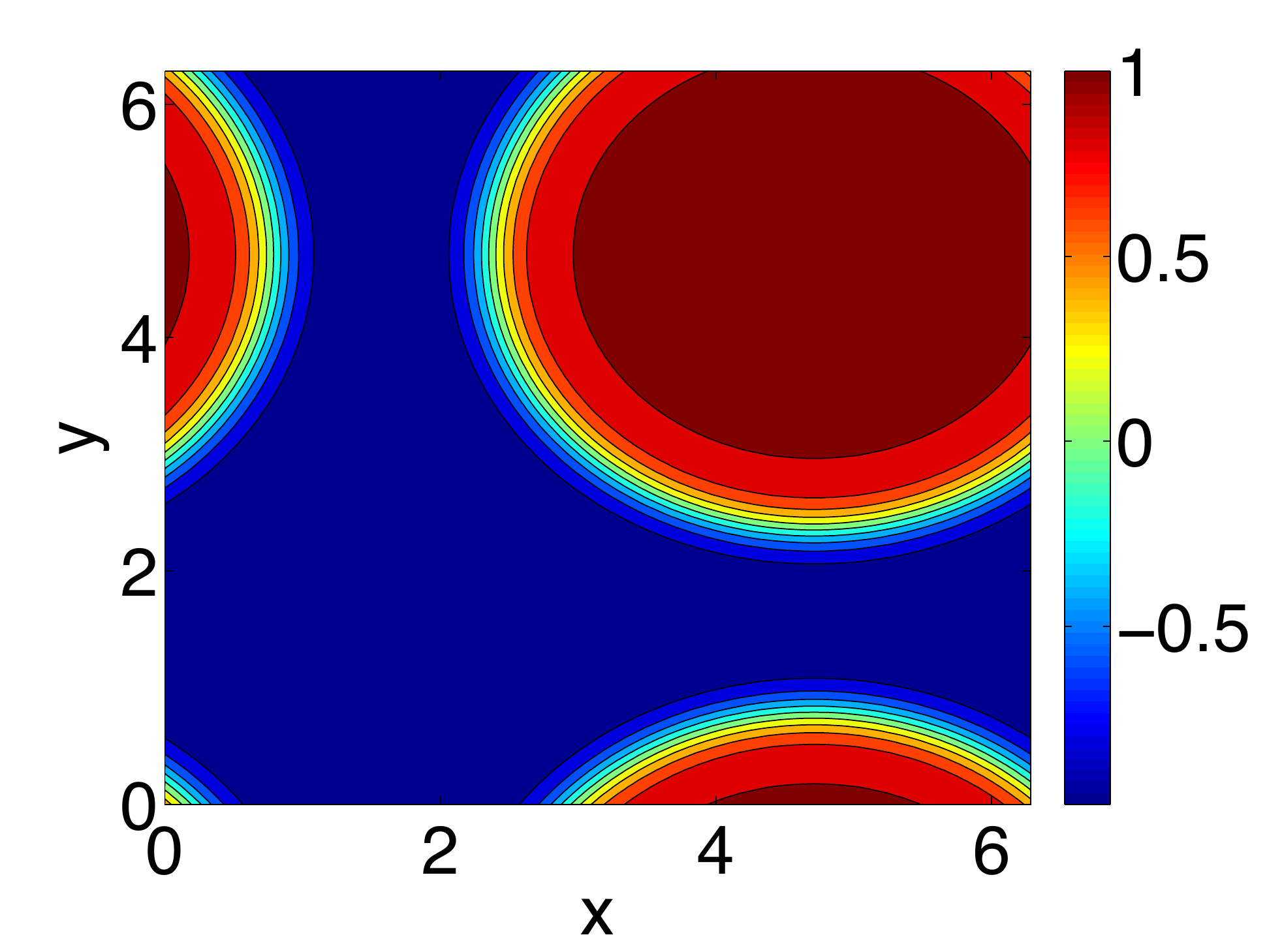}} 
 \caption{Temporal evolution of the 2D CH solution with initial condition \eqref{eqn:CH2D_initial}.}
    \label{fig:CH2D}
\end{center}
\end{figure}

The contour plots of time evolution of the phase function $u(x,y,t)$ are shown in Figure \ref{fig:CH2D}. The initial states \ref{fig:CH2D_u0} quickly reaches the metastable state, where we see two circular formations. Eventually the larger one absorbs the smaller, although over a very long time scale, shown in \ref{fig:CH2D_u2} - \ref{fig:CH2D_u4}. The final state is shown in \ref{fig:CH2D_uT}, where the larger region has fully consumed the smaller. In all plots, the total volume is preserved.

\subsection{2D CH vector model}  \label{sec:numerical_vector} 
We next apply our scheme from Section \ref{sec:CH_vector_2D} to the 2D CH system \eqref{eqn:CH_vector}, with our adaptive time stepping strategy. 
We use the same initial condition \eqref{eqn:CH2D_initial} for $u_1$ and 
\begin{equation}\label{eqn:CH2D_initial_v0}
u_2(x,y,0) = \sin(y), \qquad (x,y) \in [0,2\pi]^2
\end{equation}
for $u_2$ in the reference \cite{Jaylan}, and observe the long time behavior of the phase function ${\bf u} = (u_1,u_2)$. With the parameters 
$$\epsilon = 0.32, \Delta x = \Delta y = \frac{2\pi}{64} \approx 0.0982,  \quad  N_{\max\text{it}} = 400, \quad N_{\text{tol}} = 10^{-6}, \quad \delta_{\text{tol}} = 1e-4$$ 
we implement BE-BDF2 adaptive scheme for VCH model. 
\begin{figure}[h]
\begin{center}
\centering
    \subfigure[Time step size]{\label{fig:VCH2D_time}\includegraphics[width = 0.3\linewidth]{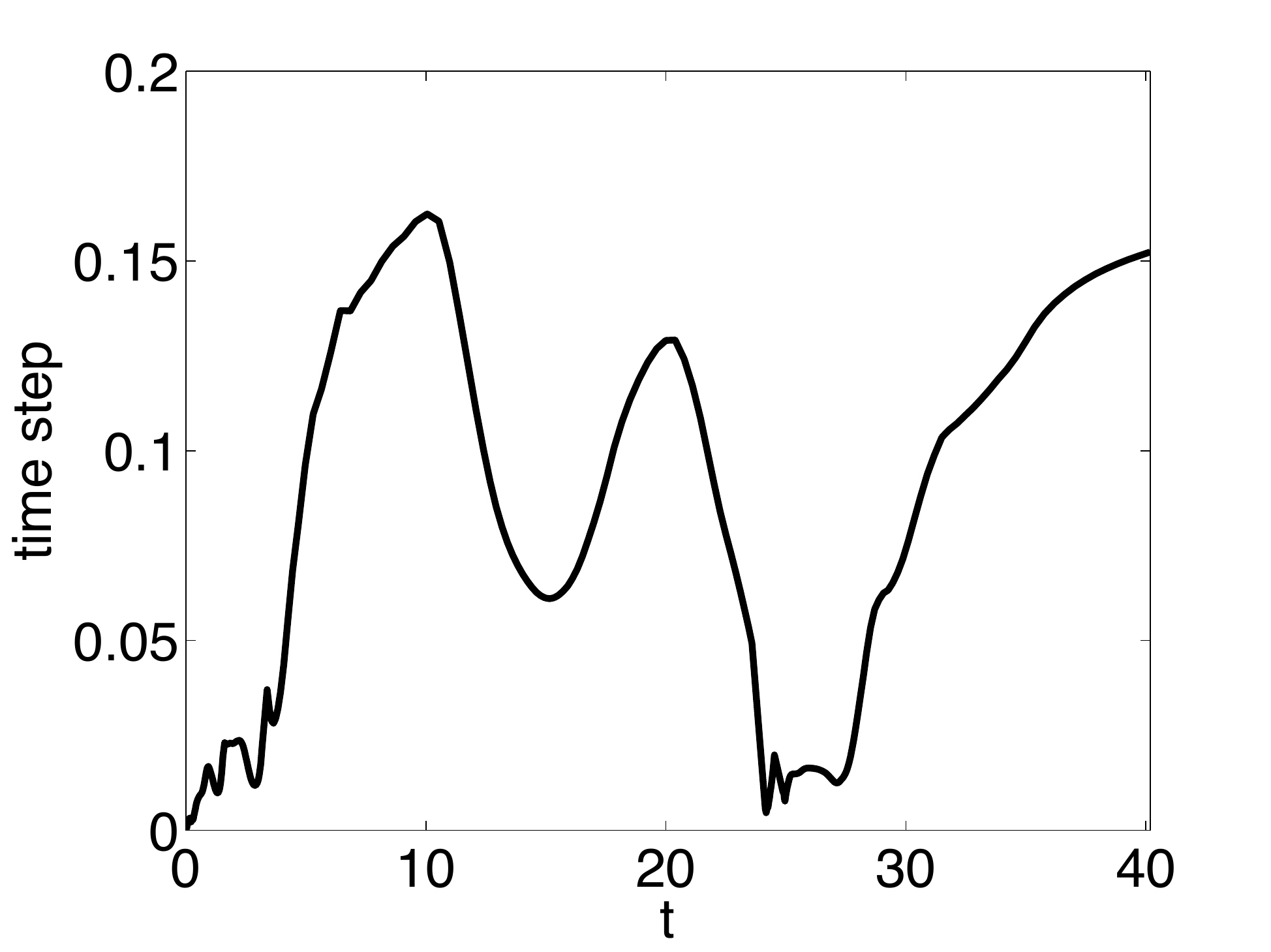}}
    \subfigure[Iteration count]{\label{fig:VCH2D_iteration}\includegraphics[width = 0.3\linewidth]{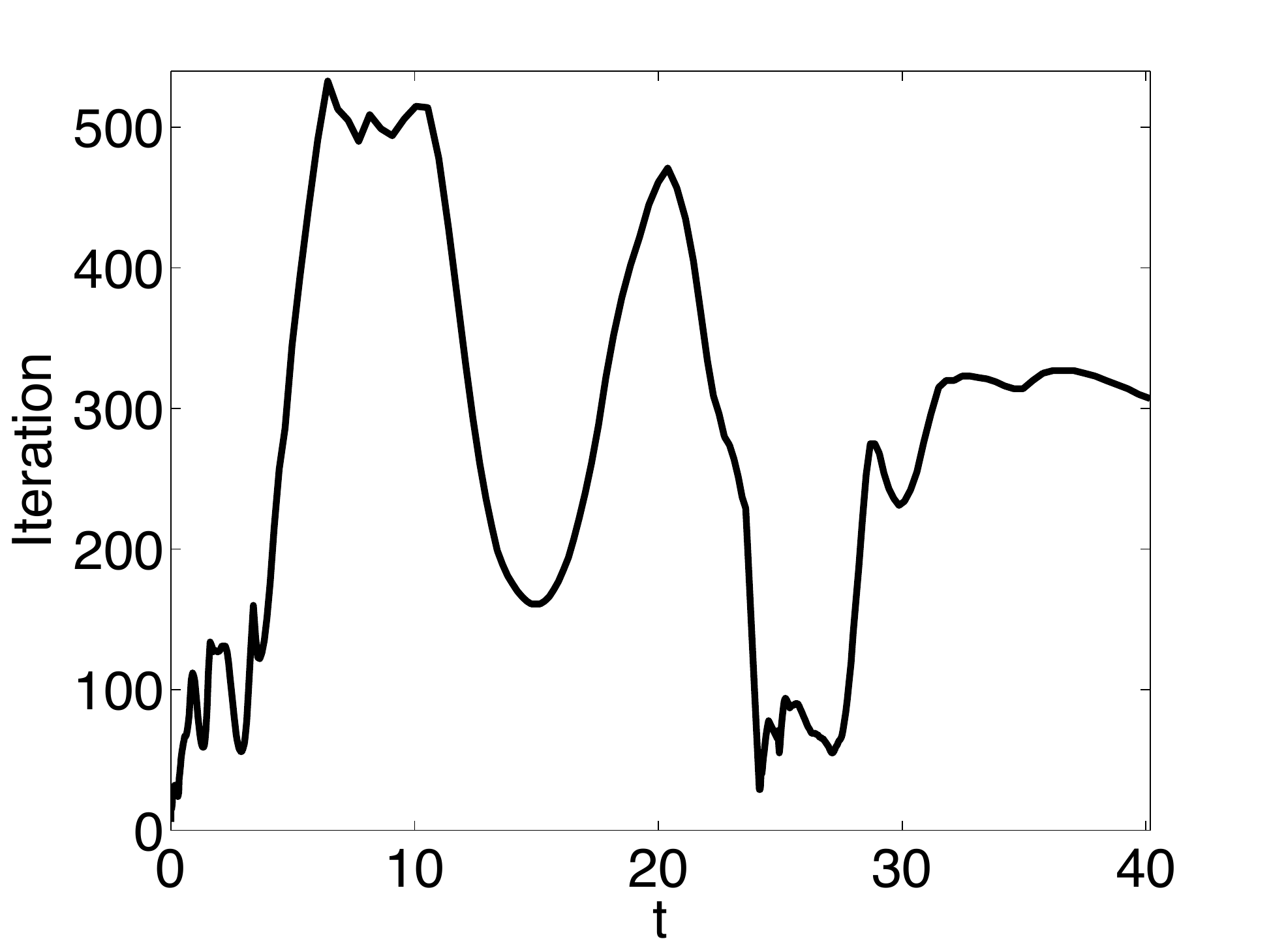}}
    \subfigure[Energy]{\label{fig:VCH2D_energy}\includegraphics[width = 0.3\linewidth]{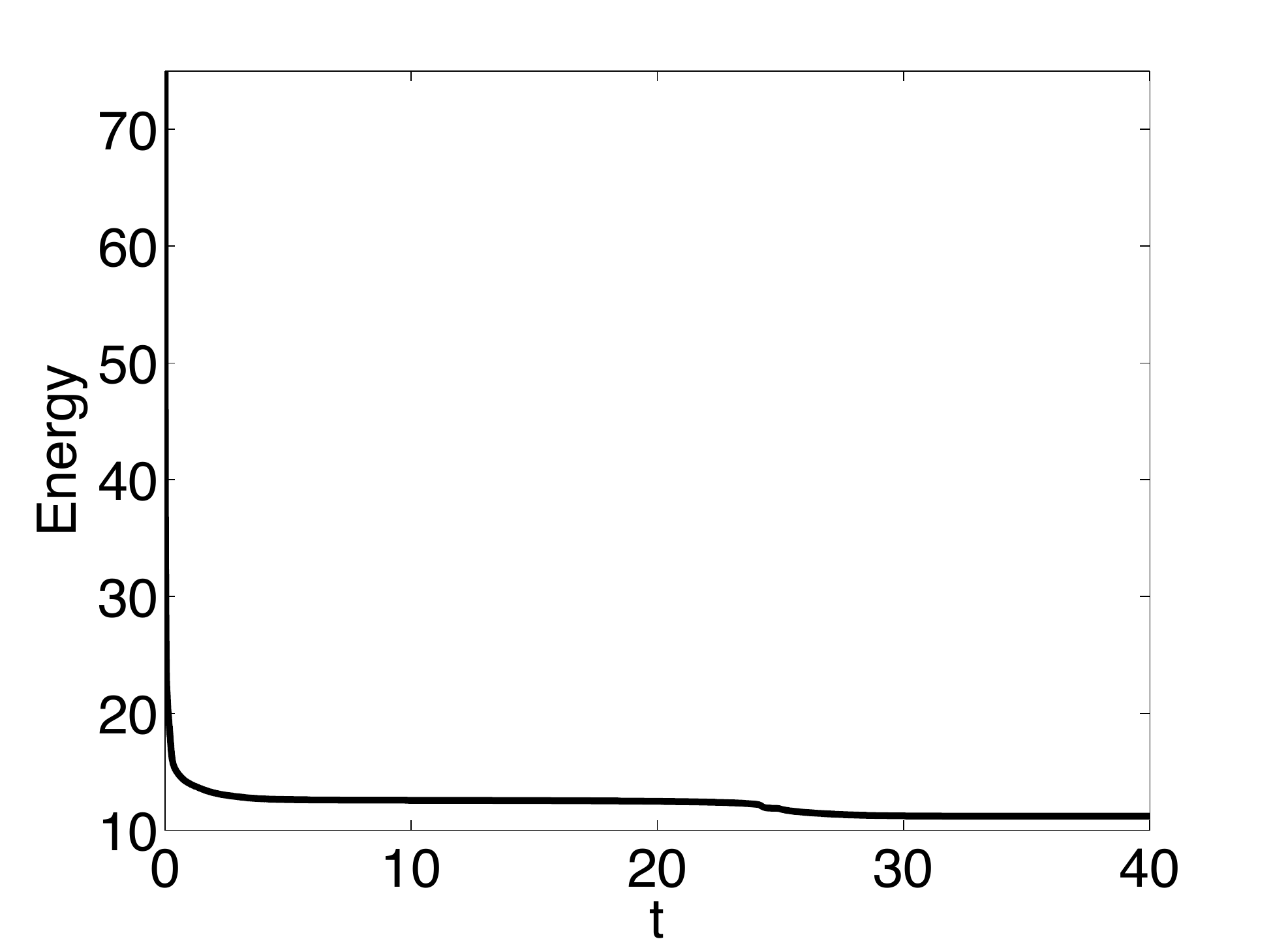}}
    \caption{The time step, iteration count, and energy history of VCH solution.}
    \label{fig:CH_2d_vector_history}
\end{center}
\end{figure}
The contour plots of  $\cos(\arg (u_1 + i u_2))$ are shown  in Figure \ref{fig:CH2D_vector}. Instead of plotting $u_1$ and $u_2$ separately, \cite{Jaylan} we define the angle $\theta \equiv \arg (u_1 + i u_2) $ at a triple junction, and plot $\cos(\theta)$ to avoid any discontinuities. We follow this benchmark plot using our numerical solution.
\begin{figure}[h!]
	\begin{center}
		\centering
		\subfigure[$t=0$]      {\label{fig:VCH2D_u0}\includegraphics[width=0.3\linewidth]{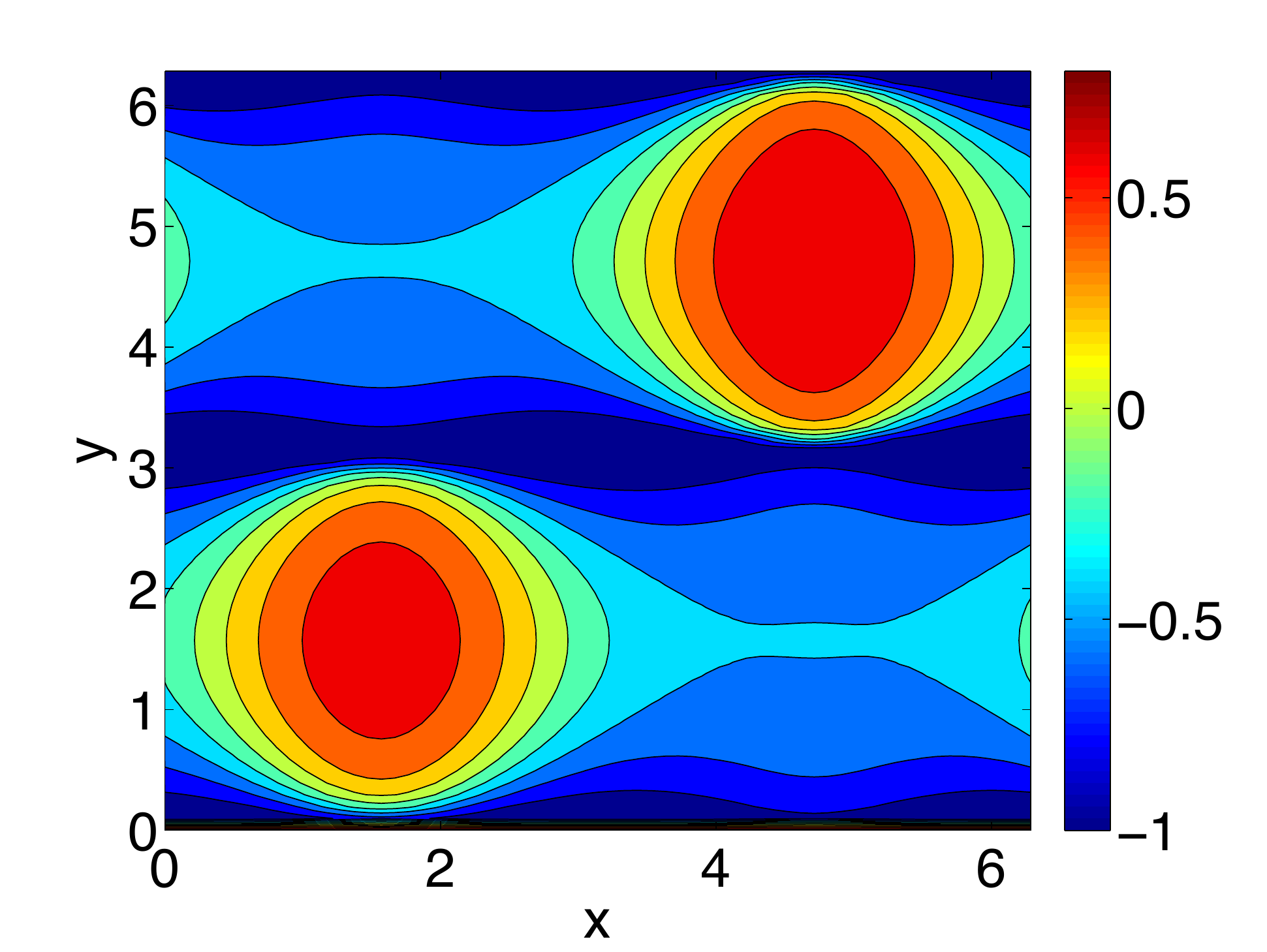}}
		\subfigure[$t=0.02503$]{\label{fig:VCH2D_u1}\includegraphics[width=0.3\linewidth]{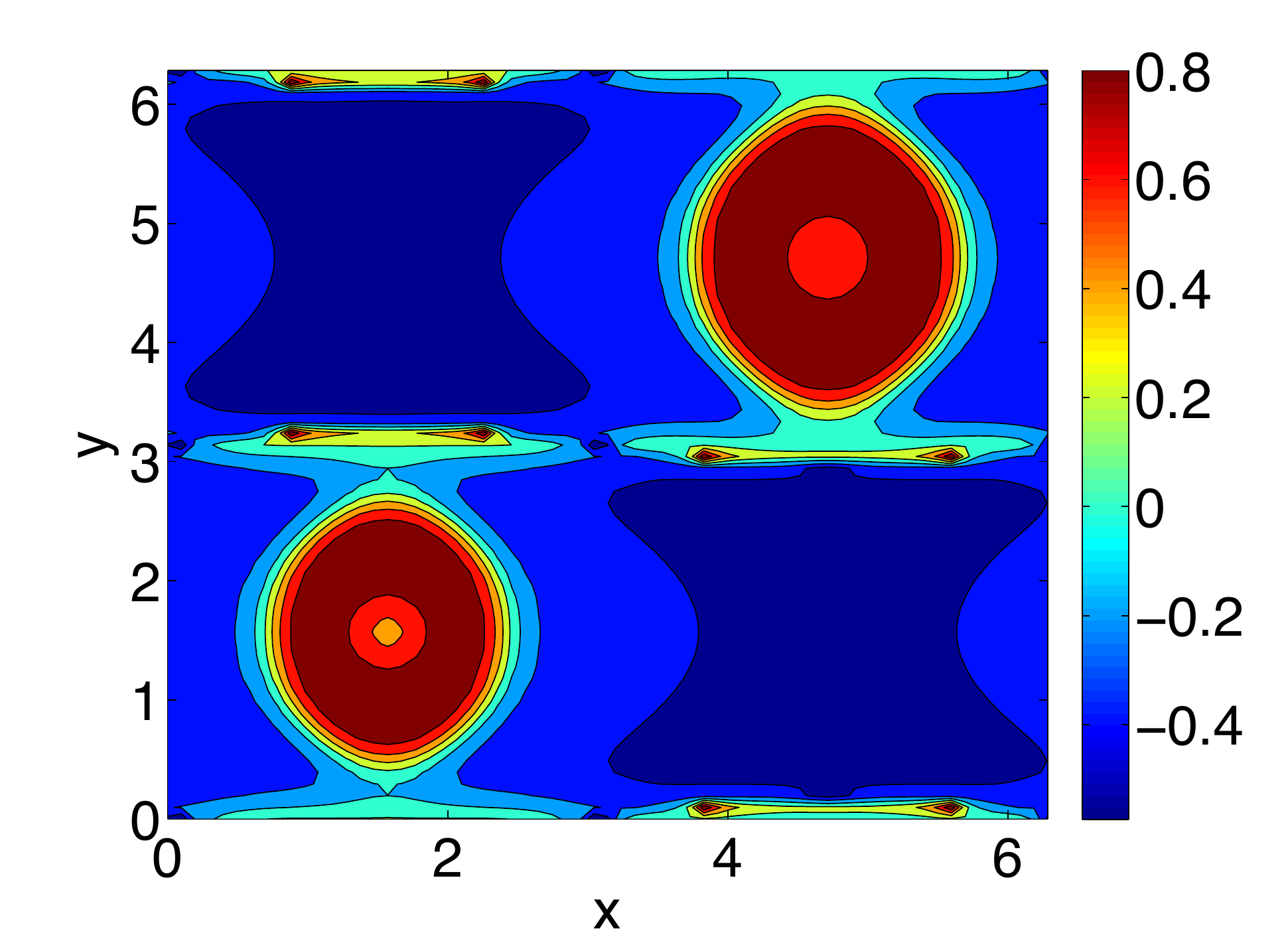}} 
		\subfigure[$t=0.5083$] {\label{fig:VCH2D_u2}\includegraphics[width=0.3\linewidth]{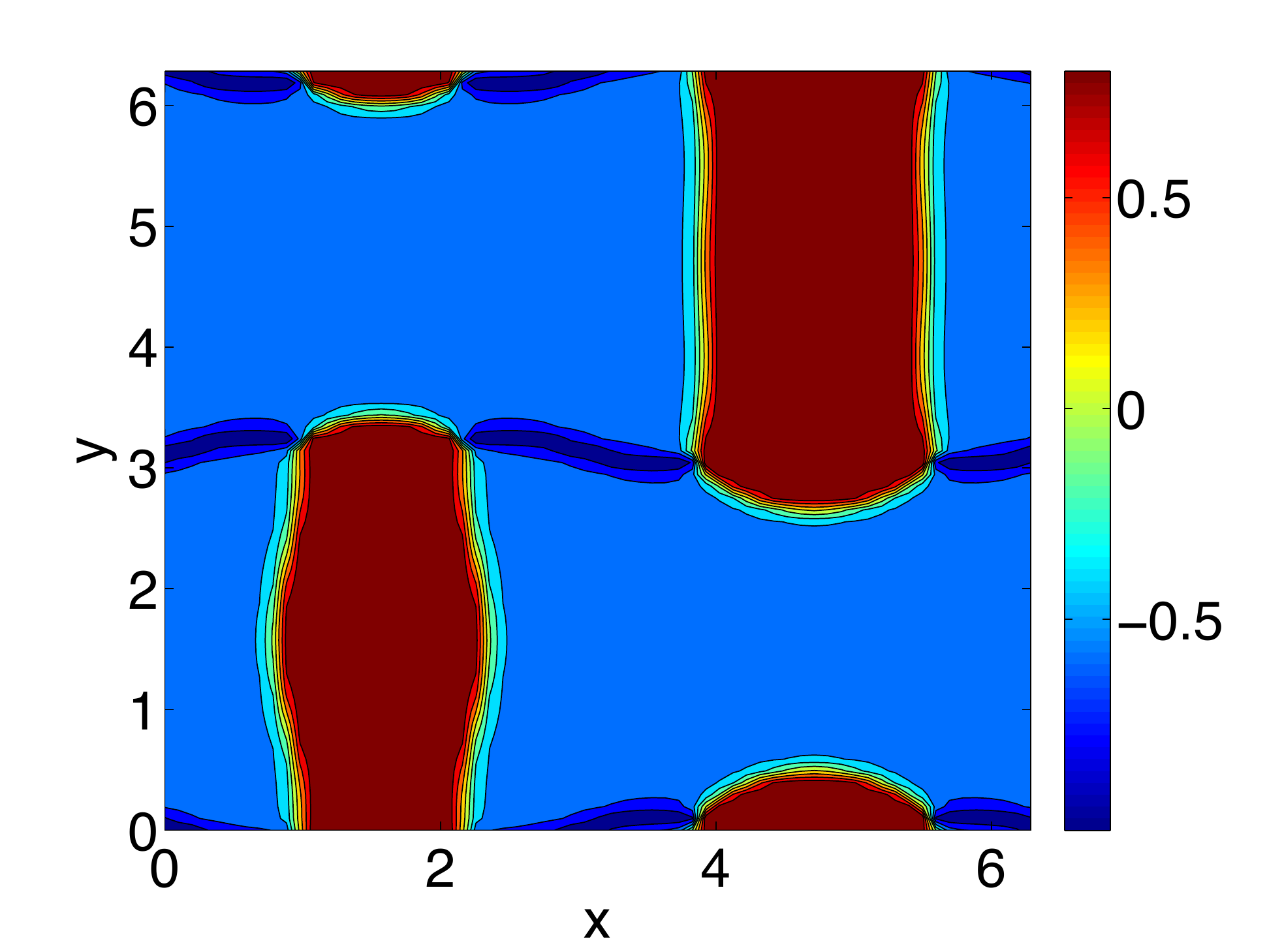}}\\
		\subfigure[$t=5.0848$] {\label{fig:VCH2D_u3}\includegraphics[width=0.3\linewidth]{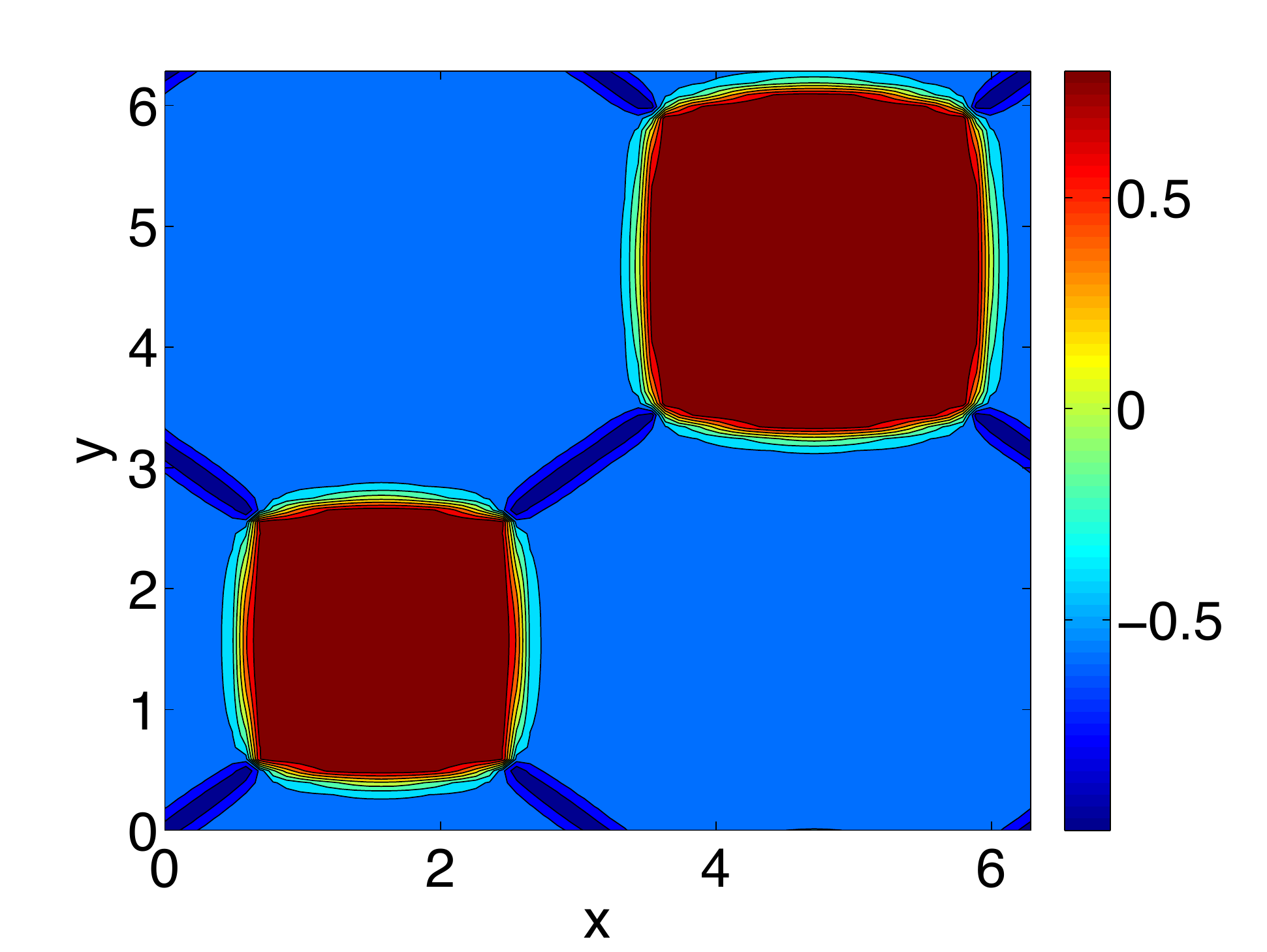}}
		\subfigure[$t=20.128$] {\label{fig:VCH2D_u4}\includegraphics[width=0.3\linewidth]{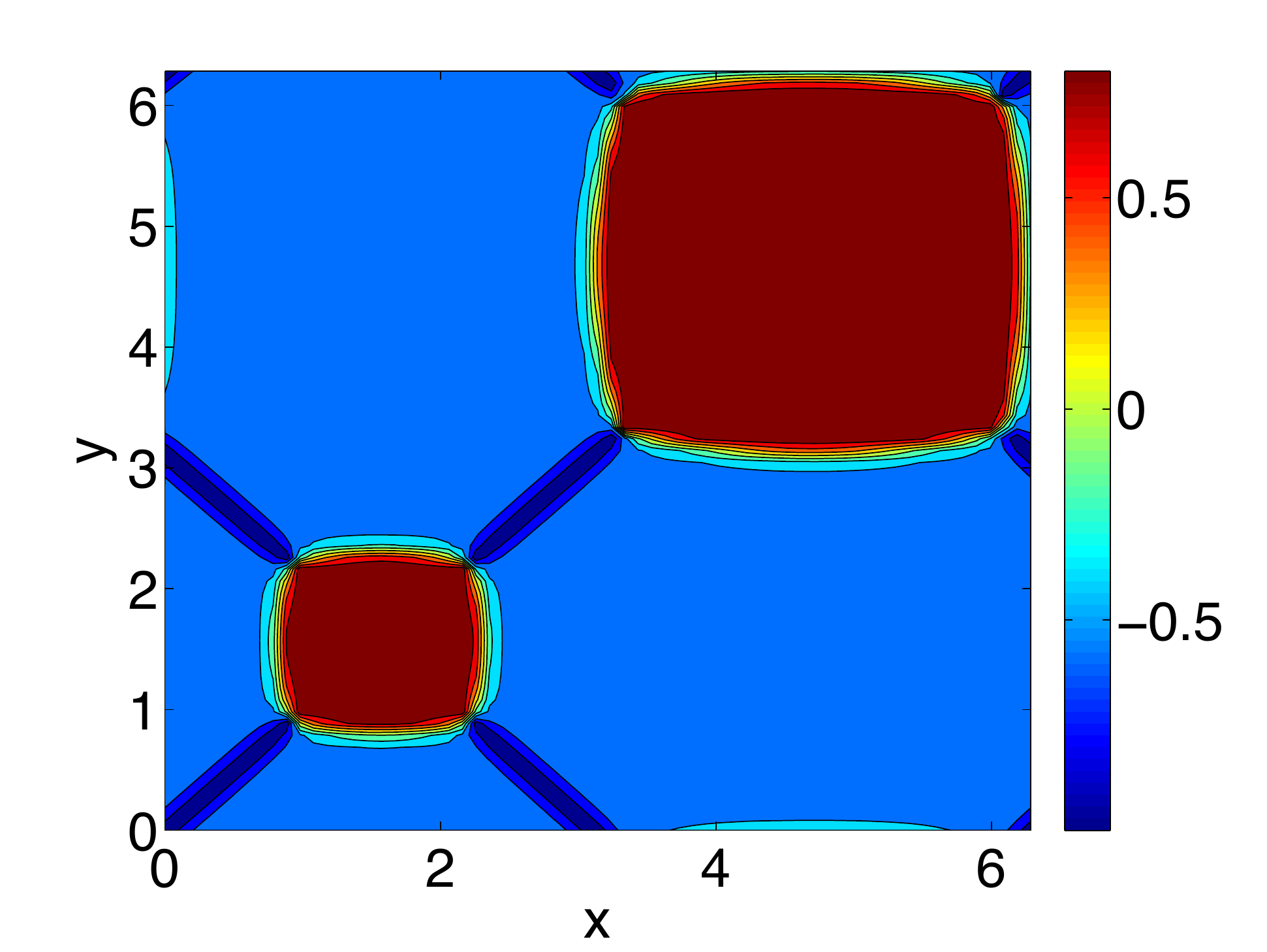}}
		\subfigure[$t=24.286$] {\label{fig:VCH2D_u5}\includegraphics[width=0.3\linewidth]{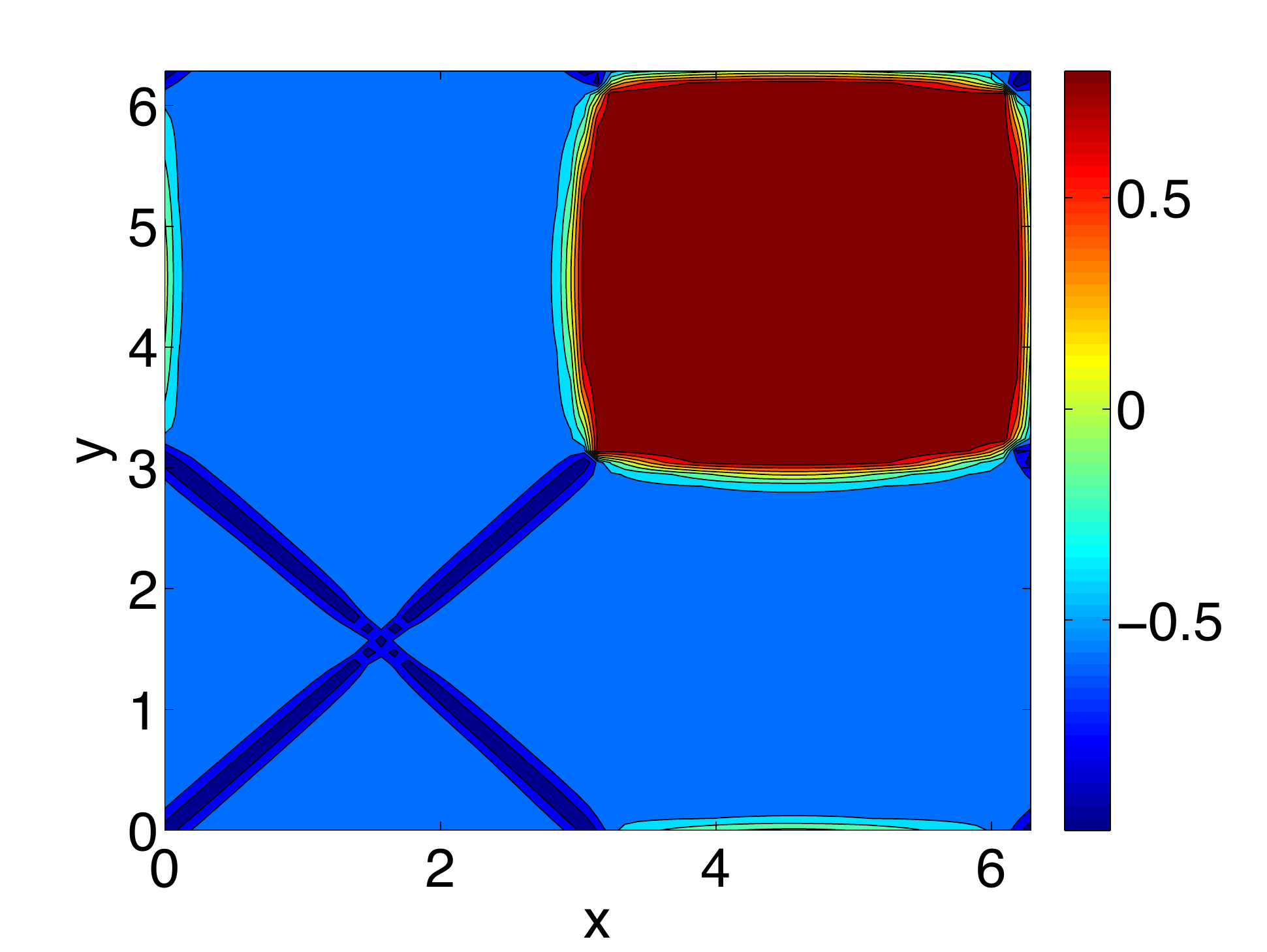}} \\
		\subfigure[$t=25.005$] {\label{fig:VCH2D_u6}\includegraphics[width=0.3\linewidth]{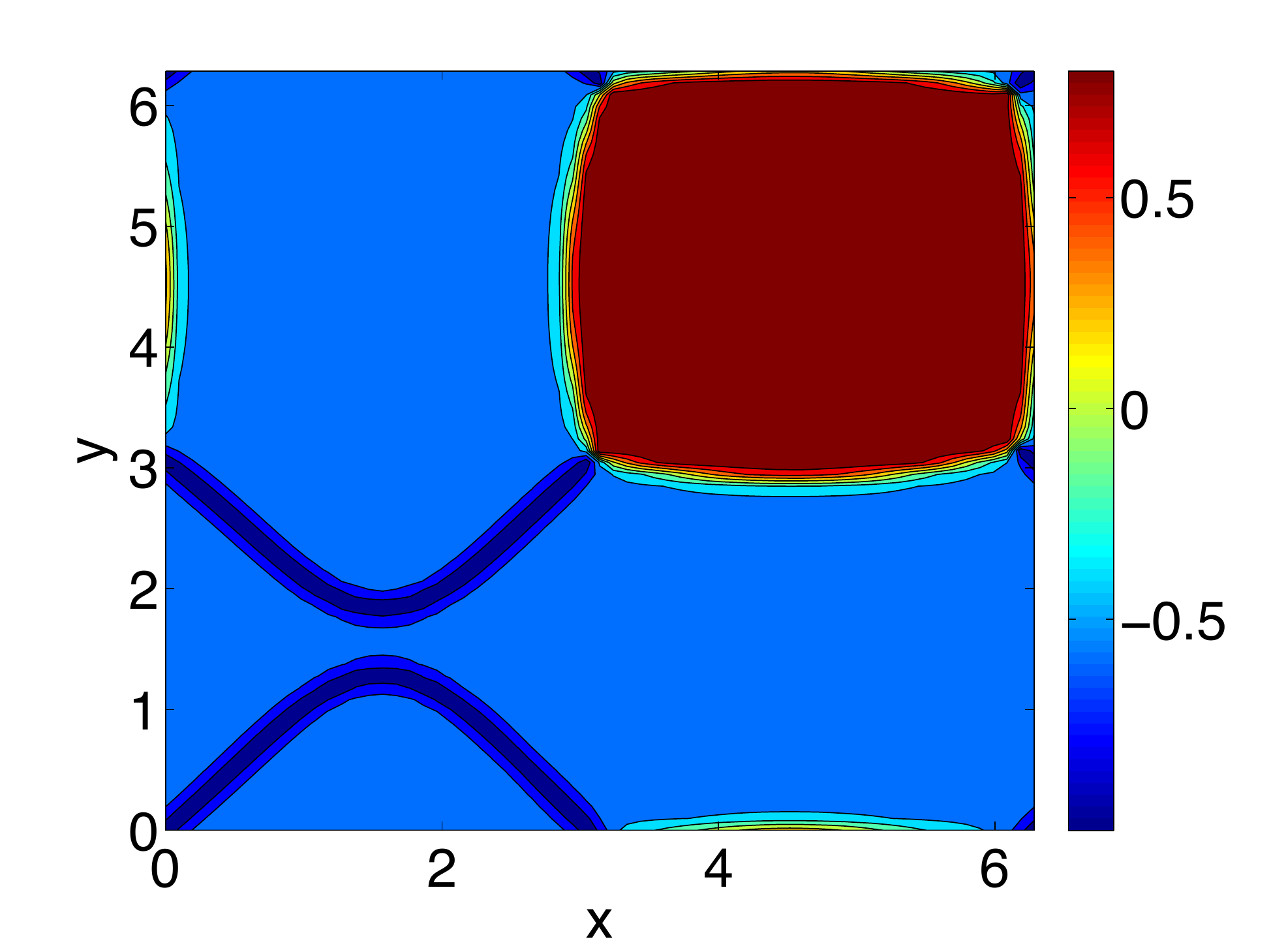}}
		\subfigure[$t=30.021$] {\label{fig:VCH2D_u7}\includegraphics[width=0.3\linewidth]{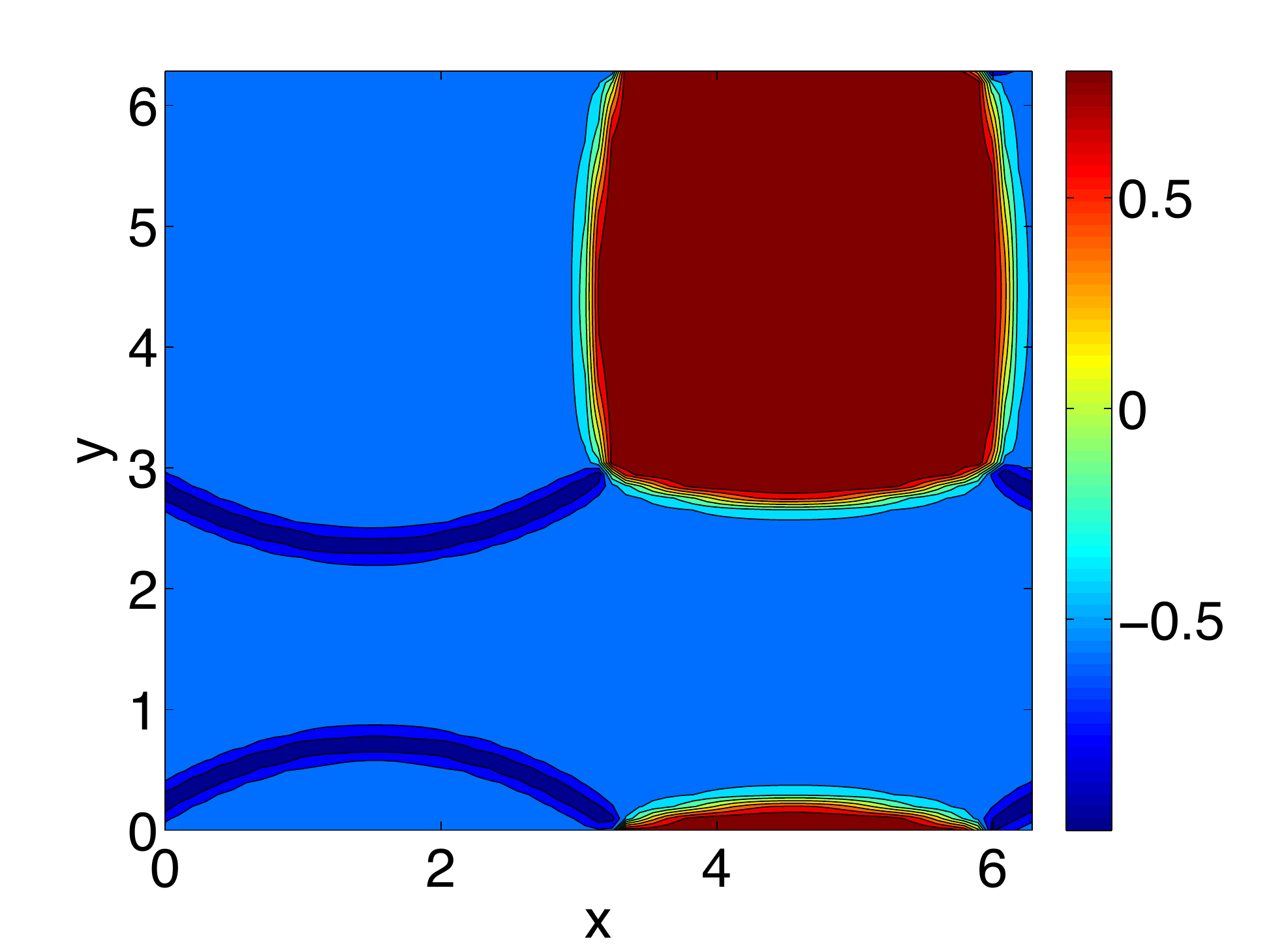}}
		\subfigure[$t=40.015$] {\label{fig:VCH2D_u8}\includegraphics[width=0.3\linewidth]{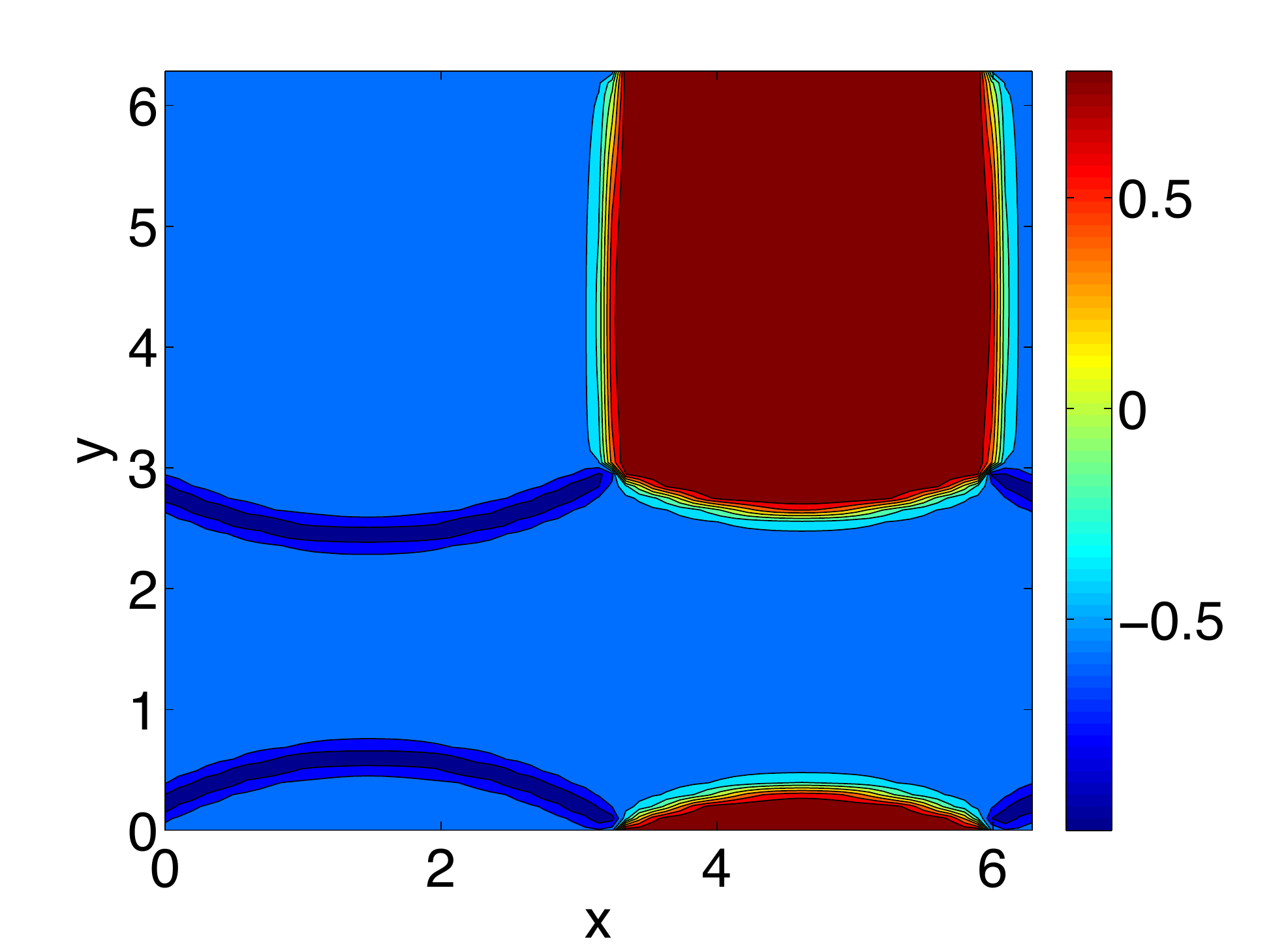}}
		\caption{Temporal evolution of the 2D CH vector solution with the initial \eqref{eqn:CH2D_initial} for $u_1$ and \eqref{eqn:CH2D_initial_v0} for $u_2$. Contours of $\cos(\arg u_1+iu_2)$ are plotted.}
		\label{fig:CH2D_vector}
	\end{center}
\end{figure}
After some initial ripening in Figure \ref{fig:VCH2D_u1}, the interfaces dividing the three states $\bf{u} = \bf{z}_i$ $(i=1,2,3)$ form around $T = 0.5$. In Figure \ref{fig:VCH2D_u2} - \ref{fig:VCH2D_u8}, two of the values have $cos(2\pi/3)$ (light blue in the plots) but separated by dark blue lines. Then the ripening process begins, and occurs over a long time scale, ending around $T = 24.286$ in Figure \ref{fig:VCH2D_u5}. Again, we remark that the volume is preserved over all time steps. For a more involved discussion of this simulation, we refer the interested reader to \cite{Jaylan}. Our goal here is to demonstrate that our MOL$^{\text{T}}$ scheme captures the correct ripening behavior. 

\subsection{2D $6^{\text{th}}$ order model}
We also consider the $6^{\text{th}}$ order problem:
\begin{equation} \label{eqn:FCH0}
u_t = \Delta \left[ (\epsilon^2\Delta -f'(u) + \epsilon^2\eta)(\epsilon^2 \Delta u - f(u) )\right], \quad f(u) = u^3 - u.
\end{equation}
where $\epsilon$ and $\eta$ are given positive constants. This problem is motivated by the functionalized Cahn-Hilliard (FCH) equation \cite{kraitzman2015overview} which models interfacial energy in amphiphilic phase-separated mixtures. We follow a similar approach as before, and introduce the transformed variable $v=u+1$. Substitution into \eqref{eqn:FCH0} results in
\begin{equation} \label{eqn:FCH}
v_t = \Delta \left[ (\epsilon^2\Delta -f'(v) + \epsilon^2\eta)(\epsilon^2 \Delta v - f(v) )\right], \quad f(v) = v^3 - 3v^2 + 2v.
\end{equation}
We again apply the backward Euler (BE) scheme for time discretization of \eqref{eqn:FCH}, so that
\begin{align} \notag
\left( I - \Delta t \epsilon^4\Delta^3 \right) v^{n+1} = v^n -\Delta t \Delta \left(\epsilon^2  \Delta f^{n+1} -\epsilon^2 (f' \Delta v)^{n+1}+(f  f')^{n+1}+\eta\epsilon^4\Delta { v}^{n+1}-\eta\epsilon^2  {f}^{n+1}\right)
\end{align}
Hence, we invert the $6^{\text{th}}$ order operator to solve $v^{n+1}$. One can invert this by completing the cube such that 
\begin{align} \notag
\left( I -  \sqrt[3]{\Delta t \epsilon^4}\Delta \right)^3 v^{n+1}  = v^n -&\Delta t \Delta \left(\epsilon^2  \Delta f^{n+1} -\epsilon^2 (f' \Delta v)^{n+1}+(f  f')^{n+1}+\eta\epsilon^4\Delta { v}^{n+1}-\eta\epsilon^2  {f}^{n+1}\right)\\ \label{eqn:FCH_BE} \vspace{4mm}
&- \left( 3\sqrt[3]{\Delta t \epsilon^4}\Delta - 3(\sqrt[3]{\Delta t \epsilon^4})^2\Delta^2 \right) v^{n+1}.
\end{align}
We can now solve $v^{n+1}$ by applying the triple inversion of our modified Helmholtz operator $\mathcal{L} =  I -  \sqrt[3]{\Delta t \epsilon^4}\Delta$ by the same procedure in Section \ref{sec:CH_2D}.
\begin{figure}[h!]
\begin{center}
\centering
    \subfigure[$u(x,y,0)$]{\label{fig:FCH2D_u0}\includegraphics[width=0.3\linewidth]{CH2D_U0}}
    \subfigure[$u(x,y,2)$]{\label{fig:FCH2D_u1}\includegraphics[width=0.3\linewidth]{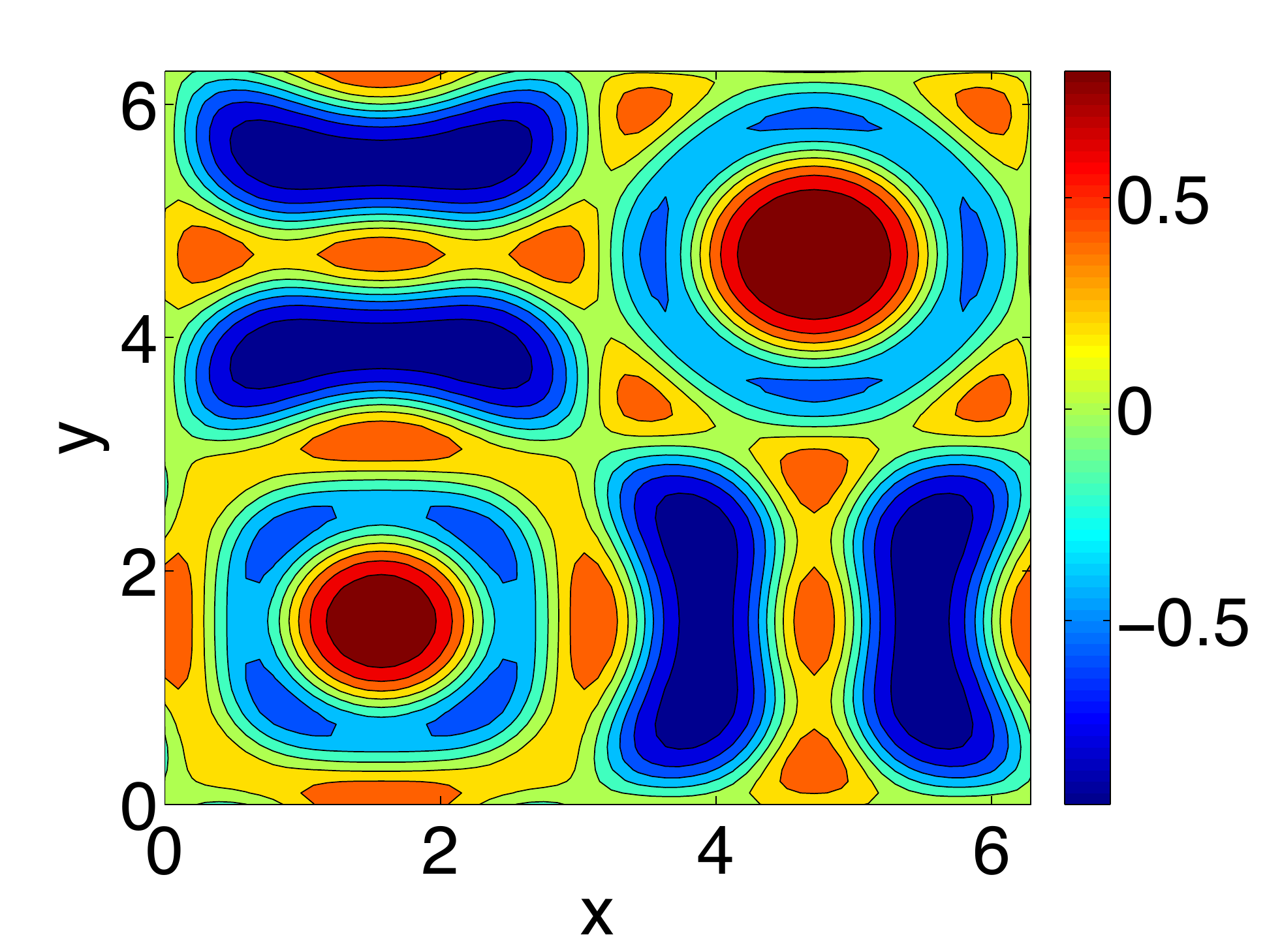}}
    \subfigure[$u(x,y,50)$]{\label{fig:FCH2D_u2}\includegraphics[width=0.3\linewidth]{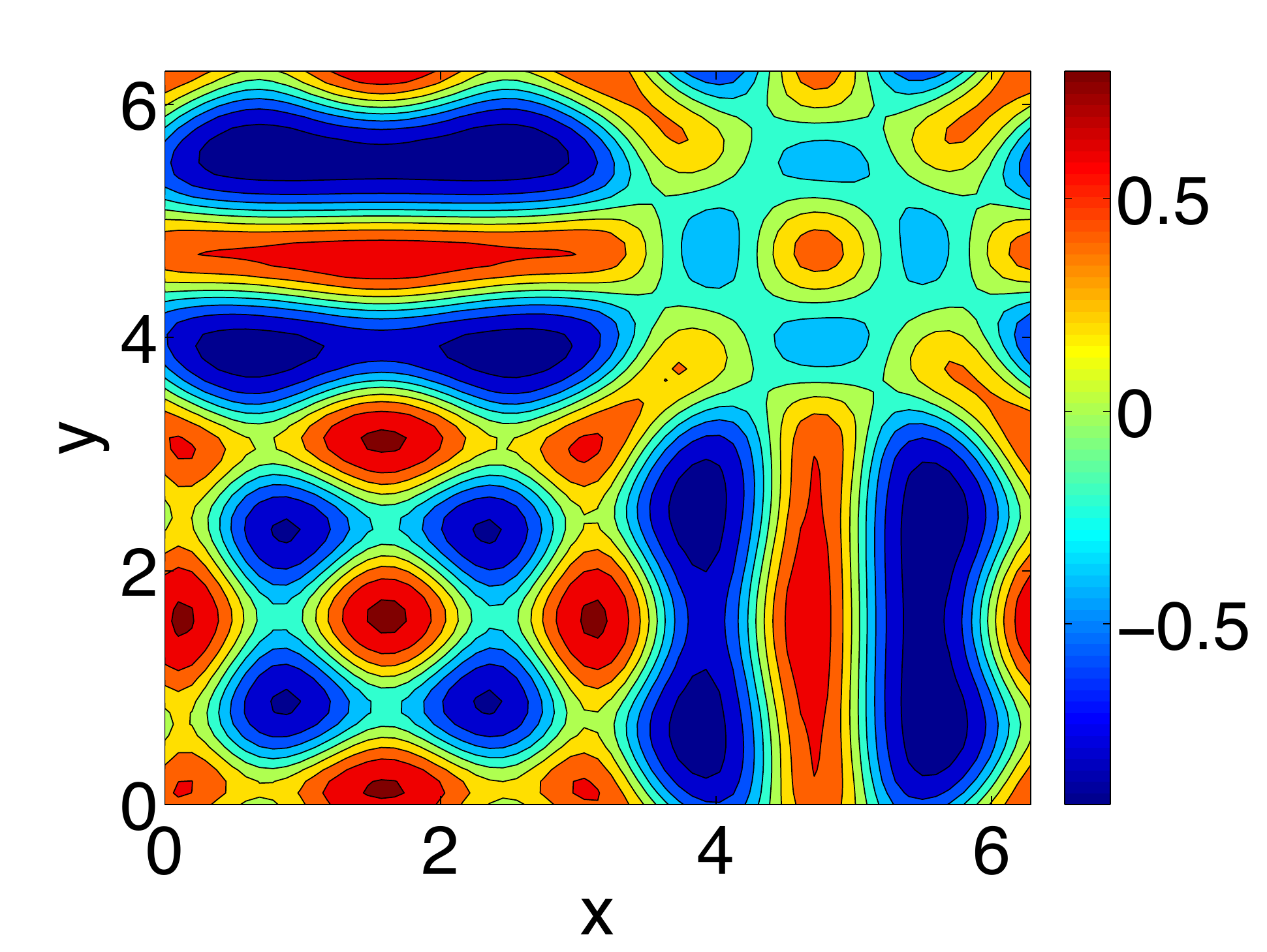}}\\
    \subfigure[$u(x,y,150)$]{\label{fig:FCH2D_u3}\includegraphics[width=0.3\linewidth]{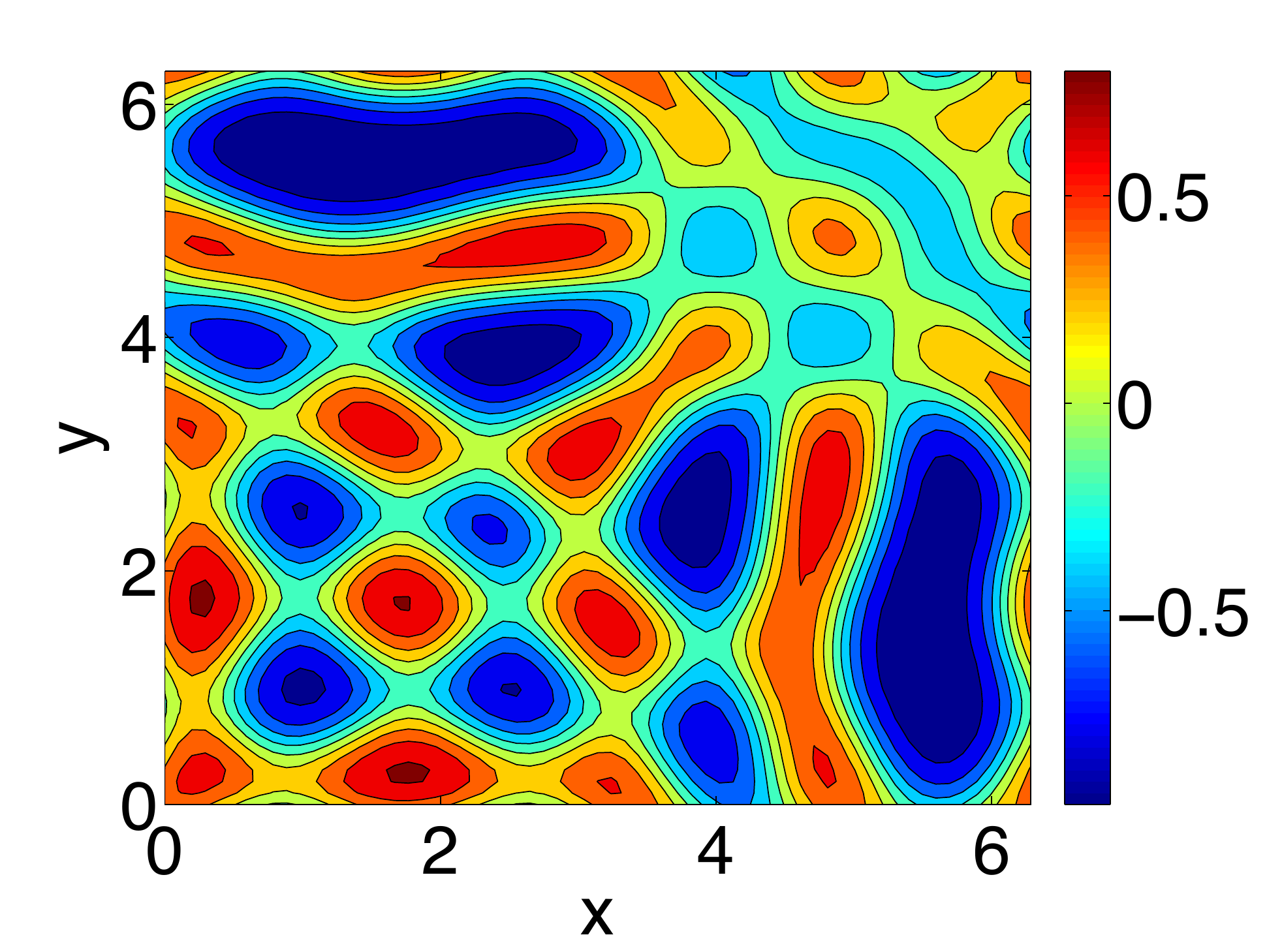}}
    \subfigure[$u(x,y,300)$]{\label{fig:FCH2D_u4}\includegraphics[width=0.3\linewidth]{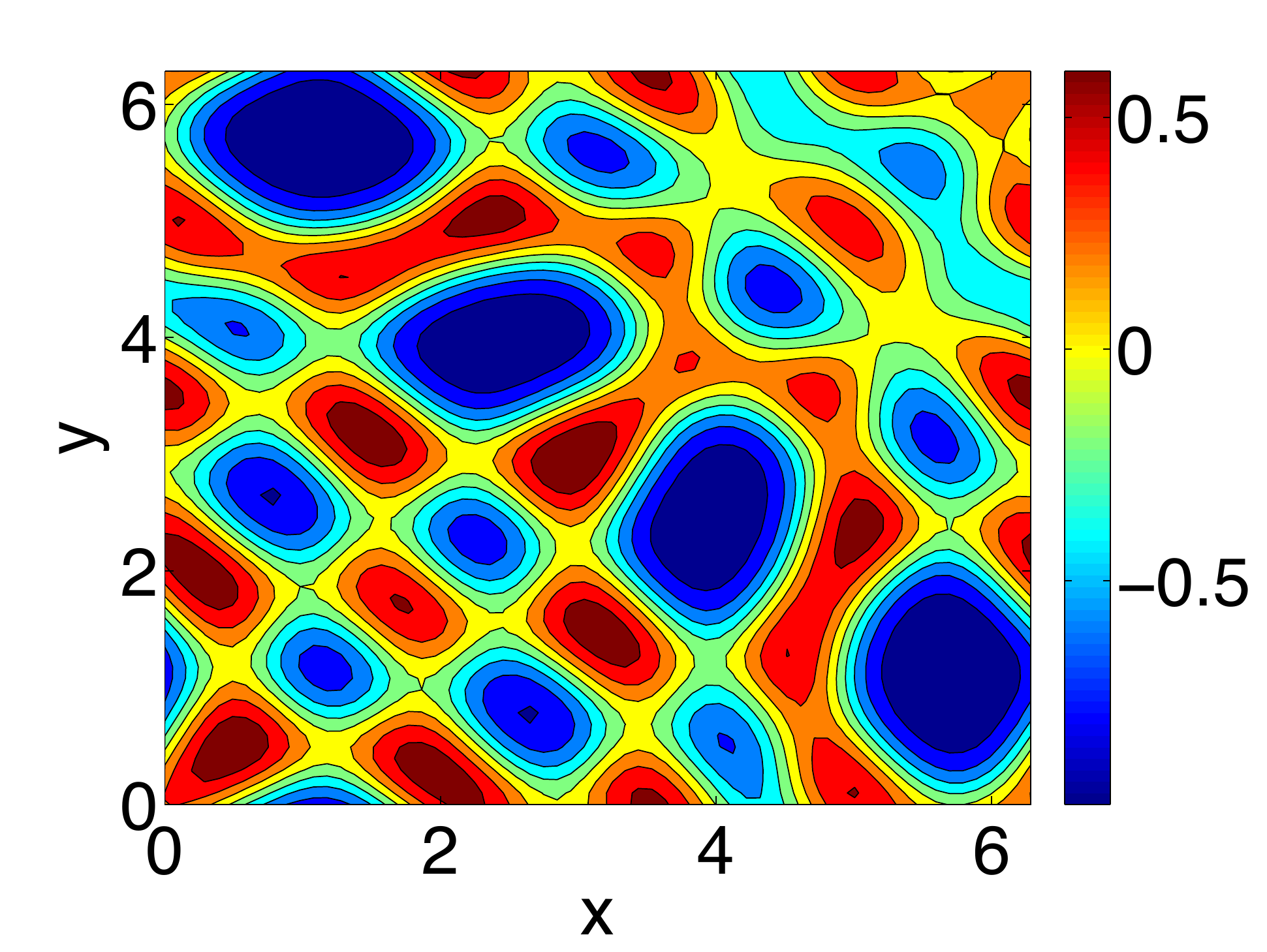}}
    \subfigure[$u(x,y,500)$]{\label{fig:FCH2D_uT}\includegraphics[width=0.3\linewidth]{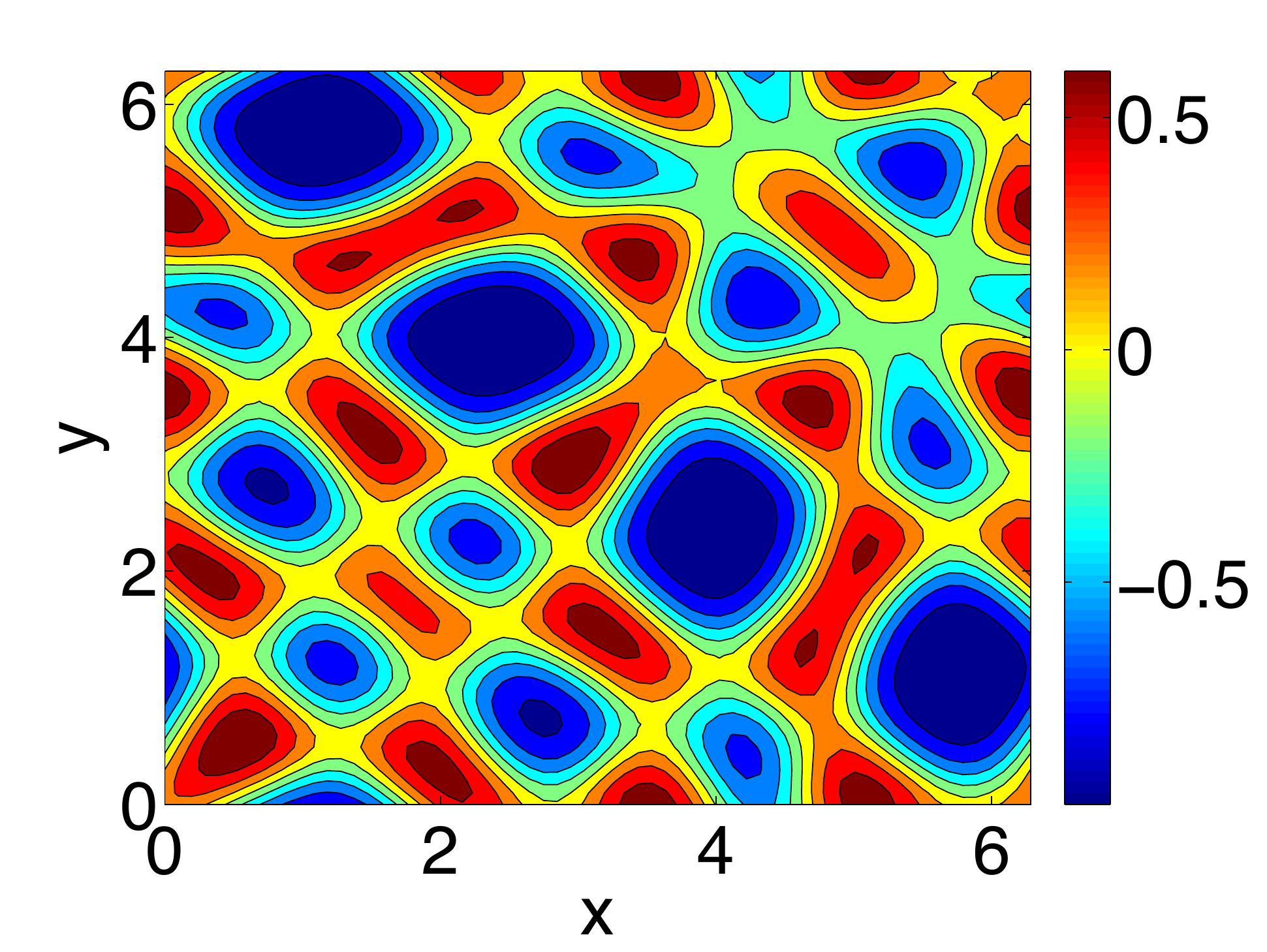}} 
 \caption{Temporal evolution of the 2D FCH equation with initial condition \eqref{eqn:CH2D_initial}.}
    \label{fig:FCH2D}
\end{center}
\end{figure}

For simplicity, we apply the first-order BE scheme \eqref{eqn:FCH_BE} with a fixed time step $\Delta t = 0.1$ to solve the $6^{\text{th}}$ order problem \eqref{eqn:FCH} $(\eta =1)$. By starting with the same initial condition \eqref{eqn:CH2D_initial} $(\epsilon = 0.18)$, and with the following parameters:
\[ \Delta x = \Delta y = \frac{2\pi}{128} \approx 0.05, \quad N_{\max\text{it}} = 200, \quad N_{\text{tol}} = 10^{-6}.\]
The contour plots of our numerical solutions of \eqref{eqn:FCH0} are shown in Figure \ref{fig:FCH2D}. As expected \cite{Jaylan}, we can confirm that the final state becomes a regular array in Figure \ref{fig:FCH2D_uT}. 

\section{Conclusion}
In this paper we have illustrated how the method of lines transpose (MOL$^T$) can be used to solve nonlinear phase field models, particularly the Cahn Hilliard, vector Cahn-Hilliard, and functionalized Cahn-Hilliard equations. We utilize a novel factorization of the semi-discrete equation to ensure gradient stability, permitting large time steps. When combined with time adaptivity, we are able to resolve rapid events such as spinodal evolution, which occur after long meta stable states.

The spatial solver is matrix free, $O(N)$, and logically Cartesian, and the splitting error is directly incorporated into the nonlinear fixed point iterations. We have considered higher order time stepping, such as backward difference formulas, implicit Runge-Kutta, and spectral deferred correction methods. Of all possible configurations, the time adaptive backward Euler-backward difference 2 (BE-BDF2) method is the most efficient in terms of accuracy and time to solution.
\bibliographystyle{abbrv}
\bibliography{MOLT_CH}

\begin{thebibliography}{10}

\bibitem{alexander1977diagonally}
R.~Alexander.
\newblock Diagonally implicit runge-kutta methods for stiff ode's.
\newblock {\em SIAM Journal on Numerical Analysis}, 14(6):1006--1021, 1977.

\bibitem{bronsard1993three}
L.~Bronsard and F.~Reitich.
\newblock On three-phase boundary motion and the singular limit of a
  vector-valued ginzburg-landau equation.
\newblock {\em Archive for Rational Mechanics and Analysis}, 124(4):355--379,
  1993.

\bibitem{cahn1958free}
J.~W. Cahn and J.~E. Hilliard.
\newblock Free energy of a nonuniform system. i. interfacial free energy.
\newblock {\em The Journal of chemical physics}, 28(2):258--267, 1958.

\bibitem{causley2014method}
M.~Causley, A.~Christlieb, B.~Ong, and L.~Van~Groningen.
\newblock Method of lines transpose: An implicit solution to the wave equation.
\newblock {\em Mathematics of Computation}, 83(290):2763--2786, 2014.

\bibitem{causley2015method}
M.~Causley, A.~Christlieb, and E.~Wolf.
\newblock Method of lines transpose: an efficient a-stable solver for wave
  propagation.
\newblock {\em Journal of Scientific Computing}, pages 1--26, 2016.

\bibitem{causley2016method}
M.~F. Causley, H.~Cho, A.~J. Christlieb, and D.~C. Seal.
\newblock Method of lines transpose: High order l-stable $o(n)$ schemes for
  parabolic equations using successive convolution.
\newblock {\em SIAM Journal on Numerical Analysis}, 54(3):1635--1652, 2016.

\bibitem{causley2014higher}
M.~F. Causley and A.~J. Christlieb.
\newblock Higher order a-stable schemes for the wave equation using a
  successive convolution approach.
\newblock {\em SIAM Journal on Numerical Analysis}, 52(1):220--235, 2014.

\bibitem{chen2002phase}
L.-Q. Chen.
\newblock Phase-field models for microstructure evolution.
\newblock {\em Annual review of materials research}, 32(1):113--140, 2002.

\bibitem{Jaylan}
A.~Christlieb, J.~Jones, K.~Promislow, B.~Wetton, and M.~Willoughby.
\newblock High accuracy solutions to energy gradient flows from material
  science models.
\newblock {\em Journal of Computational Physics}, 257:193--215, 2014.

\bibitem{Ong}
A.~Christlieb, C.~Macdonald, B.~Ong, and R.~Spiteri.
\newblock Revisionist integral deferred correction with adaptive step-size
  control.
\newblock {\em Communications in Applied Mathematics and Computational
  Science}, 10(1):1--25, 2015.

\bibitem{Zhengfu}
A.~Christlieb, K.~Promislow, and Z.~Xu.
\newblock On the unconditionally gradient stable scheme for the cahn-hilliard
  equation and its implementation with fourier method.
\newblock {\em Commun. Math. Sci}, 11:345--360, 2013.

\bibitem{douglas1955numerical}
J.~Douglas, Jr.
\newblock On the numerical integration of $\frac{\partial ^2 u}{\partial x^2 }
  + \frac{\partial ^2 u}{\partial y^2 } = \frac{\partial u}{\partial t}$ by
  implicit methods.
\newblock {\em Journal of the society for industrial and applied mathematics},
  3(1):42--65, 1955.

\bibitem{dutt2000spectral}
A.~Dutt, L.~Greengard, and V.~Rokhlin.
\newblock Spectral deferred correction methods for ordinary differential
  equations.
\newblock {\em BIT Numerical Mathematics}, 40(2):241--266, 2000.

\bibitem{eyre1998unconditionally}
D.~J. Eyre.
\newblock An unconditionally stable one-step scheme for gradient systems.
\newblock {\em Unpublished article}, 1998.

\bibitem{fairweather1967new}
G.~Fairweather and A.~Mitchell.
\newblock A new computational procedure for adi methods.
\newblock {\em SIAM Journal on Numerical Analysis}, 4(2):163--170, 1967.

\bibitem{iserles2009first}
A.~Iserles.
\newblock {\em A first course in the numerical analysis of differential
  equations}.
\newblock Number~44. Cambridge university press, 2009.

\bibitem{jia2008krylov}
J.~Jia and J.~Huang.
\newblock Krylov deferred correction accelerated method of lines transpose for
  parabolic problems.
\newblock {\em Journal of Computational Physics}, 227(3):1739--1753, 2008.

\bibitem{kraitzman2015overview}
N.~Kraitzman and K.~Promislow.
\newblock An overview of network bifurcations in the functionalized
  cahn-hilliard free energy.
\newblock In {\em Mathematics of Energy and Climate Change}, pages 191--214.
  Springer, 2015.

\bibitem{kropinski2011fast}
M.~C.~A. Kropinski and B.~D. Quaife.
\newblock Fast integral equation methods for the modified helmholtz equation.
\newblock {\em Journal of Computational Physics}, 230(2):425--434, 2011.

\bibitem{shen2012}
F.~Liu and J.~Shen.
\newblock Stabilized semi-implicit spectral deferred correction methods for
  allen-cahn and cahn-hilliard equations.
\newblock {\em Math. Methods Appl. Sci}, 2013.

\bibitem{rothe1930zweidimensionale}
E.~Rothe.
\newblock Zweidimensionale parabolische randwertaufgaben als grenzfall
  eindimensionaler randwertaufgaben.
\newblock {\em Mathematische Annalen}, 102(1):650--670, 1930.

\bibitem{shen2010}
J.~Shen and X.~Yang.
\newblock Numerical approximations of allen-cahn and cahn-hilliard equations.
\newblock {\em Discrete Contin. Dyn. Syst}, 28(4):1669--1691, 2010.

\bibitem{willoughby2011high}
M.~R. Willoughby.
\newblock {\em High-order time-adaptive numerical methods for the Allen-Cahn
  and Cahn-Hilliard equations}.
\newblock PhD thesis, University of British Columbia, 2011.

\end{thebibliography}

\end{document}